\documentclass[reqno, 12pt, reqno]{amsart}

\usepackage{amsfonts, amsthm, amsmath, amssymb}
\usepackage{hyperref}
\hypersetup{colorlinks=false}

\usepackage[margin=1.5in]{geometry}

\usepackage{helvet}

\RequirePackage{mathrsfs} \let\mathcal\mathscr

\newtheorem{theorem}{Theorem}
\newtheorem{lemma}{Lemma}
\newtheorem{proposition}{Proposition}

\theoremstyle{definition}
\newtheorem*{ack}{Acknowledgement}
\newtheorem{remark}[theorem]{Remark}

\renewcommand{\phi}{\varphi}
\renewcommand{\rho}{\varrho}

\renewcommand{\leq}{\leqslant}

\renewcommand{\geq}{\geqslant}

\renewcommand{\bar}{\overline}

\newcommand{\m}{\mathbf{m}}
\newcommand{\n}{\mathbf{n}}

\renewcommand{\v}{\mathbf{v}}
\renewcommand{\u}{\mathbf{u}}

\newcommand{\w}{\mathbf{w}}

\renewcommand{\r}{\mathbf{r}}

\renewcommand{\hat}{\widehat}

\newcommand{\1}{\mathbf{1}}

\begin{document}

\markboth{Ritabrata Munshi}{Subconvexity for symmetric square $L$-functions}
\title[Subconvexity for symmetric square $L$-functions]{Subconvexity for symmetric square $L$-functions}

\author{Ritabrata Munshi}   
\address{School of Mathematics, Tata Institute of Fundamental Research, 1 Dr. Homi Bhabha Road, Colaba, Mumbai 400005, India.}   
\curraddr{Theoretical Statistics and Mathematics Unit, Indian Statistical Institute, 203 B.T. Road, Kolkata 700108, India.}  
\email{rmunshi@math.tifr.res.in}
\thanks{The author is supported by SwarnaJayanti Fellowship, 2011-12, DST, Government of India.}

\begin{abstract}
Let $f$ be a holomorphic modular form of prime level $p$ and trivial nebentypus. We show that there exists a computable $\delta>0$, such that 
$$
L\left(\tfrac{1}{2},\mathrm{Sym}^2 f\right)\ll p^{\tfrac{1}{2}-\delta},
$$
with the implied constant depending only on $\delta$ and the weight of $f$.
\end{abstract}


\maketitle

\tableofcontents

\section{Introduction}
\label{intro}
Let $f\in S_\kappa(M)$ be a holomorphic newform, of weight $\kappa$, level $M$, which we assume to be square-free, and trivial nebentypus. At the cusp at infinity we have the Fourier expansion 
$$
f(z)=\sum_{n=1}^{\infty}\lambda_f(n)n^{\frac{\kappa-1}{2}}e(nz),
$$ 
with normalized Fourier coefficients $\lambda_f(n)$ (so that $\lambda_f(1)=1$), and $e(z)=e^{2\pi iz}$. It is known from the work of Deligne that $|\lambda_f(n)|\leq d(n)$, where $d(n)$ is the divisor function. One associates with $f$ the Dirichlet series
$$
L(s,f)=\sum_{n=1}^{\infty}\frac{\lambda_f(n)}{n^{s}},
$$
which converges absolutely in the right half plane $\sigma>1$. Moreover since $f$ is a Hecke eigenform, one has an Euler product representation
$$
L(s,f)=\prod_{q\;\text{prime}} L_q(s,f).
$$
For $q\nmid M$ we have 
\begin{align*}
L_q(s,f)=\left(1-\frac{\alpha_f(q)}{q^{s}}\right)^{-1} \left(1-\frac{\beta_f(q)}{q^{s}}\right)^{-1},
\end{align*}
and the local parameters $\alpha_f(q)$ and $\beta_f(q)$ are related to the normalized Fourier coefficients by $\alpha_f(q)+\beta_f(q)=\lambda_f(q)$ and $\alpha_f(q)\beta_f(q)=1$. The symmetric square $L$-function is defined by the degree three Euler product 
\begin{align}
\label{def-sym2}
L(s,\text{Sym}^2 f)=\prod_q \left(1-\frac{\alpha_f^2(q)}{q^{s}}\right)^{-1} \left(1-\frac{\alpha_f(q)\beta(q)}{q^{s}}\right)^{-1} \left(1-\frac{\beta_f^2(q)}{q^{s}}\right)^{-1},
\end{align}
which converges absolutely for $\sigma>1$. In this half-plane we also have an absolutely convergent Dirichlet series expansion
$$
L(s,\text{Sym}^2 f)=\zeta_M(2s)\sum_{n=1}^{\infty}\frac{\lambda_f(n^2)}{n^{s}},
$$
where $\zeta_M(s)$ is the Riemann zeta function with the Euler factors at the primes dividing $M$ missing. It is well-known that this $L$-function extends to an entire function and satisfies the functional equation (see \cite{S})
$$
\Lambda(s,\text{Sym}^2 f)=\Lambda(1-s,\text{Sym}^2 f),
$$
where the completed $L$-function is defined by
$$
\Lambda(s,\text{Sym}^2 f)=M^{s}\gamma(s)L(s,\text{Sym}^2 f)
$$
with
\begin{align*}
\gamma(s)=\pi^{-3s/2}\Gamma\left(\frac{s+1}{2}\right)\Gamma\left(\frac{s+\kappa-1}{2}\right)\Gamma\left(\frac{s+\kappa}{2}\right).
\end{align*} 
So the symmetric square $L$-function has arithmetic conductor $M^2$, when $M$ is square-free.
\\

It is expected that $L(s,\text{Sym}^2 f)$ satisfies the Riemann hypothesis, that all the non-trivial zeros should lie on the central line $\sigma=1/2$. This would imply the Lindel\"{o}f hypothesis that 
\begin{align*}
L(1/2,\text{Sym}^2 f)\ll M^{\varepsilon}
\end{align*}
for any $\varepsilon>0$. An easy consequence of the functional equation and the Phragmen-Lindel\"{o}f principle from complex analysis is the convexity bound
\begin{align*}
L(1/2,\text{Sym}^2 f)\ll M^{1/2+\varepsilon}.
\end{align*}
Recently Heath-Brown has shown that complex analysis in fact yields the improved bound without the $\varepsilon$ in the exponent (see \cite{HB-convex}). A deep result of Soundararajan \cite{So} gives an extra saving of $(\log M)^{1-\varepsilon}$. One should note that in fact, the former result is true for any automorphic $L$-functions, while the latter holds under a weak Ramanujan conjecture. \\ 

Better bounds are known for $L$-functions which are given by Euler products of degree at most two. (Here I will only mention results pertaining to the level aspect subconvexity, for other aspects, e.g. spectral or $t$-aspect, the reader may refer to the citations in \cite{MV} or \cite{Mu0}.) A classical result in this context is that of Burgess \cite{B}, who proved
\begin{align*}
L(1/2,\chi)\ll M^{1/4-1/16+\varepsilon}
\end{align*}
for primitive Dirichlet characters $\chi$ of conductor $M$. Burgess employed an ingenious technique to bound short character sums by higher moments of complete character sums, for which one has strong bounds coming from the Riemann hypothesis for curves on finite fields (Weil's theorem). Similar subconvex bound in the level aspect for $GL(2)$ $L$-functions was first obtained by Duke, Friedlander and Iwaniec \cite{DFI-2} using the amplification technique. For $f$ a newform of level $M$ and trivial nebentypus their result gives the subconvex bound
\begin{align*}
L(1/2,f)\ll M^{1/4-1/192+\varepsilon}.
\end{align*}
Such a subconvex bound can also be obtained for certain degree four $L$-functions which are given by the Rankin-Selberg convolutions of two $GL(2)$ forms. This was first obtained by Kowalski, Michel and Vanderkam \cite{KMV}, who established
\begin{align*}
L(1/2,f\otimes g)\ll M^{1/2-1/80+\varepsilon},
\end{align*}  
for $g$ a fixed holomorphic form or a Maass form, and $f$ a holomorphic newform of level $M$ and trivial nebentypus. All these results, and many more similar results in other aspects, are put in a satisfactory set up and an uniform subconvex bound is obtained in the recent work of Michel and Venkatesh \cite{MV}. \\

A subconvex bound for the symmetric square $L$-function $L(s,\text{Sym}^2 f)$ has so far proved to be elusive. This case is comparable with the Rankin-Selberg convolution $L(s,f\otimes f)$, where both the forms are varying and are in fact same. This is a classic example of a `drop in conductor', which is the precise reason why this case has proved to be so hard to tackle using the amplification technique. Curiously this precise phenomenon (drop in conductor) is the backbone of the present work. This is what we utilize to generate a large class of harmonics to give a spectral expansion of the Kronecker delta symbol (via the Petersson trace formula). The chosen class of harmonics on the other hand conspires with the existing harmonics so that the `conductor' does not go up. Indeed for $f$ a newform in $S_\kappa(p)$ and $g$ a newform in $S_k(p,\psi)$ where $\psi$ is primitive modulo $p$, the arithmetic conductor of the Rankin-Selberg convolution $L(s,f\otimes g)$ is $p^3$ instead of $p^4$ (see Lemma~\ref{func-eqn-li}). This is the key for the choice of harmonics in our version of the circle method.\\

In this paper we will prove the following, long awaited, subconvex bound.\\

\begin{theorem}
\label{mthm}
Let $f$ be a holomorphic Hecke form of prime level $p$ and trivial nebentypus. Then there exists a computable absolute constant $\delta>0$ such that
$$
L\left(\tfrac{1}{2},\mathrm{Sym}^2 f\right)\ll p^{\tfrac{1}{2}-\delta}.
$$
The implied constant depends only on the weight of $f$ and $\delta$.
\end{theorem}

\bigskip

One can show that the implied constant actually depends polynomially on the weight. Also one can produce an explicit value for $\delta$ (e.g. $\delta=1/10000$ should be fine). However from the point of view of application the explicit exponent is not required. So we will not try to make it explicit.
The proof of the theorem builds on the technique elaborated in the series `The circle method and bounds for $L$-functions I-IV', especially \cite{Mu4}, \cite{Mu0} and \cite{Mu5}. The companion paper \cite{Mu6} gives another illustration of the ideas in the simpler context of twists of $GL(3)$ $L$-functions. The difficulty in implementing the usual amplified moment method in the present context has been analysed in detail by Iwaniec and Michel (see \cite{IM}). A key ingredient in their estimation of the second moment is Heath-Brown's large sieve for quadratic characters \cite{HB}. They also point out that their method fails to yield an asymptotic formula for the second moment since the large sieve estimate is not `precise enough'. However it has turned out to be difficult to improve upon this deep estimate of Heath-Brown. This large sieve inequality also plays a vital role in our analysis. But we need other powerful ingredients like the Riemann hypothesis for curves over finite fields (Weil's theorem) and strong estimates for shifted convolution sums with special shifts. One may see reminiscent of Burgess' analysis in these parts. Also Deligne's bound for Fourier coefficients is used freely throughout the paper. \\

Finally let us mention an important application of our theorem. Indeed this is the precise reason why this subconvexity problem has been in focus of intensive research in recent times. The arithmetic quantum unique ergodicity conjecture of Rudnick and Sarnak (see \cite{RS} and \cite{Sar}) has a natural generalization to the level aspect (see \cite{Mich}). Let $f$ be a newform of weight $2$, level $p$ and trivial nebentypus. We define a probability measure on the modular curve $X_0(p)$ by
\begin{align*}
\mu_f(z)=\frac{|f(z)|^2y^2}{\|f\|^2}\;\frac{\mathrm{d}x\mathrm{d}y}{y^2}.
\end{align*}
Let $\pi_p:X_0(p)\rightarrow X_0(1)$ be the natural projection map induced by the inclusion $\Gamma_0(p)\subset \Gamma_0(1)$. The direct image of $\mu_f$ by $\pi_p$ defines a probability measure $\mu_{f,1}$ on $X_0(1)$. Then we have the following:\\
 
 \textbf{QUE Conjecture:} As $p\rightarrow \infty$ we have
 \begin{align*}
 \mu_{f,1}\mathop{\longrightarrow}_W \frac{1}{\text{vol}(X_0(1))}\;\frac{\mathrm{d}x\mathrm{d}y}{y^2}.
 \end{align*}\\
 
It is known, at least in the case of square-free level, that this conjecture follows from the level aspect subconvexity for  
$$L(1/2+it,\text{Sym}^2 f)\;\;\;\text{and}\;\;\;L(1/2,\text{Sym}^2 f\otimes g)$$ 
where $g$ is a fixed cusp form. Our theorem supplies the necessary bound for the symmetric square $L$-function for $t=0$. In fact, 
the method of this paper also works in the more general case of $L(1/2+it,\text{Sym}^2 f\otimes \chi)$ where $f$ is of square-free level $N$ and $\chi$ is a Dirichlet character. In this case the implied constant in the theorem depends polynomially on $t$ and the modulus of the character $\chi$. The recent work of Nelson \cite{N} shows that this is enough to deduce level aspect subconvexity for $L(1/2,\text{Sym}^2 f\otimes g)$ for a fixed cusp form $g$. In the light of this a stronger version of the above conjecture follows.  This will be explained in detail in an upcoming joint paper of the author with Nelson. 
\\

\ack
The author wishes to thank professors Henryk Iwaniec, Phillipe Michel, Paul Nelson, Ravi Rao and Peter Sarnak for their encouragements. A part of this work was written down when the author was visiting MSRI, Berkeley in March 2017 and was supported by Gupta Endowment Fund. The author thanks MSRI, Vinita Gupta and Naren Gupta for their generous support.


\section{The set up}
\label{sec-set-up}
From the approximate functional equation (see Section~3 of \cite{IM}) we know that 
\begin{align*}
L\left(\tfrac{1}{2},\mathrm{Sym}^2 f\right)\ll p^\varepsilon \sup_{p^{1-\theta}\leq N\leq p^{1+\varepsilon}} \frac{|S(N)|}{\sqrt{N}}+p^{1/2-\theta/2+\varepsilon},
\end{align*}
where $S(N)$ are sums of the form
\begin{align*}
S(N)=\mathop{\sum\sum}_{n,j=1}^\infty \lambda_f(n^2)V\left(\frac{nj^2}{N}\right)
\end{align*}
with $V$ a smooth bump function $V^{(i)}\ll_i 1$, with support $[1,3/2]$. Here $\lambda_f(n)$ are the normalized (i.e. $\lambda_f(1)=1$) Fourier coefficients of the Hecke form $f$. We write $S(N)$ as
\begin{align*}
S(N)=\mathop{\sum\sum\sum}_{\substack{m,n,j=1\\m=n^2}}^\infty \lambda_f(m) W\left(\frac{mj^4}{N^2}\right)V\left(\frac{nj^2}{N}\right)
\end{align*}
where $W$ is a suitable smooth bump function - $\text{Supp}(W)\subset [1/2,3]$ with $W(x)=1$ for $x\in [1,9/4]$. Note that we are just separating the `structure' from the Fourier coefficients. We will use the harmonics from the set of cusp forms
$
S_k(pq,\psi),
$ 
with some large weight $k$ (of the size $1/\varepsilon$), to detect the equation $m=n^2$.\\

Let $\psi$ be a primitive odd character modulo $pq$. Let $k$ be a large odd integer and let $H_k(pq,\psi)$ denote an orthogonal Hecke basis for the space $S_k(pq,\psi)$ of cusp forms of weight $k$, level $pq$ and nebentypus $\psi$. 
Consider the spectral sum
\begin{align}
\label{sum-d}
\mathcal{D}=&\sum_{q\in\mathcal{Q}}\;\sideset{}{^\dagger}\sum_{\psi\bmod{pq}}\:\sum_{g\in H_k(pq,\psi)}\omega_g^{-1}\mathop{\sum\sum}_{\substack{m,\ell=1}}^\infty \lambda_f(m)\overline{\lambda_g(m)}\bar{\psi}(\ell) W\left(\frac{m\ell^2}{N^2}\right)\\
\nonumber &\times\mathop{\sum\sum}_{\substack{n,j=1}}^\infty\lambda_g(n^2)\psi^2(j) V\left(\frac{nj^2}{N}\right),
\end{align}
where $\mathcal{Q}$ is a set of primes in the range $[Q,2Q]$ such that $|\mathcal{Q}|\gg Q^{1-\varepsilon}$, and
$$
\omega_g^{-1}= \frac{\Gamma(k-1)}{(4\pi)^{k-1}\|g\|^2}
$$
are the spectral weights. Here the Petersson norm is given by 
\begin{align*}
\|g\|^2=\mathop{\int}_{\Gamma_0(pq)\backslash \mathbb{H}}\;|g(z)|^2y^{k-2}\mathrm{d}x\mathrm{d}y,
\end{align*}
and $\lambda_g(n)$ are the normalized Fourier coefficients at infinity of the form $g$.
Later we will impose the condition that each prime $q\in \mathcal{Q}$ satisfies the congruence condition $q\equiv 1\bmod{4}$, so that the quadratic character modulo $q$ is even. The $\dagger$ on the $\psi$ sum indicates that we are restricting the sum to odd $\psi$ such that $\psi^2$ is primitive. The parameter $Q$ will be chosen optimally at the end.  We also set
\begin{align}
\label{sum-o}
\mathcal{O}=&\sum_{q\in\mathcal{Q}}\;\sideset{}{^\dagger}\sum_{\psi\bmod{pq}}\:\mathop{\sum\sum\sum\sum}_{\substack{m,n,\ell,j=1}}^\infty \lambda_f(m)\psi(\bar{\ell}j^2)W\left(\frac{m\ell^2}{N^2}\right)V\left(\frac{nj^2}{N}\right)\\
\nonumber &\times\mathop{\sum}_{\substack{c=1}}^\infty\frac{S_\psi(m,n^2;cpq)}{cpq} J_{k-1}\left(\frac{4\pi n\sqrt{m}}{cpq}\right).
\end{align}
Here
$$
S_{\psi}(a,b;c)=\sideset{}{^\star}\sum_{\alpha\bmod{c}}\psi(\alpha)e\left(\frac{\alpha a+\bar{\alpha}b}{c}\right)
$$ 
is the Kloosterman sum, and $J_{k-1}$ is the J-Bessel function of order $k-1$. \\

\begin{lemma}
\label{lemma-initial}
Let $f$ be as in the statement of Theorem~\ref{mthm}. Then for $Q>p^{\theta/2}$, with $1>\theta>0$, we have
\begin{align}
\label{afe-bd-2}
L(\tfrac{1}{2},\mathrm{Sym}^2 f)\ll p^\varepsilon\sup_N \:\frac{|\mathcal{D}|+|\mathcal{O}|}{\sqrt{N}pQ^2}+ p^{(1-\theta)/2+\varepsilon},
\end{align}
where the supremum is taken over $N$ in the range $p^{1-\theta}\ll N\ll p^{1+\varepsilon}$, and the sums $\mathcal{D}$, $\mathcal{O}$ are as in \eqref{sum-d} and \eqref{sum-o} respectively.
\end{lemma}

\begin{proof}
We apply the Petersson trace formula to \eqref{sum-d}. The diagonal $m=n^2$ gives
\begin{align}
\label{sum-d-diag}
\Delta=&\sum_{q\in\mathcal{Q}}\;\sideset{}{^\dagger}\sum_{\psi\bmod{pq}}\:\mathop{\sum\sum\sum}_{\substack{n,j,\ell=1}}^\infty \lambda_f(n^2)\psi(j^2\bar\ell) W\left(\frac{n^2\ell^2}{N^2}\right)V\left(\frac{nj^2}{N}\right).
\end{align}
For simplicity suppose $p\equiv 1\bmod{4}$ and we pick $q\equiv 1\bmod{4}$. Then the quadratic characters modulo $p$ or $q$ are even. So writing $\psi=\psi_p\psi_q$ with $\psi_p\bmod{p}$ and $\psi_q\bmod{q}$, we see that $\psi$ odd and $\psi^2$ non-primitive implies that either $\psi_p$ is quadratic and $\psi_q$ runs over all odd characters, or $\psi_q$ is quadratic and $\psi_p$ runs over all odd characters. Consequently
\begin{align}
\label{res-psi-sum}
2\sideset{}{^\dagger}\sum_{\psi\bmod{pq}}\psi(u)=&\phi(pq)\sum_\pm (\pm 1) 1_{u\equiv\pm 1\bmod{pq}}\\
\nonumber &-\phi(q)\sum_\pm (\pm 1) 1_{u\equiv\pm 1\bmod{q}}\sum_{\substack{\psi_p\bmod{p}\\\psi_p^2=1}}\psi_p(u)\\
\nonumber & -\phi(p)\sum_\pm (\pm 1) 1_{u\equiv\pm 1\bmod{p}}\sum_{\substack{\psi_q\bmod{q}\\\psi_q^2=1}}\psi_q(u).
\end{align}\\

So the sum over $\psi$ in \eqref{sum-d-diag} leads us to 
consider the congruence condition $\ell\equiv \pm j^2\bmod{pq}$. From size consideration, as $Q>p^{\theta/2}$ with $\theta>0$ implies $100N<pQ$, it follows that the only term from this congruence that contributes to the sum \eqref{sum-d-diag} is $\ell=j^2$. For the other two congruences we estimate the contribution trivially. Indeed given $(n,j)$ the number of $\ell$ satisfying the congruence $\ell\equiv \pm j^2\bmod{q}$ is $O(1+N/q)$. Hence the contribution of these terms in \eqref{sum-d-diag} is seen to be dominated by
\begin{align*}
\sum_{q\in\mathcal{Q}}\;q\left(1+\frac{N}{q}\right)\mathop{\sum\sum}_{\substack{n,j=1\\nj^2\sim N}}^\infty |\lambda_f(n^2)|\ll p^\varepsilon NQ(Q+N).
\end{align*}
The congruence modulo $p$ is treated similarly, and it yields a similar bound with $p+N$ in place of $Q+N$. 
It follows that
\begin{align*}
\Delta=&\sum_{q\in\mathcal{Q}}\;\frac{\phi(pq)}{2}\:\mathop{\sum\sum}_{\substack{n,j=1\\(j,pq)=1}}^\infty \lambda_f(n^2) V\left(\frac{nj^2}{N}\right)+O\left(p^\varepsilon NQ(Q+p)\right)
\end{align*}
and consequently (after clearing the coprimality $(j,pq)=1$)
\begin{align*}
\mathcal{D}=Q^\star S(N) +2\pi i^{-k} \mathcal{O}+O\left(p^\varepsilon NQ(Q+p)\right),
\end{align*}
where $Q^\star$ denotes the number of $(\psi,q)$ (by choice we will have $Q^\star\gg pQ^{2-\varepsilon}$) and $\mathcal{O}$ is as defined in \eqref{sum-o}. The lemma follows.
\end{proof}

\bigskip

\noindent
\textbf{Notation:}
We will conclude this section by introducing some notations. Let 
$$
\mathcal{S}=\cup_{q\in\mathcal{Q}}\:\cup_{\substack{\psi\bmod{pq}\\\psi\:\text{odd}\\\psi^2\:\text{primitive}}}\:H_k(pq,\psi).
$$
For two complex valued functions $\mathcal{F}$ and $\mathcal{G}$ on the set $\mathcal{S}$, we define
\begin{align}
\label{in-prod}
\left\langle\mathcal{F},\mathcal{G}\right\rangle=\sum_{q\in\mathcal{Q}}\;\sideset{}{^\dagger}\sum_{\psi\bmod{pq}}\:\sum_{g\in H_k(pq,\psi)}\omega_g^{-1}\:\mathcal{F}(g)\mathcal{G}(g).
\end{align} 
Set 
\begin{align}
\label{sum-m}
S_1=\mathop{\sum\sum}_{\substack{m,\ell=1}}^\infty \lambda_f(m)\overline{\lambda_g(m)}\bar{\psi}(\ell) W\left(\frac{m\ell^2}{N^2}\right)
\end{align}
and
\begin{align}
\label{sum-n}
S_2=\mathop{\sum\sum}_{\substack{n,j=1}}^\infty\lambda_g(n^2)\psi^2(j) V\left(\frac{nj^2}{N}\right).
\end{align}
These are viewed as functions on $\mathcal{S}$, and we have
\begin{align*}
\mathcal{D}=\left\langle S_1,S_2 \right\rangle.
\end{align*}\\

 We will use the following convention from \cite{Mu6}. Suppose we have 
$$\mathcal{A}\ll p^\varepsilon \sup_{b\in\mathcal{F}}\:|\mathcal{B}_b| + p^{-A}$$ for all $A>0$, where the implied constant depends on $A$ and $\varepsilon$. Also suppose $|\mathcal{F}|\ll p^\varepsilon$. Then we will write 
\begin{align*}
\mathcal{A}\triangleleft_{\mathcal{F}} \mathcal{B}_b 
\end{align*}
or simply $\mathcal{A} \triangleleft \mathcal{B}$ where there is no scope of confusion. Basically this means that the problem of bounding $\mathcal{A}$ reduces to obtaining an uniform bound for $\mathcal{B}_b$ over the family $\mathcal{F}$. We will say the $\mathcal{B}$ is a `good model' for $\mathcal{A}$, as long as our goal is to get an upper bound for $\mathcal{A}$. \\

\bigskip

\section{The off-diagonal $\mathcal{O}$}

Our next proposition gives a sufficient bound for the off-diagonal contribution $\mathcal{O}$ as defined in \eqref{sum-o}. This sum is negligibly small if we pick $Q\gg p^{1+\varepsilon}$, as in this case the size of the Bessel function is negligibly small due to the choice of the large weight $k$. However while analysing the dual sum $\mathcal{D}$ we will see that we have to pick $Q$ smaller than $p^{1-\varepsilon}$. So we need a non-trivial treatment for $\mathcal{O}$.\\

\begin{proposition}
\label{prop1}
If $Q>p^{1/2+\theta/4}$ then we have
\begin{align*}
\mathcal{O}\ll N^{1/2}p^{3/2-\theta/2+\varepsilon} Q^2.
\end{align*}
\end{proposition}

\begin{proof}
Consider the off-diagonal contribution \eqref{sum-o}. As $k$ is very large, from the size of the Bessel function it follows that it is enough to consider 
$$1\leq c\ll N^2p^\varepsilon/\ell j^2pq\ll p^{1+\varepsilon}/Q.$$ 
So if $Q$ is picked as $p^h$ with $h>1$, then there are no contributing $c$ and hence the off-diagonal contribution is negligibly small. We will now extend the range further. Observe that the terms where $p|c$ make a negligible contribution as $Q>p^h$ for some $h>0$. On the other hand if $q|c$, then writing $c=qc'$ we see that $1\leq c'\ll p^{1+\varepsilon}/Q^2$. But since we are picking $Q=p^h$ with $h>1/2$, these terms also make a negligible contribution. So we only focus on $c$ with $(c,pq)=1$.\\

Next we study the sum over $\psi$, 
\begin{align*}
&\sideset{}{^\dagger}\sum_{\psi}\;\psi(\bar\ell j^2)S_\psi(m,n^2;cpq)=\sideset{}{^\star}\sum_{\alpha\bmod{cpq}}e\left(\frac{\alpha m+\bar{\alpha}n^2}{cpq}\right)\;\sideset{}{^\dagger}\sum_{\psi}\:\psi(\bar \ell j^2\alpha).
\end{align*}
We now apply \eqref{res-psi-sum}. This yields three terms. In the generic term we use the congruence $\bar\ell j^2\alpha\equiv \pm 1\bmod{pq}$ to solve for $\alpha$. Furthermore since $(c,pq)=1$ (note that $(\ell j,pq)=1$ otherwise the vanishes) we see that the contribution of this term to the character sum is given by
$$
S(\overline{pq}m,\overline{pq}n^2;c)\phi(pq)\:\left\{e\left(\frac{\bar{c}(\ell\bar{j}^2 m+\bar\ell j^2n^2)}{pq}\right)-e\left(-\frac{\bar{c}(\ell\bar{j}^2 m+\bar\ell j^2n^2)}{pq}\right)\right\}.
$$
Consider the first term. Applying the reciprocity relation
$$
e\left(\frac{\bar{c}(\ell\bar{j}^2 m+\bar\ell j^2n^2)}{pq}\right)=e\left(-\frac{\overline{pq}(\ell^2m+j^4n^2)}{c\ell j^2}\right)e\left(\frac{(\ell^2m+j^4n^2)}{c\ell j^2pq}\right)
$$
we are led to consider the following sum
\begin{align*}
\sum_{q\in\mathcal{Q}}\frac{\phi(q)}{q}\:\sum_{\substack{\ell=1\\(pq,\ell)=1}}^\infty &\mathop{\sum}_{\substack{c\ll N^2p^\varepsilon/\ell^2pq\\(pq,c)=1}}c^{-1}\:\mathop{\sum\sum}_{\substack{m,n=1}}^\infty  \lambda_f(m)\:S(\overline{pq}m,\overline{pq}n^2;c)\\
&\times e\left(-\frac{\overline{pq}(\ell^2m+j^4n^2)}{c\ell j^2}\right)U\left(\frac{m\ell^2}{N^2},\frac{nj^2}{N},\frac{cpq\ell j^2}{N^2}\right)
\end{align*}
where
\begin{align*}
U(x,y,z)=W\left(x\right)V\left(y\right)e\left(\frac{x+y^2}{z}\right)J_{k-1}\left(\frac{4\pi \sqrt{x}y}{z}\right).
\end{align*}
We apply the Voronoi summation formula on the sum over $m$. We pick $Q>p^{1/2+\varepsilon}$. Then $p(c\ell j^2)^2\ll pN^4p^\varepsilon/p^2q^2\ll N^2p^{-\varepsilon}$, and hence the sum is negligibly small. Indeed after Voronoi we have the Henkel transform
\begin{align*}
\int_0^\infty W(x)e\left(\frac{xN^2}{cpq\ell j^2}\right)J_{k-1}\left(\frac{4\pi N\sqrt{x}n}{cpq\ell}\right)J_{\kappa-1}\left(\frac{4\pi N\sqrt{xm}}{c\ell j^2\sqrt{p}}\right)\mathrm{d}x,
\end{align*}
where $\kappa$ is the weight of the form $f$. Using the decomposition of the Bessel function 
\begin{align*}
J_{\kappa-1}(2\pi x)=\sum_\pm e(\pm x)W_\pm(x)
\end{align*}
where $x^jW_\pm^{(j)}(x)\ll_j 1$, we arrive at the integral
\begin{align*}
\int_0^\infty U(x)e\left(\pm\frac{4\pi N\sqrt{m}x}{c\ell j^2\sqrt{p}}\right)\mathrm{d}x.
\end{align*}
The new weight $U$ is smooth compactly supported in $(0,\infty)$ and satisfies the condition 
\begin{align*}
U^{(h)}(x)\ll_h \left(1+\frac{N^2}{cpq\ell j^2}\right)^h
\end{align*}
for all $h\geq 1$. Hence by integration by parts the Hankel transform is bounded by
\begin{align*}
\left(1+\frac{N^2}{cpq\ell j^2}\right)^h\;\left(\frac{c\ell j^2\sqrt{p}}{N\sqrt{m}}\right)^h\ll_h p^{-\varepsilon h},
\end{align*}
uniformly for all $m\geq 1$.\\

We now consider the non-generic terms. First consider the contribution where we get the congruence $\bar\ell j^2\alpha\equiv 1\bmod{q}$. This solves for $\alpha$ modulo $q$. The character sum splits as
$$
\psi_p(\bar{\ell}) S_{\psi_p}(\overline{q}m,\overline{q}n^2;cp)\:\phi(q)\:e\left(\frac{\overline{cp}(\ell\bar{j}^2 m+\bar\ell j^2n^2)}{q}\right),
$$
where we need to take $\psi_p$ trivial or quadratic. We can again apply the reciprocity relation to the last term. With this we arrive at the following sum over $m$,
\begin{align*}
\mathop{\sum}_{\substack{m=1}}^\infty  \lambda_f(m)S_{\psi_p}(\overline{q}m,\overline{q}n^2;cp)\: e\left(-\frac{\overline{q}\ell m}{cp j^2}\right)U\left(\frac{m\ell^2}{N^2},\frac{nj^2}{N},\frac{cpq\ell j^2}{N^2}\right)
\end{align*}
where $U$ is as before. We again apply the Voronoi summation formula, but now with modulus $cpj^2$. The analysis of the Hankel transform yields that the dual length is given by 
\begin{align*}
\frac{\ell^2}{N^2}\:\left((Cpj^2)^2+\frac{N^4}{Q^2\ell^2}\right)\ll p^\varepsilon\frac{N^2}{Q^2}.
\end{align*}
Hence from Voronoi we save $Q/\ell$. In the sum over $\psi$ we have saved $pQ^{1/2}$. Hence in total we have saved $pQ^{3/2}/\ell$, and consequently the contribution of this term to \eqref{sum-o} is bounded by
\begin{align*}
p^\varepsilon\mathop{\sum\sum}_{\ell,j<p^2}\:\frac{N}{j^2}\:\frac{N^2}{\ell^2}\:\frac{N}{\ell^{1/2}jpQ}\:pQ^2\:\frac{\ell}{pQ^{3/2}}\ll \frac{N^{1/2}p^{5/2+\varepsilon}}{Q^{1/2}}.
\end{align*}
Note that
this is satisfactory for our purpose if $Q\gg p^{2/5+\delta}$ for some $\delta>0$. \\

Finally we consider the contribution where we get the congruence $\bar\ell j^2\alpha\equiv 1\bmod{p}$. This solves for $\alpha$ modulo $p$. The character sum now splits as
$$
\psi_q(\bar{\ell}) S_{\psi_q}(\overline{p}m,\overline{p}n^2;cq)\:\phi(p)\:e\left(\frac{\overline{cq}(\ell\bar{j}^2 m+\bar\ell j^2n^2)}{p}\right),
$$
where we need to take $\psi_q$ trivial or quadratic. Applying the reciprocity relation we arrive at the following sum over $m$,
\begin{align*}
\mathop{\sum\sum}_{\substack{m,n=1}}^\infty  \lambda_f(m)S_{\psi_q}(\overline{p}m,\overline{p}n^2;cq)\: e\left(-\frac{\overline{p}\ell m}{cq j^2}\right)U\left(\frac{m\ell^2}{N^2},\frac{nj^2}{N},\frac{cpq\ell j^2}{N^2}\right).
\end{align*}
We now apply the Voronoi summation formula with modulus $cqj^2$. The analysis of the Hankel transform yields that the dual length is given by 
\begin{align*}
\frac{\ell^2}{N^2}\:\left(p(CQj^2)^2+\frac{N^4}{p\ell^2}\right)\ll p^\varepsilon\frac{N^2}{p}.
\end{align*}
Hence from Voronoi we save $p^{1/2}/\ell$. In the sum over $\psi$ we have saved $p^{1/2}Q$, and so in this case in total we have saved $pQ/\ell$. Consequently the contribution of this term to \eqref{sum-o} is bounded by
\begin{align*}
p^\varepsilon\mathop{\sum\sum}_{\ell,j<p^2}\:\frac{N}{j^2}\:\frac{N^2}{\ell^2}\:\frac{N}{\ell^{1/2}jpQ}\:pQ^2\:\frac{\ell}{pQ}\ll N^{1/2}p^{5/2+\varepsilon}.
\end{align*}
This is worse than the bound obtained above, but is satisfactory for our purpose if $Q\gg p^{1/2+\delta}$ for some $\delta>0$. The lemma follows.
\end{proof}

\bigskip

Substituting this bound for the off-diagonal in Lemma~\ref{lemma-initial} we conclude the following.\\

\begin{lemma}
\label{lemma-initial-second}
Let $f$ be as in the statement of Theorem~\ref{mthm}. Then for $Q>p^{1/2+\theta/4}$, with $0<\theta<1$, we have
\begin{align}
\label{afe-bd-2-after-off-diagonal}
L(\tfrac{1}{2},\mathrm{Sym}^2 f)\ll p^\varepsilon\sup_N \:\frac{|\mathcal{D}|}{\sqrt{N}pQ^2}+ p^{(1-\theta)/2+\varepsilon}.
\end{align}
\end{lemma}
 
\bigskip

Our job now is to prove a bound for $\mathcal{D}$ of the form 
\begin{align*}
\mathcal{D}\ll N^{1/2}p^{3/2-\theta/2}Q^2
\end{align*}
for some $\theta>0$, where $N$ ranges in $p^{1-\theta}<N<p^{1+\varepsilon}$, and $Q$ is taken to be $Q>p^{1/2+\theta/4}$.\\

\section{A sketch of the proof}

In this section we present a brief outline of the proof. We will use the following colloquial language. Suppose a process $P$ (e.g. Poisson summation) transforms a sum $S$ to another sum $S'$. Suppose a trivial estimation yields $S\ll B$ and that $S'\ll B'$. Then we say that the process gives a `saving' of $B/B'$. This language will be loosely used throughout this section. For the sake of simplicity we take $N=p$. Consider the sum given in \eqref{sum-d}, and recall that we have expressed it as $\mathcal{D}=\left\langle S_1,S_2 \right\rangle$ where the sums $S_i$ are as given in \eqref{sum-m} and \eqref{sum-n}. Our job is to save $p^2$ plus a little more. In the next section we will use the functional equations of the Rankin-Selberg $L$-function $L(s,f\otimes g)$ to derive a summation formula for the sum $S_1$. Roughly speaking the sum gets transformed into 
\begin{align*}
S_1^\star = \frac{\sqrt{p}}{Q}\;\varepsilon_{\bar{\psi}}^2\;\mathop{\sum\sum}_{\substack{m,\ell\\m\ell^2\sim pQ^2}}\;\lambda_f(m)\lambda_g(pq^2m)\;\psi(\ell),
\end{align*}
where $\varepsilon_\eta$ denotes the sign of the Gauss sum $g_\eta$ associated with the character $\eta$. Secondly using the functional equation of the symmetric square $L$-function $L(s,\text{Sym}^2\:g)$ we derive a summation formula for the sum $S_2$. More precisely the sum gets transformed into
\begin{align*}
S_2^\star(N_\star) = \frac{N_\star}{pQ^3}\;\varepsilon_{\psi}^2\;\mathop{\sum\sum}_{\substack{n,j\\nj^2\sim N_\star}}\;\bar\lambda_g(p^2q^2n^2)\;\bar\psi(j^2),
\end{align*}
where $N_\star\ll p^{1+\varepsilon}Q^2$. Unlike $S_1$, in the case of $S_2$ we have to consider smaller values $N_\star$ as one of the gamma function appearing in the functional equation of the symmetric square $L$-function has a pole at $s=-1$. However this does not turn out to be an issue. So in this sketch we consider the worst case scenario where $N_\star=pQ^2$. For the sake of simplicity in this outline we also drop the sums over $\ell$ and $j$. (However the reader will notice that in the proof of Proposition~\ref{prop-for-o1} in Section~\ref{sec-sum-o1}, the integer $j$ contributes to the conductor and hence one needs to control its size.)\\

With this $\mathcal{D}$ gets transformed into 
\begin{align*}
&\sum_{q\in\mathcal{Q}}\;\sideset{}{^\dagger}\sum_{\psi\bmod{pq}}\:\sum_{g\in H_k(pq,\psi)}\omega_g^{-1}\mathop{\sum}_{\substack{m\sim pQ^2}} \lambda_f(m)\lambda_g(m)\;\mathop{\sum}_{\substack{n\sim pQ^2}}\bar\lambda_g(n^2p).
\end{align*}
In the process we have saved $(p^{1/2}/Q)\times (1/Q)=p^{1/2}/Q^2$. (Actually we are losing in both the applications of the summation formulae.) It now remains to save $p^{3/2}Q^2$ in the above sum. Our next step is an application of the Petersson trace formula. There is no diagonal contribution as the equation $m=n^2p$ is ruled out due to size restrictions. (For smaller values of $N_\star$ we do have a diagonal contribution, but it is easily shown to be small - see Lemma~\ref{diagonal-small}.) The off-diagonal contribution is given by
\begin{align*}
&\sum_{q\in\mathcal{Q}}\;\sideset{}{^\dagger}\sum_{\psi\bmod{pq}}\:\mathop{\sum}_{\substack{m\sim pQ^2}} \lambda_f(m)\;\mathop{\sum}_{\substack{n\sim pQ^2}}\; \sum_{c\sim C}\frac{S_\psi(n^2p,m;cpq)}{cpq}J_{k-1}\left(\frac{4\pi \sqrt{m}n}{c\sqrt{p}q}\right).
\end{align*}
Here $C$ ranges upto $p^{1+\varepsilon}Q^2$, as for larger values the Bessel function is negligibly small as the weight $k$ is chosen to be large like $1/\varepsilon$. 
At the transition range the Bessel function does not oscillate. But for smaller values we have analytic oscillation coming from the Bessel function. This complicates the situation, and we need to focus on all values of $C$.\\

Extracting the oscillation of the Bessel and taking into account its size, our job reduces to saving $(p^{3/2}Q^2)\times (C/pQ^{3/2})=C(pQ)^{1/2}$ in the sum 
\begin{align*}
&\sum_{q\in\mathcal{Q}}\;\sideset{}{^\dagger}\sum_{\psi\bmod{pq}}\:\mathop{\sum}_{\substack{m\sim pQ^2}} \lambda_f(m)\;\mathop{\sum}_{\substack{n\sim pQ^2}}\; \sum_{c\sim C}S_\psi(n^2p,m;cpq)e\left(\frac{2\sqrt{m}n}{c\sqrt{p}q}\right).
\end{align*}
The sum over $\psi$ will give a saving of $(pQ)^{1/2}$, so in the remaining sums we need to save $C$ plus a little extra. Next we apply the Poisson summation formula on the sum over $n$ with modulus $cq$. With this we arrive at the expression
\begin{align*}
&\sum_{q\in\mathcal{Q}}\;\sideset{}{^\dagger}\sum_{\psi\bmod{pq}}\:\mathop{\sum}_{\substack{m\sim pQ^2}} \lambda_f(m)\;\mathop{\sum}_{\substack{|n|\ll Q}}\; \sum_{c\sim C}\;\mathcal{C}_\psi\;\mathcal{I}
\end{align*}
where the character sum is given by \eqref{char-sum-1} and the integral is given by \eqref{int}. It follows by analysing the integral that we only need to consider $|n|\ll p^\varepsilon Q$. We also see that the integral is negligibly small if $|4m-pn^2|\gg p^\varepsilon C$. Consequently at the transition range $C\sim pQ^2$ we save $(pQ)^{1/2}$ in this process. In general we are saving $(pQ)^{1/2}\times (pQ^2/C)^{1/2}=pQ^{3/2}/C^{1/2}$. So we need to save $C^{3/2}/pQ^{3/2}$ (or $(CQ)^{1/2}$ if we sacrifice the information on the restriction on $d$) in the last sum. It follows that if $C\ll p^{2/3-\delta}Q$ then we have saved enough. (In fact, as we note at the end of Section~\ref{sec-dual-sum}, we can get a better range. This is however not utilized later, as our analysis in hindsight, is robust enough to tackle all sizes of $c$.) \\

We next apply a summation formula on the sum over the modulus $c$. Evaluating the character sum in Section~\ref{char-sum-eva}, we realize that the sum over $c$ is arithmetic in nature. In fact, the free part of the sum runs only over square-free integers. So it is not possible to completely dualize this sum. Also since we still need to save a lot, simply dealing the square-free  condition using Mobius and throwing away the large divisors does not work. We take recourse to $L$-functions. This is the main content of Section~\ref{sec-sum-modulus}. In Lemma~\ref{lem-c-sum-l-func} we are able to substitute the sum over $c$ by the central value of an $L$-function $L(1/2,\bar{\psi}\otimes \chi_{dp})$ where $\chi_r$ stands for the quadratic character modulo $r$. The conductor of the $L$-function is of size $Cpq$. Ideally if the sum over $c$ had been over all integers, a summation formula would have yielded a dual sum of length $pQ$. However this is not actually the case. We will now use the approximate functional equation to replace the central value by two finite Dirichlet polynomials (see Lemma~\ref{lem-dirich-poly}). We put a smaller length on the first term and a longer length on the second (dual) term, indeed we take $C^\star\ll C^{1/2}p^{-\delta}$ and $C^\dagger\ll p^{1+\delta}QC^{1/2}$ for some $\delta>0$. In the first sum we have saved $(C/C^\star)^{1/2}$ and in the second sum we have saved $(C/C^\dagger)^{1/2}$.\\

In Section~\ref{sec-sum-o1} we analyse the  first sum, the semi-dual $\mathcal{O}_1(C,C^\star)$ (see \eqref{od-sum-d-smooth-2-after-c-sum-1}) which is roughly of the form
\begin{align*}
&\sum_{q\in\mathcal{Q}}\;\sideset{}{^\dagger}\sum_{\psi\bmod{pq}}\;\mathop{\sum}_{\substack{m\sim pQ^2}} \lambda_f(m)\;\sum_{n\ll Q}\:\psi^\star(d)g_{\bar{\psi}^\star}\;\sum_{c\sim C^\star}\; \bar{\psi(c)}\left(\frac{dp}{c}\right)\;\mathcal{J}.
\end{align*}
Here $\psi^\star=\psi\otimes \chi_q$. The integral $\mathcal{J}$ still retains the information that $d=4m-pn^2\sim D\ll C$. The main output of this section is Proposition~\ref{prop-for-o1} where a satisfactory bound is obtained for the semi-dual sum. As the first step we execute the sum over $\psi$ to arrive at
\begin{align*}
&\sum_{q\in\mathcal{Q}}\;\mathop{\sum}_{\substack{m\sim pQ^2}} \lambda_f(m)\;\sum_{n\ll Q}\;\sum_{c\sim C^\star}\; \left(\frac{dpq}{c}\right)\;e\left(\frac{\bar{c}d}{pq}\right)\:\mathcal{J}.
\end{align*}
When $C\asymp pQ^2$, at the transition range, there is no saving in the integral $\mathcal{J}$, as it is not oscillating. In this case,
an application of the reciprocity relation transforms the above sum to a sum of the form
\begin{align*}
&\sum_{q\in\mathcal{Q}}\;\mathop{\sum}_{\substack{m\sim pQ^2}} \lambda_f(m)\;\sum_{n\ll Q}\;\sum_{c\sim C^\star}\; \left(\frac{dpq}{c}\right)\;e\left(-\frac{\bar{pq}d}{c}\right).
\end{align*}
Our aim is to save $(C^\star Q)^{1/2}$. We replace the quadratic character by additive characters using Gauss sums, and then apply the Voronoi summation formula. This gives a saving of the size $p^{1/2}Q^2/C^\star$. However this is not enough. Next we get rid of the Fourier coefficients by taking absolute values
\begin{align*}
&\mathop{\sum}_{\substack{m\sim C^{\star 2}/Q^2}}\;\sum_{n\ll Q}\;\sum_{c\sim C^\star}\;\left|\sum_{q\in\mathcal{Q}} \:\text{char sum}\:\right|,
\end{align*}
where inside the absolute value sign we have a character sum modulo $c$ (which results from our shift to additive characters from multiplicative characters, see \eqref{back-to-future}). Then we apply the Cauchy inequality followed by the Poisson summation on the sum over $(m,n)$. This yields an extra saving of $\min\{Q^{1/2},(C^\star/Q)^{1/2}\}$. So in total we have saved 
$$\min\left\{\frac{p^{1/2}Q^{5/2}}{C^\star}, \frac{p^{1/2}Q^{3/2}}{C^{\star 1/2}}\right\},$$ 
which is sufficient for our purpose - it is larger than $p^\delta\:CC^{\star 1/2}/pQ^{3/2}$ - as we are taking $C^\star<C^{1/2}p^{-\delta}\ll p^{1/2-\delta}Q$ and $Q>p^{1/2}$. \\

In case $C$ is not large enough the above procedure does not work as the integral $\mathcal{J}$ is highly oscillating which increases the conductor of the $m$ sum, and hence Voronoi summation is not that effective. However in this case we have an easier treatment. We use the decomposition 
\begin{align}
\label{sketch-exp}
\mathcal{O}_1(C,C^\star)\ll \sum_{q\in\mathcal{Q}}\;\sum_{d\sim D}\left|\sum_{n\ll Q}\;\lambda_f(d+pn^2)\;\mathcal{J}\right|\;\left|\sum_{c\sim C^\star}\;\left(\frac{dpq}{c}\right)\;e\left(\frac{\bar{c}d}{pq}\right)\right|.
\end{align}
We have enough harmonics inside the absolute value as we need to save $CC^{\star 1/2}/pQ^{3/2}$ which is less than $(C^\star Q)^{1/2}$ as $C\ll p^{1-\delta}Q^2$. Next we apply Cauchy inequality to bound \eqref{sketch-exp} by
\begin{align*}
\sum_{q\in\mathcal{Q}}\mathcal{Z}^{1/2}\;\Omega^{1/2}
\end{align*}
where
\begin{align*}
\mathcal{Z}=\sum_{d\sim D}\;\left|\sum_{c\sim C^\star}\;\left(\frac{dpq}{c}\right)\;e\left(\frac{\bar{c}d}{pq}\right)\right|^2,
\end{align*}
and
\begin{align}
\label{omega-sketch}
\Omega=\sum_{d\sim D}\left|\sum_{n\ll Q}\;\lambda_f(d+pn^2)\;\mathcal{J}\right|^2.
\end{align}
We open absolute square in $\mathcal{Z}$ and apply Poisson summation on the sum over $d$. This process saves $\min\{C^\star,C/C^\star\}$ in $\mathcal{Z}$. This is not enough for our purpose. To get a satisfactory bound for all possible parameter values $C$, $D$ and $C^\star$ (in the chosen range), we need to get a saving in the sum $\Omega$. This particular sum reappears in our treatment of $\mathcal{O}_2$ as well.  Proposition~\ref{shifted-prop} and Remark~\ref{shifted-p>Q} of Section~\ref{sec-shifted} gives a saving of size 
$$\min\left\{\frac{C^{3/2}}{p^{3/2}Q^2},\frac{C}{Q^2}\right\}$$ 
(which is non-trivial for sufficiently large $C$) in $\Omega$. Note that if we assume that $Q<p$ then we can drop the second term. But we want to point out why the size restriction $Q<p$ comes naturally in our treatment of the dual sum. Indeed suppose we have saved $C^\star$ in $\mathcal{Z}$ and $C/Q^2$ in $\Omega$, then this is enough for our purpose if $(CC^\star)^{1/2}/Q>p^\delta\:CC^{\star 1/2}/pQ^{3/2}$, i.e. $C<p^{2-2\delta}Q$. The last inequality holds in the range of $C$ under study if $p^{1-\delta}Q^2<p^{2-2\delta}Q$, i.e. we need $Q<p^{1-\delta}$ for some $\delta>0$. From now on let us assume that we have the restriction $p^{1/2+\delta}<Q<p^{1-\delta}$ for some $\delta>0$. So in $\Omega$ we save $C^{3/2}/p^{3/2}Q^2$. Thus our total saving in \eqref{sketch-exp} is 
\begin{align}
\label{comp-sketch}
\min\left\{\frac{C^{3/4}C^{\star 1/2}}{p^{3/4}Q},\frac{C^{5/4}}{p^{3/4}QC^{\star 1/2}}\right\}.
\end{align}
This is sufficient as $C^{3/4}C^{\star 1/2}/p^{3/4}Q>p^\delta CC^{\star 1/2}/pQ^{3/2}$ because $C\ll p^{1-\delta}Q^2$ for some $\delta>0$, and $C^{5/4}/p^{3/4}QC^{\star 1/2}>p^\delta CC^{\star 1/2}/pQ^{3/2}$ because $C^\star\ll C^{1/4}p^{1/4-\delta}Q^{1/2}$, as we are picking $C^\star<C^{1/2}p^{-\delta}$ for some $\delta>0$. This also explains our choice of the sizes for $C^\star$ and $C^\dagger$.\\

Sections~\ref{sec-pre-final} and \ref{sec-final} are devoted to obtaining a sufficient bound for the dual sum $\mathcal{O}_2(C,C^\dagger)$, which is defined in \eqref{od-sum-d-smooth-2-after-c-sum-2}. As we will observe at the beginning of Section~\ref{sec-pre-final} (see \eqref{od-sum-d-smooth-2-after-c-sum-22}) this sum is essentially of the form
\begin{align*}
&\sum_{q\in\mathcal{Q}}\;\sideset{}{^\dagger}\sum_{\psi\bmod{pq}}\; g_{\bar{\psi}^\star}g_{\bar{\tilde\psi}}\mathop{\sum}_{\substack{m\sim pQ^2}} \lambda_f(m)\;\sum_{n\ll Q}\:\sum_{c\sim C^\dagger} \psi(c)\left(\frac{dp}{c}\right)\;\mathcal{J},
\end{align*}
where $\tilde{\psi}=\psi\otimes\chi_p$ and $\psi^\star=\psi\otimes \chi_q$. We seek to save $CC^{\dagger 1/2}/pQ^{3/2}$ in the above sum (beyond square root cancellation in the $\psi$ sum). The above sum is dominated by
\begin{align}
\label{term-sketch}
&\sum_{c\sim C^\dagger}\left|\sum_{q\in\mathcal{Q}}\;\sideset{}{^\dagger}\sum_{\psi\bmod{pq}}\; \psi(c)g_{\bar{\psi}^\star}g_{\bar{\tilde\psi}}\right|\;\left|\mathop{\sum}_{\substack{d\sim D}}\;\beta(d)\left(\frac{d}{c}\right)\right|,
\end{align}
where 
\begin{align*}
\beta(d)=\sum_{n\ll Q} \;\lambda_f(d+pn^2)\;\mathcal{J}.
\end{align*}
The second part of the sum in \eqref{term-sketch} can also be written as 
\begin{align*}
\mathop{\sum\sum}_{\substack{m\sim pQ^2\\n\ll Q}}\;\lambda_f(m)\left(\frac{4m-pn^2}{c}\right)\:\mathcal{J}.
\end{align*}
In Section~\ref{sec-pre-final} we deal with the case where $C^\dagger<p^{1/2+\delta}Q$ for some small $\delta>0$. Roughly speaking, this is the range where $C^\dagger$ is smaller than square-root of the initial size of the modulus $C$ (see \eqref{c-sum-length-initial}). So in this case our treatment is similar to that in Section~\ref{sec-sum-o1}. Indeed when $C\asymp pQ^2$ is in the transition range, so that there is no oscillation in the integral $\mathcal{J}$, the above sum is roughly of the form
\begin{align*}
\frac{1}{c^{1/2}}\sum_{\alpha\bmod{c}}\left(\frac{\alpha}{c}\right)\:\mathop{\sum\sum}_{\substack{m\sim pQ^2\\n\ll Q}}\;\lambda_f(m)e\left(\frac{\alpha(4m-pn^2)}{c}\right).
\end{align*}
We apply the Voronoi summation formula on the sum over $m$ and Poisson summation on the sum over $n$. This transforms the above sum into 
\begin{align*}
\frac{pQ^3}{c^2p^{1/2}}\:\mathop{\sum\sum}_{\substack{m\ll c^2/Q^2\\n\ll c/Q}}\;\lambda_f(m)\mathfrak{c}_c(m-n^2),
\end{align*}
where $\mathfrak{c}_c$ is the Ramanujan sum modulo $c$. Hence, on average over $c$, we have saved $\min\{p^{1/2}Q,p^{1/2}Q^3/C^{\dagger}\}$.
This is enough for our purpose as $p^{1/2}Q>p^\delta CC^{\dagger 1/2}/pQ^{3/2}$ because $C^\dagger <p^{1-\delta}Q$, and $p^{1/2}Q^3/C^{\dagger}>p^\delta CC^{\dagger 1/2}/pQ^{3/2}$ because $C^\dagger<p^{1/3-\delta}Q^{5/3}$ (for the last inequality we need $Q>p^{1/4+\delta}$).\\ 

When $C$ is away from the transition range, the above method does not work, and so we proceed differently.
Applying the Cauchy inequality to \eqref{term-sketch} we end up bounding it by $$\mathcal{A}^{1/2}\:\mathcal{B}^{1/2}$$ where
\begin{align*}
\mathcal{A}=\sum_{c\sim C^\dagger}\left|\sum_{q\in\mathcal{Q}}\;\sideset{}{^\dagger}\sum_{\psi\bmod{pq}}\; \psi(c)g_{\bar{\psi}^\star}g_{\bar{\tilde\psi}}\right|^2,
\end{align*}
and
\begin{align*}
\mathcal{B}=\sum_{c\sim C^\dagger}\;\left|\mathop{\sum}_{\substack{d\sim D}}\;\beta(d)\left(\frac{d}{c}\right)\right|^2.
\end{align*}
In the next step we apply large sieve for quadratic characters. (There is an issue as the variables are not a priori square-free. This needs to be addressed, and we do it directly by extracting the square-free parts from the variables $d$ and $c$.) It follows that 
\begin{align*}
\mathcal{B}\ll p^\varepsilon (C^\dagger + D)\;\sum_{d\sim D}|\beta(d)|^2.
\end{align*}
In the process we have saved $\min\{C^\dagger,D\}$ in $\mathcal{B}$. There is a room of extra saving as the sum $\sum_d |\beta(d)|^2$ is exactly same as the sum $\Omega$ which appeared above in our analysis of the semi-dual $\mathcal{O}_1$. So we have a total saving of
$$
\min\left\{\frac{C^{3/4}C^{\dagger 1/2}}{p^{3/4}Q},\frac{C^{5/4}}{p^{3/4}Q}\right\}
$$
in \eqref{term-sketch}. Compare this with \eqref{comp-sketch}. Again this is fine if $C<p^{1-\delta}Q^2$ and $C^\dagger <p^{1/2}QC^{1/2}$. (One can take $C\gg p^\delta$ for some small $\delta>0$.) Let us demonstrate yet again why we need $Q<p$. Indeed without this condition we save at most $C/Q^2$ in $\Omega$, and hence the total saving in \eqref{term-sketch} is at most $(CC^\dagger)^{1/2}/Q$. This is sufficient for our purpose if $(CC^\dagger)^{1/2}/Q>p^\delta\:CC^{\dagger 1/2}/pQ^{3/2}$, i.e. $C<p^{2-2\delta}Q$. Again, the last inequality holds in the range of $C$ under study if $p^{1-\delta}Q^2<p^{2-2\delta}Q$, i.e. we need $Q<p^{1-\delta}$ for some $\delta>0$. (The condition $Q<p$ also appears in a subtle manner in the proof of Lemma~\ref{c-dagger-2}.) \\

In Section~\ref{sec-final} we treat the case where $p^{1/2+\delta}Q<C^\dagger<pQC^{1/2}$. In this range we have some saving 
in $\mathcal{A}$. Indeed in Lemma~\ref{bound-for-A} we show that we have a saving of $\min\{Q, C^\dagger/p^{1/2}Q\}$ in the sum $\mathcal{A}$. If $p^{1/2+\delta}Q<C^\dagger<p^{1/2}Q^2$, then the total saving in \eqref{term-sketch} is 
$$
\min\left\{\frac{C^{3/4}C^{\dagger}}{pQ^{3/2}},\frac{C^{5/4}C^{\dagger 1/2}}{pQ^{3/2}}\right\}
$$
which is sufficient as $C^{3/4}C^{\dagger}/pQ^{3/2}>p^\delta CC^{\dagger 1/2}/pQ^{3/2}$ because $C\ll pQ^2\ll p^{1+2\delta}Q^2\ll C^{\dagger 2}$ for some $\delta>0$, and $C^{5/4}C^{\dagger 1/2}/pQ^{3/2}>p^\delta CC^{\star 1/2}/pQ^{3/2}$ trivially. On the other hand if $p^{1/2}Q^2<C^\dagger<pQC^{1/2}$, then the total saving in \eqref{term-sketch} is 
$$
\min\left\{\frac{C^{3/4}C^{\dagger 1/2}}{p^{3/4}Q^{1/2}},\frac{C^{5/4}}{p^{3/4}Q^{1/2}}\right\}
$$
which is sufficient as $C^{3/4}C^{\dagger 1/2}/p^{3/4}Q^{1/2}>p^\delta CC^{\dagger 1/2}/pQ^{3/2}$ because $C\ll p^{1+\delta}Q^2\ll pQ^4$ for some $\delta>0$, and $C^{5/4}/p^{3/4}Q^{1/2}>p^\delta CC^{\star 1/2}/pQ^{3/2}$ because $C^\star<C^{1/2}p^{1/2-\delta}Q^2$.
 This explains why we win at the end.\\

 In Section~\ref{sec-shifted} we achieve a non-trivial bound for $\Omega$ \eqref{omega-sketch} by realizing the sum as an averaged shifted convolution sum problem. This is the technical heart of the paper. The particular shifted convolution sum that we need to tackle is of the form
\begin{align*}
\sum_{d\sim D}\lambda_f(d)\lambda_f(d+pr)
\end{align*}
where $r$ is of the form $n_1^2-n_2^2$ and there is an extra average over $n_i$. We solve this additive problem via the circle method. An important point is the particular shape of the shift, namely it is a multiple of $p$ which is the level of the form $p$. This is used in the application of the circle method to reduce the conductor. Indeed the equation $m=d+pr$ is factorized via the congruence $m\equiv d\bmod{p}$ into the smaller integral equation $(m-d-pr)/p=0$ which is detected using the delta method with modulus ranging upto $Q$. Observe that the shifted convolution sum actually comes with an oscillatory weight. This oscillation is large when $C$ is smaller and dies down when $C\sim pQ^2$ is at the transition range. (The origin of this oscillation is the Bessel function which we get from the Petersson formula.) It turns out that when $C\sim pQ^2$ we make the largest possible saving in $\Omega$ which is $Q$. Since there are at most $Q$ many terms inside the absolute value one can not save any more. For smaller values of $C$ we save a little less, due to the analytic oscillation. This is the reason for the extra factors in Proposition~\ref{shifted-prop}. The savings from this proposition is used to get sufficient bounds for $\mathcal{O}_1(C,C^\star)$ as well as $\mathcal{O}_2(C,C^\dagger)$. Note that in Section~\ref{sec-shifted} we need square root bound for Salie sums and Kloosterman sums (due to Weil). For Fourier coefficients we will be using the Deligne bound $\lambda_f(m)\ll m^\varepsilon$, but it seems that one can use weaker results towards Ramanujan conjecture, e.g. $\sum_{m\sim p}|\lambda_f(m)|^2\ll p^{1+\varepsilon}$ and similar bound on average but over shorter interval.

\bigskip


\section{Summation formulae}

We will now analyse the sum $\mathcal{D}$ as defined in \eqref{sum-d}. We are only concerned with odd $\psi$, such that $\psi^2$ is primitive modulo $pq$. In this section we will derive suitable summation formulae for the sums over $m$ and $n$. These will be derived from the respective functional equations. (Following our convention in \cite{Mu5} we will use $V$, and sometimes $W$, to denote generic smooth functions with compact support. They are not necessarily the same in each occurrence.)\\

Corresponding to the sum defined in \eqref{sum-n}
we set the dual sum
\begin{align}
\label{s2-star}
S_2^\star=\varepsilon_2\:\frac{N^2N_\star}{(pQ)^3}\:\mathop{\sum\sum}_{\substack{n,j=1}}^\infty\:\overline{\lambda_g}(n^2)\bar{\psi}^2(j) V\left(\frac{nj^2}{N_\star}\right)
\end{align}
with $N_\star\ll p^{2+\varepsilon}q^2/N$. Here the sign $\varepsilon_2$ is defined by \eqref{sign-2}. In the next lemma we use the notation introduced at the end of Section~\ref{sec-set-up}.\\

\begin{lemma}
\label{dual-s2}
We have
\begin{align*}
\mathcal{D}\triangleleft \left\langle S_1,S_2^\star \right\rangle.
\end{align*}
Here the supremum is taken over all $N_\star\ll p^{2+\varepsilon}Q^2/N$.
\end{lemma}
  
\begin{proof}
Consider the sum over $n$ (as given in \eqref{sum-n}), which by inverse Mellin transform reduces to
$$
S_2=\frac{1}{2\pi i}\int_{(2)}\tilde{V}(s)N^s\:L(s,\text{Sym}^2 g)\mathrm{d}s.
$$
Since we are assuming that $\psi^2$ is primitive modulo $pq$, we can derive the functional equation for the symmetric square $L$-function from \cite{L}. Let
\begin{align*}
\Lambda(s,\text{Sym}^2\:g)=(pq)^s\gamma_2(s)L(s,\text{Sym}^2 g)
\end{align*}
be the completed $L$-function, with 
$$
\gamma_2(s)=\pi^{-3s/2}\Gamma\left(\frac{s+1}{2}\right)\Gamma\left(\frac{s+k-1}{2}\right)\Gamma\left(\frac{s+k}{2}\right).
$$
Then we have the functional equation
\begin{align}
\label{fe2}
\Lambda(s,\text{Sym}^2 g)=\varepsilon_2 \Lambda(1-s,\text{Sym}^2 \bar{g}),
\end{align}
where
\begin{align}
\label{sign-2}
\varepsilon_2=-\frac{\overline{\lambda_g(p^2q^2)}g_\psi^2}{pq}.
\end{align}
We now shift the contour to $-\varepsilon$ and apply the functional equation to arrive at
$$
S_2=\varepsilon_2\:\frac{1}{2\pi i}\int_{(-\varepsilon)}\tilde{V}(s)N^s(pq)^{1-2s}\frac{\gamma_2(1-s)}{\gamma_2(s)}\:L(1-s,\text{Sym}^2 \bar{g})\mathrm{d}s.
$$
Then expanding the $L$-function we get
$$
S_2=\varepsilon_2pq\:\mathop{\sum\sum}_{n,j=1}^\infty \frac{\overline{\lambda}_g(n^2)\bar{\psi}^2(j)}{nj^2}\frac{1}{2\pi i}\int_{(-\varepsilon)}\tilde{V}(s)\left(\frac{Nnj^2}{p^2q^2}\right)^s\frac{\gamma_2(1-s)}{\gamma_2(s)}\mathrm{d}s.
$$
Next we take a smooth dyadic partition of the $(n,j)$ sum. If $nj^2\gg p^{2+\varepsilon}q^2/N$, then we shift the contour to the left and show that the contribution of these terms are negligibly small. For smaller values of $nj^2$ we shift the contour to $2-\varepsilon$. The contribution of these terms are given by sums of the form \eqref{s2-star}. The lemma follows. 
\end{proof}

\bigskip

Next we consider the sum over $m$ which is given by \eqref{sum-m}. We define the corresponding dual sum by
\begin{align}
\label{sum-m1}
S_1^\star=\varepsilon_1\:\frac{N^2}{p^{3/2}Q}\mathop{\sum\sum}_{m,\ell=1}^\infty \:\lambda_f(m)\lambda_g(m)\psi(\ell)\:V\left(\frac{m\ell^2}{M}\right)
\end{align} 
where $p^{3}Q^2/p^\varepsilon N^2\ll M\ll p^{3+\varepsilon}Q^2/N^2$. The sign $\varepsilon_1$ is given by \eqref{sign-1}.\\

Since the functional equation of the Rankin-Selberg $L$-function $L(s,f\otimes \bar{g})$ will play a crucial role here, we will first recall it in some detail. The main reference for this is \cite{L}. Also this is where we will see the crucial role played by our choice of the harmonics. Indeed we are taking $g$ to be a newform of level $pq$ and nebentypus $\psi$ where $\psi^2$ is primitive modulo $pq$, and $f$ is a newform of level $p$ and trivial nebentypus. The inclusion of $p$ the level of $f$ as a divisor of the level of $g$ acts as a level lowering mechanism. Indeed the conductor of $f\otimes \bar{g}$ is $p^3q^2$ instead of $p^4q^2$. One will recognize that this particular drop in conductor makes the subconvexity problem for the symmetric square $L$-function hard to start with. The basic philosophy is that a drop in conductor is bad for amplification technique (as one needs to consider a higher moment) but good for circle method (which is the basis of the present work). Also this will be the reason why we will have shifts of the form $pr$ in the shifted convolution sum problem $\sum_d \lambda_f(d)\lambda_f(d+pr)$ in Section~\ref{sec-shifted}. Again since the shifts are multiples of the level $p$, the congruence-equation trick (see  \cite{Mu4}) will yield a drop in the conductor in the application of the circle method to tackle this problem.\\

We define the completed $L$-function by
$$
\Lambda(s,f\otimes\bar{g})=(p^3q^2)^{s/2}\gamma(s)L(s,f\otimes\bar{g}),
$$
where the gamma factor is given by
$$
\gamma(s)=(2\pi)^{-2s}\Gamma\left(s+\frac{k-\kappa}{2}\right)\Gamma\left(s+\frac{k+\kappa}{2}-1\right).
$$
Here $k$ is the weight of $g$ and $\kappa$ is the weight of $f$.\\

\begin{lemma}
\label{func-eqn-li}
We have the functional equation
\begin{align}
\label{fe1}
\Lambda(s,f\otimes \bar{g})=\varepsilon_0\varepsilon_1\Lambda(1-s,\bar{f}\otimes g)
\end{align}
where $|\varepsilon_0|=1$ depends on $f$ and $k$, and
\begin{align}
\label{sign-1}
\varepsilon_1=\frac{\lambda_g(pq^2)g^2_{\bar{\psi}}}{pq}.
\end{align}
\end{lemma}

\begin{proof}
We will apply the main theorem of \cite{L}. In this proof we use the notation from \cite{L}. So $F_1=f$, $F_2=g$, $N_1=p$, $N_2=pq$, $N=pq$, $\nu_1$ is trivial, $\nu_2=\psi$, $\varepsilon=\bar{\psi}$ (see pp. 141 of \cite{L}). We have the decomposition $N=MM'M''$ with $N=M=pq$ the conductor of $\varepsilon$, and $M'=M''=1$. Hence the conditions of the theorem hold trivially. Also for the prime $q$ we have $Q=q$, $Q'=1$, $Q_1=1$ and $Q_2=q$, and for the prime $p$ we have $Q=p$, $Q'=1$, $Q_1=p$ and $Q_2=p$. \\

Also observe that in our case $\varepsilon'=\varepsilon$ as $\varepsilon$ is primitive (see pp. 142 of \cite{L}). In the notation of Theorem~2.2 of \cite{L} we have
\begin{align*}
L_{F_1,F_2}(s)=L(s,f\otimes \bar{g})
\end{align*}
and
\begin{align*}
\Psi_{F_1,F_2}(s)=\gamma(s)L(s,f\otimes \bar{g}).
\end{align*}
From the theorem we now conclude that
\begin{align*}
\gamma(s)L(s,f\otimes \bar{g})=A(s)\gamma(1-s)L(1-s,f\otimes g)
\end{align*}
(as $f$ is self dual), with
\begin{align*}
A(s)=A_p(s)A_q(s)
\end{align*}
where
\begin{align*}
A_p(s)=g_{\bar{\psi}_p}\Lambda_p(F_1,F_2)p^{1-3s}\bar{\psi}_q(p^2)
\end{align*}
and
\begin{align*}
A_q(s)=g_{\bar{\psi}_q}\Lambda_q(F_1,F_2)q^{1-2s}\bar{\psi}_p(q^2).
\end{align*}
From (2.8) of \cite{L} we get
\begin{align*}
\Lambda_p(F_1,F_2)=\lambda_p(F_1)\overline{\lambda_p(F_2)}=\frac{-1}{p^{1/2}\lambda_f(p)}\:\frac{\overline{g_{\psi_p}}}{p^{1/2}\bar\lambda_g(p)}
\end{align*}
and
\begin{align*}
\Lambda_q(F_1,F_2)=a_2(q)q^{-k/2}\lambda_q(F_1)\overline{\lambda_q(F_2)}=\frac{\lambda_g(q)}{q^{1/2}}\;\frac{\overline{g_{\psi_q}}}{q^{1/2}\bar\lambda_g(q)}.
\end{align*}
It follows that
\begin{align*}
A(s)&=\frac{-1}{p^{1/2}\lambda_f(p)}\:\frac{\psi(-1)g^2_{\bar\psi_p}g^2_{\bar{\psi}_q}\bar{\psi}_p(q^2)\bar{\psi}_q(p^2)}{p^{1/2}q}\:\lambda_g(p)\lambda_g(q)^2\:p^{1-3s}q^{1-2s}\\
&=(-1)^k\;\frac{-1}{p^{1/2}\lambda_f(p)}\;\frac{g_{\bar\psi}^2}{pq}\;\lambda_g(pq^2)\;p^{3/2-3s}q^{1-2s}.
\end{align*}
The lemma follows.
\end{proof}
 
 \bigskip
 
\begin{lemma}
\label{dual-s1-s2}
We have
\begin{align*}
\mathcal{D}\triangleleft \left\langle S_1^\star,S_2^\star \right\rangle,
\end{align*}
where $S_1^\star$ is as given in \eqref{sum-m1} and $S_2^\star$ is as given  in \eqref{s2-star}.
Here the supremum is taken over all $M$ in the range 
$$p^{3}Q^2/p^\varepsilon N^2\ll M\ll p^{3+\varepsilon}Q^2/N^2$$ 
and $N_\star$ in the range
$$N_\star\ll p^{2+\varepsilon}Q^2/N.$$
\end{lemma}

\begin{proof}
Applying Mellin transform we get
$$
S_1=\frac{1}{2\pi i}\int_{(2)} \tilde{W}(s)N^{2s}\:L(s,f\otimes \bar{g})\mathrm{d}s.
$$
We now use the functional equation from the previous lemma to arrive at
\begin{align*}
&\varepsilon_0\frac{\lambda_g(pq^2)g^2_{\bar{\psi}}}{pq}\;(p^3q^2)^{1/2}\:\frac{1}{2\pi i}\int_{(2)} \tilde{W}(s)\left(\frac{N^2}{p^3q^2}\right)^{s}\:\frac{\gamma(1-s)}{\gamma(s)} L(1-s,f\otimes g)\mathrm{d}s
\end{align*}
and then shifting contour to $\sigma=-1$, and expanding the $L$-function into a Dirichlet series we get
\begin{align*}
\varepsilon_0\frac{\lambda_g(pq^2)g^2_{\bar{\psi}}}{pq}\;(p^3q^2)^{1/2}\:&\mathop{\sum\sum}_{m,\ell=1}^\infty \frac{\lambda_f(m)\lambda_g(m)\psi(\ell)}{m\ell^2}\\
&\times \frac{1}{2\pi i}\int_{(-1)} \tilde{W}(s)\left(\frac{N^2m\ell^2}{p^3q^2}\right)^{s}\:\frac{\gamma(1-s)}{\gamma(s)} \mathrm{d}s.
\end{align*}
Recall that the weight $k$ is taken to be large and $\kappa$ is fixed. As there is no pole of the gamma factor $\gamma(1-s)$ in the region $-\infty <\sigma<k/10$, we can shift the contour anywhere in this strip. It follows that if $m\ell^2$ is not in the range $p^{3-\varepsilon}Q^2/N^2\ll m\ll p^{3+\varepsilon}Q^2/N^2$, the integral can be made arbitrarily small, and hence the total contribution of $m$'s outside this range is negligibly small. Accordingly we define the dual sums \eqref{sum-m1}. The lemma follows.
\end{proof}

\bigskip


\section{The dual sum $\left\langle S_1^\star,S_2^\star \right\rangle$: Poisson on $n$}
\label{sec-dual-sum}

In the next lemma we will write the dual sum $\left\langle S_1^\star,S_2^\star \right\rangle$ in terms of sums of the form
\begin{align}
\label{od-sum-d-smooth}
&\mathcal{O}^\star_{\ell,J,\pm}(C)=\frac{J\sqrt{\ell}N^{9/2}N_\star^{1/2}}{C^{1/2}p^{6}Q^{5}}\sum_{j\sim J}\:\sum_{q\in\mathcal{Q}}\;\sideset{}{^\dagger}\sum_{\psi\bmod{pq}}\:\psi(\ell\bar{j}^2)\\
\nonumber &\times \mathop{\sum\sum\sum}_{\substack{m,n,c=1}}^\infty \lambda_f(m)S_\psi(pn^2,m;cpq)e\left(\pm\frac{2n\sqrt{m}}{cq\sqrt{p}}\right) U\left(\frac{m\ell^2}{M},\frac{n j^2}{N_\star},\frac{c}{C}\right),
\end{align}
where
$$
U(x,y,z)=W(x)V(y)V(z). 
$$
The functions $W$ and $V$ are bump functions supported in $[1,2]$, with oscillations of size $O(p^\varepsilon)$, i.e. $W^{(i)}\ll_{\varepsilon, i} p^{\varepsilon i}$ and $V^{(i)}\ll_{\varepsilon, i} p^{\varepsilon i}$. \\

\begin{remark}
\label{wt-convention}
The outer sums over $j$ and $q$ can be endowed with arbitrary weights $\omega_j$ and $\nu_q$ with $|\omega_j|, |\nu_q|\ll 1$. Our analysis works even in this general case. In fact, one will observe that such weights (e.g. $(q/Q)^a$ with an absolutely bounded $a$) arise automatically in the analysis below. But for notational simplicity we will replace them by $1$. This convention will be followed in the rest of the paper.
\end{remark}

\bigskip

\begin{lemma}
\label{diagonal-small}
We have
\begin{align*}
\left\langle S_1^\star,S_2^\star \right\rangle \ll p^\varepsilon\:\sup \ell \left|\mathcal{O}^\star_{\ell,J,\pm}(C)\right| +p^\varepsilon N^{1/2}pQ^2,
\end{align*}
where the supremum is taken over $\ell,J\ll p^{1+\varepsilon}Q$, signs $\pm$, and  
\begin{align}
\label{c-sum-length-initial}
C\ll \frac{N_\star M^{1/2}p^\varepsilon}{\ell J^2p^{1/2}Q}\asymp \frac{N_\star p^{1+\varepsilon}}{\ell J^2N}.
\end{align}
\end{lemma}

\begin{proof}
In the sum \eqref{sum-d} replace $S_1$ by $S_1^\star$ and $S_2$ by $S_2^\star$. Note that the product of the signs of the functional equations $\varepsilon_1\varepsilon_2$ is essentially $\overline{\lambda_g(p)}$. This leads us to consider the sum 
\begin{align}
\label{sum-d-rest}
&\frac{N^{4}N_\star}{p^{9/2}Q^4}\sum_{q\in\mathcal{Q}}\;\sideset{}{^\dagger}\sum_{\psi\bmod{pq}}\:\mathop{\sum\sum}_{\ell,j=1}^{\infty}\;\psi(\ell\bar{j}^2)\\
\nonumber &\times \sum_{g\in H_k(pq,\psi)}\omega_g^{-1}\:\mathop{\sum}_{\substack{m=1}}^\infty \lambda_f(m)\lambda_g(m) W\left(\frac{m\ell^2}{M}\right)
\:\mathop{\sum}_{\substack{n=1}}^\infty\:\overline{\lambda_g(pn^2)}V\left(\frac{nj^2}{N_\star}\right).
\end{align}
We apply the Petersson formula. The diagonal term is given by
\begin{align*}
\frac{N^{4}N_\star}{p^{9/2}Q^4}&\mathop{\sum\sum}_{\ell,j=1}^\infty \sum_{q\in\mathcal{Q}}\;\sideset{}{^\dagger}\sum_{\psi\bmod{pq}}\:\psi(\ell\bar{j}^2)\mathop{\sum}_{\substack{n=1}}^\infty \lambda_f(pn^2)W\left(\frac{pn^2\ell^2}{M}\right)V\left(\frac{nj^2}{N_\star}\right).
\end{align*}
The sum over $n$ is bounded by $O(p^\varepsilon M^{1/2}/\ell p)$, since $|\lambda_f(p)|\ll p^{-1/2}$. Substituting this bound we get
\begin{align*}
p^\varepsilon\frac{N^{4}N_\star M^{1/2}}{p^{11/2}Q^4}&\mathop{\sum\sum}_{\substack{\ell\ll M^{1/2}\\j\ll N_\star^{1/2}}} \sum_{q\in\mathcal{Q}}\;\sideset{}{^\dagger}\sum_{\psi\bmod{pq}}\;\frac{1}{\ell}\ll p^\varepsilon\frac{N^{4}N_\star^{3/2} M^{1/2}}{p^{9/2}Q^2}\ll p^\varepsilon N^{1/2}pQ^2.
\end{align*}
This bound for the diagonal is satisfactory for our purpose. Note that we did not require to use the cancellation in the sum over $\psi$. Utilizing this sum one can get a better bound. But already the above bound is of the strength of 
Lindel\"{o}f.\\

The off-diagonal is essentially given by
\begin{align*}
\mathop{\sum\sum}_{\ell,j\ll p^{1/2}Q}\;\mathcal{O}_{\ell,j},
\end{align*}
where
\begin{align}
\label{od-sum-d-rest}
\mathcal{O}_{\ell,j}=&\frac{N^{4}N_\star}{p^{9/2}Q^4}\sum_{q\in\mathcal{Q}}\;\sideset{}{^\dagger}\sum_{\psi\bmod{pq}}\:\psi(\ell\bar{j}^2)\mathop{\sum}_{\substack{m=1}}^\infty \lambda_f(m) W\left(\frac{m\ell^2}{M}\right)\\
\nonumber &\times 
\:\mathop{\sum}_{\substack{n=1}}^\infty\:V\left(\frac{nj^2}{N_\star}\right)\sum_{c=1}^\infty \frac{S_\psi(pn^2,m;cpq)}{cpq}J_{k-1}\left(\frac{4\pi n\sqrt{m}}{cp^{1/2}q}\right).
\end{align}
We only need to tackle the range for $c$ as defined in \eqref{c-sum-length-initial} as for larger $c$ the Bessel function is small due to the choice of the large weight $k$. Observe that this also implies that we only need to consider $(\ell,j)$ with $\ell j^2\ll N_\star p^{\theta+\varepsilon}$ for some $\theta>0$. 
\\

For $c\sim C$ in the range \eqref{c-sum-length-initial} we use the decomposition 
$$J_{k-1}(4\pi x)=e(2x)\mathcal{W}_{+,k}(x)x^{-1/2}+e(-2x)\mathcal{W}_{-,k}(x)x^{-1/2}$$ 
where 
$$x^{j}\frac{\partial^j}{\partial x^j}[\mathcal{W}_{\pm,k}(x)x^{-1/2}]\ll_j \min\{x^{-1/2},x^{k-1}\}.$$  
One can now use Mellin transform to separate the variables involved in the weight function. Indeed the Mellin transform 
\begin{align*}
\tilde{\mathcal{W}}_{\pm,k}(s)=\int_0^\infty \mathcal{W}_{\pm, k}(x)x^{s-1}\mathrm{d}x
\end{align*}
is holomorphic in the strip $-k+1/2<\sigma<0$. Also in this strip by repeated integration by parts, and using the above bound for the derivatives, we get that 
\begin{align*}
\tilde{\mathcal{W}}_{\pm,k}(s)\ll_i (1+|t|)^{-i}
\end{align*}
for any positive integer $i$. By inverse Mellin transform we get
\begin{align*}
\mathcal{W}_{\pm,k}(x)=\frac{1}{2\pi i}\mathop{\int}_{(-\varepsilon)}\;\tilde{\mathcal{W}}_{\pm,k}(s)\:x^{-s}\mathrm{d}s.
\end{align*}
The integral can be truncated at the height $|t|\ll p^\varepsilon$ at the cost of a negligible error term. In the remaining integral we estimate the sum pointwise for every given $s=-\varepsilon+it$ with $|t|\ll p^\varepsilon$. With this we are able to substitute 
\begin{align*}
e\left(\pm\frac{2n\sqrt{m}}{cq\sqrt{p}}\right)\;\left(\frac{2n\sqrt{m}}{cq\sqrt{p}}\right)^{-1/2+\varepsilon-it}
\end{align*}
with $|t|\ll p^\varepsilon$.
Then we take a smooth dyadic subdivision of the $c$ sum, and a dyadic subdivision of the $j$ sum, to arrive at
the sums introduced in \eqref{od-sum-d-smooth}. The lemma follows. (Note that we are using the convention given in Remark~\ref{wt-convention}.)
\end{proof}

\bigskip

 As it will turn up, our analysis is not sensitive to the sign $\pm$, and hence we will continue with the $+$ term only and will simply write
\begin{align*}
\mathcal{O}^\star(C)=\mathcal{O}^\star_{\ell,J,+}(C).
\end{align*}
Observe that we have substantial oscillation coming from the $J$-Bessel function when $C$ is comparatively small. This will create some complications.
\\

Let 
\begin{align}
\label{l-def}
D=p^\varepsilon\:\frac{CQ^2p^2J^2}{N_\star N \ell},
\end{align}
and
\begin{align}
\label{cal-n}
\mathcal{N}= p^\varepsilon\:\frac{pQ}{N\ell}.
\end{align}
We define the character sum 
\begin{align}
\label{char-sum-1}
\mathcal{C}_\psi=\sum_{\beta\bmod{cq}}S_\psi(p\beta^2,m;cpq)e\left(\frac{\beta n}{cq}\right)
\end{align}
and the integral 
\begin{align}
\label{int}
\mathcal{I}=\int_\mathbb{R}U\left(\frac{m\ell^2}{M},y,\frac{c}{C},\frac{n}{\mathcal{N}}\right)e\left(\frac{2N_\star\sqrt{m}y}{CQ\sqrt{p}J^2}-\frac{N_\star ny}{CQJ^2}\right)\mathrm{d}y.
\end{align}
The weight function 
$$U(x,y,z,w)=V(x)V(y)V(z)W(w)$$
where $V$'s are bump functions supported on $[1,2]$ with $V^{(i)}\ll p^{i\varepsilon}$, and $W$ is a bump function supported on $[-1,1]$. 
We set $\hat{\mathcal{O}}^\star (C)=\sum_{j\sim J}\hat{\mathcal{O}}^\star_j(C)$ with
\begin{align}
\label{od-sum-d-smooth-pre2}
\hat{\mathcal{O}}^\star_j(C)=&\frac{\sqrt{\ell}N^{9/2}N_\star^{3/2}}{JC^{3/2}p^{6}Q^{6}}\:\sum_{q\in\mathcal{Q}}\\
\nonumber &\times \mathop{\sum\sum}_{\substack{1\leq m<\infty\\ |n|\leq \mathcal{N}\\|4m-pn^2|\leq D}} \sum_{c=1}^\infty \lambda_f(m)\;\left[\;\sideset{}{^\dagger}\sum_{\psi\bmod{pq}}\:\psi(\ell\bar{j}^2)\mathcal{C}_\psi\right]\;\mathcal{I}.
\end{align}
Our next lemma says that this sum is a good model for the off-diagonal $\mathcal{O}^\star(C)$. This will be a simple consequence of the Poisson summation formula.
\\

\begin{lemma}
We have 
\begin{align}
\label{od-sum-d-smooth-2}
\mathcal{O}^\star(C)\triangleleft \hat{\mathcal{O}}^\star(C).
\end{align}
\end{lemma}

\begin{proof}
Applying the Poisson summation on the $n$ sum with modulus $cq$, we get
\begin{align*}
\mathcal{O}^\star(C)=&\frac{J\sqrt{\ell}N^{9/2}N_\star^{1/2}}{C^{1/2}p^{6}Q^{5}}\;\sum_{j\sim J}\:\sum_{q\in\mathcal{Q}}\;\sideset{}{^\dagger}\sum_{\psi\bmod{pq}}\:\psi(\ell\bar{j}^2)\\
\nonumber &\times \mathop{\sum}_{\substack{m=1}}^\infty \sum_{c=1}^\infty \lambda_f(m)\;\frac{N^\star}{j^2cq}\sum_{n\in\mathbb{Z}}\:\mathcal{C}_\psi\;\mathcal{I}_0,
\end{align*}
where the character sum $\mathcal{C}_\psi$ is given by \eqref{char-sum-1}, and the integral is given by
\begin{align*}
\mathcal{I}_0=\int_\mathbb{R}U\left(\frac{m\ell^2}{M},y,\frac{c}{C}\right)e\left(\frac{2N_\star\sqrt{m}y}{cq\sqrt{p}j^2}-\frac{N_\star ny}{cqj^2}\right)\mathrm{d}y.
\end{align*}
Here the weight function $U$ is as in \eqref{od-sum-d-smooth}.
By repeated integration by parts we get that the integral is bounded by
\begin{align*}
\left[p^\varepsilon\:\left(1+\frac{N_\star M^{1/2}}{CQ\ell j^2p^{1/2}}\right)\:\frac{CQj^2}{N_\star n}\right]^i\ll \left[p^\varepsilon\:\frac{N_\star M^{1/2}}{CQ\ell j^2p^{1/2}}\:\frac{CQj^2}{N_\star n}\right]^i,
\end{align*}
where we have used \eqref{c-sum-length-initial}. This implies that we only need to consider $n$ with 
\begin{align*}
|n|\ll p^\varepsilon\frac{M^{1/2}}{p^{1/2}\ell}\asymp p^\varepsilon\frac{pQ}{N\ell}.
\end{align*}
So we can now cut the tail of the $n$ sum by introducing a weight function of the form $W(n/\mathcal{N})$ where $W$ is a smooth bump function with support $[-1,1]$.
By integrating by parts differently, it follows that the integral is negligibly small if 
\begin{align*}
\left|\frac{2N_\star\sqrt{m}}{cq\sqrt{p}j^2}-\frac{N_\star n}{cqj^2}\right|\gg p^\varepsilon.
\end{align*}
Consequently we only need to consider $(m,n)$ pairs satisfying 
\begin{align*}
|2\sqrt{m}-\sqrt{p} n|\ll \frac{CQ p^{1/2+\varepsilon}j^2}{N_\star}.
\end{align*}
Multiplying both sides by $|2\sqrt{m}+\sqrt{p}n|$ and using the above obtained bound for $n$, it follows that we only need to consider $(m,n)$ pairs satisfying the condition
\begin{align*}
|4m-pn^2|\ll \frac{CQM^{1/2}p^{1/2+\varepsilon}j^2}{N_\star \ell}\asymp \frac{CQ^2p^{2+\varepsilon}j^2}{N_\star N \ell}.
\end{align*}
This explains the truncations of the sums over $(m,n)$. \\

We now proceed to simplify the integral a bit. Indeed by a change of variables we get
\begin{align*}
\mathcal{I}_0=\frac{cqj^2}{CQJ^2}W\left(\frac{m\ell^2}{M}\right)V\left(\frac{c}{C}\right)\;\int_\mathbb{R}V\left(\frac{cqj^2y}{CQJ^2}\right)e\left(\frac{2N_\star\sqrt{m}y}{CQ\sqrt{p}J^2}-\frac{N_\star ny}{CQJ^2}\right)\mathrm{d}y.
\end{align*}
Then let $V_2$ be a new bump function with support $[1/100,100]$ such that $V_2(x)=1$ for $x\in [1/50,50]$. This weight function can be introduced in the last integral without altering the value, as $V$ is supported in $[1,2]$. Then using Mellin inversion we arrive at
\begin{align*}
\mathcal{I}_0=&\frac{cqj^2}{CQJ^2}W\left(\frac{m\ell^2}{M}\right)V\left(\frac{c}{C}\right)\;\frac{1}{2\pi i}\mathop{\int}_{(0)}\tilde{V}(s)\left(\frac{cqj^2}{CQJ^2}\right)^{-s}\\
&\times \int_\mathbb{R}V_2(y)y^{-s}e\left(\frac{2N_\star\sqrt{m}y}{CQ\sqrt{p}J^2}-\frac{N_\star ny}{CQJ^2}\right)\mathrm{d}y\mathrm{d}s.
\end{align*}
Now at a cost of a negligible error term the integral over $s$ can be truncated at $|t|\ll p^\varepsilon$. From this we can conclude that in our analysis we can replace the integral $\mathcal{I}_0$ by 
the one given in \eqref{int}.
The lemma follows.
\end{proof}

\bigskip

We now seek to prove a bound for the dual off-diagonal contribution of the type
\begin{align}
\label{seek-bd-off-diag}
\hat{\mathcal{O}}^\star_j(C)\ll \frac{N^{1/2}p^{3/2-\theta/2}Q^2}{\ell J}.
\end{align}
This will yield the desired bound for the dual sum $\mathcal{D}$ (as we noted after Lemma~\ref{lemma-initial-second}). One can show quite easily that the above bound holds if $C$ is small enough, e.g. if we have
$$
C\ll \min\left\{\frac{p^{7/3-2\theta/3}Q^{2}\ell^{1/3}}{N^{4/3}N_\star^{1/3}J^{4/3}},\frac{p^{7/2-\theta}Q^{5/2}}{N^{2}N_\star^{1/2}J^2}\right\}.
$$
However since $D$, which will be a part of the conductor of the $c$ sum, gets smaller proportionally with $C$, our treatment below which begins by dualizing the $c$ sum does not get affected by the initial size of $c$. So the above cut-off for the $c$ sum will not be utilized in our analysis below.

\bigskip

\section{Evaluating character sums}
\label{char-sum-eva}

We will now evaluate the character sum in terms of simpler character sums like Gauss sums. (One may compare the results of this section with those in Sections~8 and 9 in \cite{IM}.) For any character $\chi$ modulo $r$ we define the Gauss sum
$$
g_\chi(u)=\sideset{}{^\star}\sum_{a\bmod{r}}\chi(a) e\left(\frac{au}{r}\right),
$$ 
and set $g_\chi=g_\chi(1)$, $g_r(u)=g_{(\frac{.}{r})}(u)$. Also let us define 
\begin{align*}
\tilde{\psi}=\psi \left(\frac{.}{p}\right),\;\;\; \psi^\star=\psi \left(\frac{.}{q}\right),
\end{align*}
which are primitive characters modulo $pq$ as $\psi^2$ is primitive modulo $pq$.\\

\begin{lemma}
\label{lemma:char-sum}
Suppose $(c,pq)=1$ and $c\equiv 1\bmod{4}$, then we have
\begin{align}
\label{char-sum-mid-way-22}
\mathcal{C}_\psi=\sqrt{cq}\:\tilde{\psi}(\overline{4c})\:\psi^\star(d)\;g_c(d)\:g_{\bar{\psi}^\star}.
\end{align}
\end{lemma}

\begin{proof}
Opening the Kloosterman sum we get
\begin{align*}
\mathcal{C}_\psi=\sideset{}{^\star}\sum_{\alpha\bmod{cpq}}\psi(\alpha)e\left(\frac{\bar{\alpha}m}{cpq}\right)\sum_{\beta\bmod{cq}}e\left(\frac{\alpha\beta^2+\beta n}{cq}\right).
\end{align*}
The inner sum is a quadratic Gauss sum, and it can be evaluated explicitly. For a positive integer $r$, and a pair of integers $(a,b)$ with $(a,r)=1$, we set
\begin{align*}
g(a,b;r)=\sum_{x\bmod{r}}e\left(\frac{ax^2+bx}{r}\right).
\end{align*}
The evaluation of this sum depends on the parity of $b$. Let us first focus on the case where $b$ is odd, where we have $g(a,b;r)=0$ if $4|r$, and 
\begin{align*}
g(a,b;r)=\sqrt{2r}\:\varepsilon_{r'}\:\left(\frac{2a}{r'}\right)\:e\left(-\frac{\bar{8a}b^2}{r'}\right)
\end{align*}
if $r=2r'$ with $r'$ odd, and 
\begin{align*}
g(a,b;r)=\sqrt{r}\:\varepsilon_{r}\:\left(\frac{a}{r}\right)\:e\left(-\frac{\bar{4a}b^2}{r}\right)
\end{align*}
if $r$ is odd. Now suppose $b$ is an even integer and we write $b=2b'$. We also set $r=2^kr'$ with $r'$ odd. Then the evaluation depends on the parity of $k$. For $k=0$ we have 
\begin{align*}
g(a,b;r)=\sqrt{r}\:\varepsilon_{r}\:\left(\frac{a}{r}\right)\:e\left(-\frac{\bar{a}b'^2}{r}\right),
\end{align*}
and $g(a,b;r)=0$ if $k=2$, and 
\begin{align*}
g(a,b;r)=\sqrt{r}\:\varepsilon_{r'}\:\left(\frac{a}{r'}\right)\:e\left(-\frac{\bar{a}b'^2}{r}\right)\begin{cases}(1+i\chi_{-4}(r'a)) &\text{$k\geq 2$ even};\\
(\chi_8(a)+i\chi_{-4}(r')\chi_{-8}(a)) &\text{$k\geq 3$ odd}.
\end{cases}
\end{align*}
The formula is notationally nice in the case where $r\equiv 1\bmod{4}$. Indeed for $r\equiv 1\bmod{4}$, and $(a,r)=1$, we have
$$
\sum_{x\bmod{r}}e\left(\frac{ax^2+bx}{r}\right)=\sqrt{r}\:e\left(-\frac{\overline{4a}b^2}{r}\right)\:\left(\frac{a}{r}\right),
$$
no matter whether $b$ is even or odd.
Consequently, for $c\equiv 1\bmod{4}$ we get
\begin{align*}
\mathcal{C}_\psi=\sqrt{cq}\sideset{}{^\star}\sum_{\alpha\bmod{cpq}}\psi(\alpha)\left(\frac{\alpha}{cq}\right) e\left(\frac{\bar{4}\bar{\alpha}d}{cpq}\right),
\end{align*}
where we are using the short hand notation $d=4m-pn^2$. At this point we also observe that the character sum vanishes if $d=0$. In the generic case $(c,pq)=1$ the remaining character sum further splits as 
\begin{align}
\label{char-sum-mid-way}
\mathcal{C}_\psi=\sqrt{cq}\sideset{}{^\star}\sum_{\alpha\bmod{pq}}\psi(\alpha)\left(\frac{\alpha}{q}\right) e\left(\frac{\overline{4\alpha c}d}{pq}\right)\;\sideset{}{^\star}\sum_{\alpha\bmod{c}}\left(\frac{\alpha}{c}\right) e\left(\frac{\overline{4\alpha pq}d}{c}\right).
\end{align}
The lemma now follows.
\end{proof}

\bigskip

In general, for $(c,pq)=1$ we have similar expression for the character sum even if $c\equiv 3\bmod{4}$ or $2|c$. For example, consider the case $2\|c$. We write $c=2c'$ with $c'$ odd. In this case the character sum vanishes if $n$ is even. For odd $n$ we get
\begin{align*}
\mathcal{C}_\psi=\sqrt{2cq}\;\varepsilon_{c'}\:\sideset{}{^\star}\sum_{\alpha\bmod{cpq}}\psi(\alpha)\;\left(\frac{2\alpha}{c'q}\right)\: e\left(\frac{\bar{\alpha}m}{cpq}-\frac{\overline{8\alpha}n^2}{c'q}\right),
\end{align*}
where $\varepsilon_{c'}=1$ if $c'\equiv 1\bmod{4}$ and $i$ otherwise. Then we split the character sum as a product of two character sums. The one with modulus $pq$ is given by
\begin{align*}
\sideset{}{^\star}\sum_{\alpha\bmod{pq}}\psi(\alpha)\;\left(\frac{\alpha}{q}\right)\: e\left(\frac{\overline{\alpha c}m}{pq}-\frac{\overline{8\alpha c'}n^2}{q}\right),
\end{align*}
which exactly coincides with the mod $pq$ sum in \eqref{char-sum-mid-way}. Now consider the character sum modulo $c=2c'$. We observe that the sum vanishes unless $2|m$. In this case we write $m=2m'$ and conclude that the sum modulo $c'$ is given by 
\begin{align*}
\sideset{}{^\star}\sum_{\alpha\bmod{c'}}\;\left(\frac{\alpha}{c'}\right)\: e\left(\frac{\overline{\alpha pq}m'}{c'}-\frac{\overline{8\alpha q}n^2}{c'}\right),
\end{align*}
which one can compare with the mod $c$ sum in \eqref{char-sum-mid-way}. Similar case by case analysis yields explicit expression for the character sum in each case, and it turns out that the expression in Lemma~\ref{lemma:char-sum} is typical.\\

We next remark that in the non generic case, where $(c,pq)\neq 1$, we can obtain a satisfactory bound for the dual off-diagonal $\mathcal{O}^\star(C)$ without much trouble. Indeed,
if $p|c$ then the character sum vanishes unless $p|d$ and consequently $p|m$. In this case as $\lambda_f(p)\approx 1/\sqrt{p}$, we are able to make an extra saving of $p^2$ over $pQ$ which we save from Poisson over $n$ and the sum over $\psi$. Hence the total saving is $p^3Q$, which is satisfactory if $p>Q$. More precisely the contribution of the term with $p|c$ to $\hat{\mathcal{O}}^\star_j(C)$ is bounded by (see \eqref{od-sum-d-smooth-2-trivial})
\begin{align*}
\ll 
&\frac{NN_\star^{2}}{p^{5/2}Q^{1/2}\ell J} \ll  \frac{N^{1/2}Q^{7/2}}{\ell J}
\end{align*}
which is satisfactory for our purpose if $Q<p^{1-2\theta/3}$.
Similarly in the case $q^2|c$, the character sum vanishes unless $q^2|d$. So this gives a total saving of $Q^{3}$ over the bound in \eqref{od-sum-d-smooth-2-trivial}, which is satisfactory if $Q\gg p^{1/3+2\theta/3}$. Recall that we are taking $Q\gg p^{1/2}$. So this contribution is satisfactory for our purpose if $\theta<1/4$. So we are left with two cases - (i) $(c,pq)=1$ and (ii) $q\|c$ with $p\nmid c$. For the second case we have the following lemma.  \\

\begin{lemma}
Suppose $(c,p)=1$ and $q\|c$. Then we have $\mathcal{C}_\psi=0$ if $q\nmid d$. Otherwise we have
\begin{align}
\label{char-sum-mid-way-22-deg}
\mathcal{C}_\psi=\sqrt{c'}\:q^2\; \left(\frac{pq}{c'}\right)\;\psi(\bar{c'}d')g_{c'}(d')\:g_{\bar{\psi}},
\end{align}
where $c=qc'$ and $d=qd'$.
\end{lemma}

\begin{proof}
In the degenerate case $(c,pq)=q$ (assuming that $q$ is prime), we write $c=qc'$. The character sum then splits as 
\begin{align}
\label{char-sum-mid-way-deg}
\mathcal{C}_\psi=\sqrt{c'q^2}\sideset{}{^\star}\sum_{\alpha\bmod{pq^2}}\psi(\alpha) e\left(\frac{\overline{4\alpha c'}d}{pq^2}\right)\;\sideset{}{^\star}\sum_{\alpha\bmod{c'}}\left(\frac{\alpha}{c'}\right) e\left(\frac{\overline{4\alpha pq^2}d}{c'}\right).
\end{align}
This vanishes unless $q|d$, in which we write $d=d'q$. The lemma follows.
\end{proof}

\bigskip

Observe that in the degenerate case $q\|c$ we make an extra saving of $Q^{3/2}$. This simplifies our work tremendously. Moreover the analysis that we will carry out next for the generic case, also works for this degenerate case, and at the end we get a much stronger bound.
\\

\section{Summing over the modulus}
\label{sec-sum-modulus}

For a given integer $d\neq 0$ we introduce the finite Euler product
\begin{align*}
E_\psi(d;s)=&\prod_{\substack{r^\alpha\|d\\\alpha\geq 3\:\text{odd}\\r\;\text{prime}}}\left(\sum_{0\leq j<\alpha/2}\frac{\overline{\tilde{\psi}(r^{2j})}\phi(r^{2j})}{r^{2js+j}}-\frac{\overline{\tilde{\psi}(r^{\alpha+1})}r^\alpha}{r^{(\alpha+1)(s+1/2)}}\right)\\
&\times\prod_{\substack{r^\alpha\|d\\\alpha\geq 2\:\text{even}\\r\;\text{prime}}}\left(\sum_{0\leq j\leq \alpha/2}\frac{\overline{\tilde{\psi}(r^{2j})}\phi(r^{2j})}{r^{2js+j}}+\frac{\overline{\tilde{\psi}(r^{\alpha+1})}\left(\frac{dr^{-\alpha}}{r}\right) r^{\alpha+1/2}}{r^{(\alpha+1)(s+1/2)}}\right).
\end{align*}
In particular if $d$ is square-free then $E_\psi(d;s)=1$. In general we will write $d=d_1d_2^2d_3^2$ with $d_1$ square-free, $d_2|d_1^\infty$, $(d_3,d_1)=1$. For $\Lambda\ll p^{1/\varepsilon}$ a positive real number and $\tau\in\mathbb{R}$, we set
\begin{align}
\label{od-sum-d-smooth-2-after-c-sum}
&\mathcal{O}^\star_{\tau,\Lambda} (C)= \frac{\sqrt{\ell}N^{9/2}N_\star^{3/2}}{Jp^{6}Q^{11/2}}\;\sum_{\lambda\sim \Lambda}\:\frac{\mu(\lambda)}{\lambda^{1+2i\tau}}\sum_{j\sim J}\:\sum_{q\in\mathcal{Q}}\;\sideset{}{^\dagger}\sum_{\psi\bmod{pq}}\:\psi(\ell\bar{j}^2\bar{\lambda}^2)\\
\nonumber &\times \mathop{\sum}_{\substack{m=1}}^\infty \lambda_f(m)\;\sum_{n\in\mathbb{Z}}\:\psi^\star(d)g_{\bar{\psi}^\star}\;L(1/2+i\tau,\bar{\psi}(\tfrac{d_1p}{.}))E_\psi(d;1/2+i\tau)\;\mathcal{J}_\tau.
\end{align}
The new weight function is given by the integral
\begin{align}
\label{int-after-mellin}
\mathcal{J}_\tau=\int_0^\infty\int_\mathbb{R}U\left(\frac{m\ell^2}{M},y,z,\frac{n}{\mathcal{N}}\right)e\left(\frac{2N_\star\sqrt{m}y}{CQ\sqrt{p}J^2}-\frac{N_\star ny}{CQJ^2}\right)z^{-1/2-i\tau}\mathrm{d}y\mathrm{d}z,
\end{align}
which is the Mellin transform of the previous integral \eqref{int}\\

\begin{lemma}
\label{lem-c-sum-l-func}
We have
\begin{align*}
&\hat{\mathcal{O}}^\star(C)\triangleleft \mathcal{O}^\star_{\tau,\Lambda}(C)
\end{align*}
where the family on the right consists of real $\tau$ in the range $|\tau|\ll p^\varepsilon$, and dyadic integers $\Lambda\ll p^{1/\varepsilon}$.
\end{lemma}

\begin{proof}
We consider the sum over $c$. We focus on the case where $(c,pq)=1$ and $c\equiv 1\bmod{4}$. Let
\begin{align}
\label{sum-over-c}
\mathcal{S}=2\sum_{\substack{c\equiv 1\bmod{4}\\(c,pq)=1}} \frac{\overline{\tilde{\psi}(c)}g_c(d)}{c^{1/2}}V\left(\frac{c}{C}\right).
\end{align}
The Gauss sum $g_c(d)$ is not multiplicative in $c$. However the rectified Gauss sum
\begin{align*}
G_c(d)=\left[\frac{1-i}{2}+\left(\frac{-1}{c}\right)\frac{1+i}{2}\right]g_c(d)
\end{align*}
is multiplicative in $c$. In the particular case $c\equiv 1\bmod{4}$ the sums coincide, and so we write 
\begin{align*}
\mathcal{S}=\sum_{\substack{(c,2)=1}} \left(1+\left(\frac{-1}{c}\right)\right)\frac{\overline{\tilde{\psi}(c)}G_c(d)}{c^{1/2}}V\left(\frac{c}{C}\right).
\end{align*}
By inverse Mellin transform we get
\begin{align}
\label{contour-007}
\mathcal{S}=\frac{1}{2\pi i}\int_{(\sigma)} \tilde{V}(s)C^s\left[D_1(s)+D_2(s)\right]\mathrm{d}s
\end{align}
where
\begin{align}
\label{sum-over-c-ds}
D_1(s)=\sum_{\substack{(c,2)=1}} \frac{\overline{\tilde{\psi}(c)}G_c(d)}{c^{s+1/2}},
\end{align}
and $D_2(s)$ is a similar Dirichlet series with an extra twist by the quadratic character modulo $4$.\\

We have the Euler product representation 
\begin{align*}
D_1(s)=&\prod_{\substack{r\nmid d\\r\;\text{prime}}}\left(1+\frac{\overline{\tilde{\psi}(r)}(\frac{d}{r})}{r^s}\right)\times \prod_{\substack{r^\alpha\|d\\\alpha\:\text{odd}}}\left(\sum_{0\leq j<\alpha/2}\frac{\overline{\tilde{\psi}(r^{2j})}\phi(r^{2j})}{r^{2js+j}}-\frac{\overline{\tilde{\psi}(r^{\alpha+1})}r^\alpha}{r^{(\alpha+1)(s+1/2)}}\right)\\
&\times\prod_{\substack{r^\alpha\|d\\\alpha\:\text{even}}}\left(\sum_{0\leq j\leq \alpha/2}\frac{\overline{\tilde{\psi}(r^{2j})}\phi(r^{2j})}{r^{2js+j}}+\frac{\overline{\tilde{\psi}(r^{\alpha+1})}\left(\frac{dr^{-\alpha}}{r}\right) r^{\alpha+1/2}}{r^{(\alpha+1)(s+1/2)}}\right).
\end{align*}
This boils down to
\begin{align*}
D_1(s)=\frac{L(s,\bar{\psi}(\frac{dp}{.}))}{L(2s,\bar{\psi}^2)}E_\psi(d;s)
\end{align*}
where $E_\psi$ is the Euler product defined above, which converges absolutely for $\sigma\geq 1/2$ and satisfies $|E_\psi(d;1/2+it)|\ll (d_2d_3)^\varepsilon$. We now  expand the Dirichlet $L$-function in the denominator as a Dirichlet series. At a cost of a small error, say $O(p^{-2015})$ we can cut the tail of the series at $p^{1/\varepsilon}$. In the remaining sum we take a dyadic subdivision $\lambda\sim \Lambda$, and then move the contour to $\sigma=1/2$. The horizontal line segments again contribute a small error.
With this we are able to estimate the sum $\mathcal{S}$ by 
\begin{align*}
\tilde{V}(1/2+i\tau)C^{1/2+i\tau}\sum_{\lambda\sim\Lambda}\:\frac{\mu(\lambda)\bar{\psi}(\lambda^2)}{\lambda^{1+2i\tau}}\:L\left(\tfrac{1}{2}+i\tau,\bar{\psi}\left(\tfrac{d_1p}{.}\right)\right)\:E_\psi(d;1/2+i\tau)
\end{align*}
with $|\tau|\ll p^\varepsilon$. The lemma follows.
\end{proof}

\bigskip

It will be clear that our analysis is not sensitive to $\tau$ (as long as it is small), and so we will only analyse the case $\tau=0$. In other words we estimate the sum $\mathcal{S}$ by
\begin{align*}
C^{1/2}\;\sum_{\lambda\sim \Lambda}\frac{\mu(\lambda)}{\lambda} \:\bar\psi^2(\lambda)L(1/2,\bar{\psi}(\tfrac{d_1p}{.}))E_\psi(d)
\end{align*}
where the factor $E_\psi(d)=E_\psi(d;1/2)$ is a finite Euler product, and $\Lambda$ ranges upto $p^{1/\varepsilon}$. To make the Euler factor more explicit, let us write $d_1=uv$ with $(u,d_2)=1$ and $v|d_2$. Consequently $d=uw$ with $w=v(d_2d_3)^2$ the powerful part of $d$ and $u$ the square-free part. Then 
\begin{align}
\label{epsi}
E_\psi(d)=\sideset{}{^\#}\sum_{\delta_1,\delta_2}\rho(\delta_1,\delta_2;d_2,d_3)\bar{\psi}(\delta_1\delta_2)\left(\frac{d_1}{\delta_1}\right)
\end{align}
where $\delta_1|d_3r(d_3)$ involves only odd powers of primes and $\delta_2|d_2d_3r(d_2)^2$ involves only even powers of primes. Here $r(k)$ denotes the radical of $k$. The weights $\rho$ are bounded by $O(p^\varepsilon)$, and are arithmetic in nature. The important fact that we need is that they do not depend on $\psi$ or $d_1$. 
\\

Next we will take smooth dyadic subdivision for all the variables, e.g. $d_i\sim D_i$, $u\sim U$, $w\sim W$. There will be some interrelation among the sizes as $4m-pn^2=d=d_1(d_2d_3)^2=uw$. Recall that $d\ll D$. So that we have
$$
D_1(D_2D_3)^2\sim UW\ll D,
$$
also we have $D_1\ll UW^{1/3}$. Then we introduce the following two sums - the semi-dual sum
\begin{align}
\label{od-sum-d-smooth-2-after-c-sum-1}
&\mathcal{O}_1(C,C^\star)=\frac{\sqrt{\ell}N^{9/2}N_\star^{3/2}}{\Lambda Jp^{6}Q^{11/2}}\;\sum_{j\sim J}\:\sum_{\lambda\sim \Lambda}\:\sum_{q\in\mathcal{Q}}\;\sideset{}{^\dagger}\sum_{\psi\bmod{pq}}\:\psi(\ell\bar{j}^2\bar\lambda^2)\\
\nonumber &\times \mathop{\sum}_{\substack{m=1}}^\infty \sum_{n\in\mathbb{Z}}\:\lambda_f(m)\psi^\star(d)\:g_{\bar{\psi}^\star}\;\sum_{c=1}^\infty \frac{\bar{\psi(c)}(\tfrac{d_1p}{c})}{c^{1/2}}V\left(\frac{c}{C^\star}\right)\:E_\psi(d)\;\mathcal{J},
\end{align}
and the full dual sum
\begin{align}
\label{od-sum-d-smooth-2-after-c-sum-2}
&\mathcal{O}_2(C,C^\dagger)=\frac{\sqrt{\ell}N^{9/2}N_\star^{3/2}}{\Lambda J p^{6}Q^{11/2}}\;\sum_{j\sim J}\:\sum_{\lambda\sim L}\:\sum_{q\in\mathcal{Q}}\;\sideset{}{^\dagger}\sum_{\psi\bmod{pq}}\:\psi(\ell\bar{j}^2\bar\lambda^2)\\
\nonumber &\times \mathop{\sum}_{\substack{m=1}}^\infty \sum_{n\in\mathbb{Z}}\:\lambda_f(m)\psi^\star(d)\:g_{\bar{\psi}^\star}\;\frac{g_{\bar{\tilde{\psi}}(\tfrac{d_1}{.})}}{\sqrt{pqd_1}}\sum_{c=1}^\infty \frac{\psi(c)(\tfrac{d_1p}{c})}{c^{1/2}}V\left(\frac{c}{C^\dagger}\right)E_\psi(d)\;\mathcal{J},
\end{align}
where $\mathcal{J}=\mathcal{J}_0$. Note that we are adopting the convention mentioned in Remark~\ref{wt-convention} for the sum over $j$ and $q$, and extending it further to cover the sum over $\lambda$. Note that both the sums depend on the other parameters as well, and one should write $\mathcal{O}_1(C,C^\star;J,\Lambda,D_1,D_2,D_3,U,W,M,N,Q)$ in place of $\mathcal{O}_1(C,C^\star)$, and similarly for $\mathcal{O}_2(C,C^\dagger)$. \\

\begin{lemma}
\label{lem-dirich-poly}
Let $X$ and $Y$ be two positive real numbers such that $XY=pQD_1$. Then we have
\begin{align*}
\hat{\mathcal{O}}^\star(C)\triangleleft |\mathcal{O}_1(C,C^\star)|+|\mathcal{O}_2(C,C^\dagger)|
\end{align*}
where the family for the first term is all dyadic $C^\star\ll Xp^\varepsilon$ and the family for the second term is all dyadic $C^\dagger\ll Yp^\varepsilon$.
\end{lemma}

\begin{proof}
We use the approximate functional equation to expand the $L$-value which appears in \eqref{od-sum-d-smooth-2-after-c-sum} as a finite Dirichlet series. Indeed, since $\bar{\psi}(\frac{d_1p}{.})$ is a primitive Dirichlet character, we have
\begin{align}
\label{afe}
L(1/2,\bar{\psi}(\tfrac{d_1p}{.}))=\sum_{c=1}^\infty &\frac{\bar{\psi(c)}(\tfrac{d_1p}{c})}{c^{1/2}}\Phi_1\left(\frac{c}{X}\right)\\
\nonumber &+\frac{g_{\bar{\psi}(\tfrac{d_1p}{.})}}{\sqrt{pqd_1}}\sum_{c=1}^\infty \frac{\psi(c)(\tfrac{d_1p}{c})}{c^{1/2}}\Phi_2\left(\frac{c}{Y}\right)
\end{align}
with $XY=pQD_1$, where $d_1\sim D_1$. Here the functions $\Phi_i(x)$ decay rapidly for $x\gg p^\varepsilon$, and behave like $1$ for $x\ll p^\varepsilon$, roughly speaking. But the functions are not compactly supported near $0$. Moreover the functions do not depend on $p$ and $q$. On applying this approximate functional equation, the sum in \eqref{od-sum-d-smooth-2-after-c-sum} splits as a sum of two terms which are exactly the sums we defined before the statement of the lemma.
\end{proof}

\bigskip


\section{The semi-dual sum $\mathcal{O}_1(C,C^\star)$}
\label{sec-sum-o1}

The semi-dual sum $\mathcal{O}_1(C,C^\star)$ is structurally almost similar to the initial sum \eqref{od-sum-d-smooth-pre2}. The only advantage that we have gained is the length of the $c$ sum is now shorter, as we can put a bigger mass on the dual side which will have a very different structure. We now seek a satisfactory bound for this sum, more precisely we will prove the following result.\\

\begin{proposition}
\label{prop-for-o1}
Let $\theta<1/104$. Suppose $p^{1/2}<Q<p$ then there exists a computable absolute constant $A>0$ such that
\begin{align}
\label{seek-bd-off-diag-dual1}
\mathcal{O}_1(C,C^\star)\ll \frac{N^{1/2}p^{3/2-\theta}Q^2}{\ell},
\end{align}
for any values of $C$ in the range \eqref{c-sum-length-initial} as long as 
$$C^\star< D_1^{1/2}W^{1/12}p^{-A\theta}.$$
\end{proposition}

This proposition implies that we need to take $C^\dagger$ in the dual sum to range upto $p^{1+\delta}QD_1^{1/2}/W^{1/12}$. Roughly speaking, this is of size $p^{3/2}Q^{2}$.
We note that trivial estimation at this stage, assuming square-root cancellation in the sum over $\psi$, yields
\begin{align*}
\mathcal{O}_1(C,C^\star)\ll \frac{(NpQ)^{1/2}J^2}{\ell^{3/2}}\;C^{\star 1/2}C.
\end{align*}
This is already satisfactory if $C^{\star 1/2}C\ll \ell^{1/2}p^{1-\theta}Q^{3/2}/J^2$. In general, our task will be to save $C^{\star 1/2}CJ^2/\ell^{1/2}p^{1-\theta}Q^{3/2}$. Since $C^\star$ is taken to be smaller than square-root of the initial size of the modulus \eqref{c-sum-length-initial}, we will be able to show that there is a way to save by applying the Voronoi summation on the sum over $m$. But there are other factors which boost up the conductor of the sum, and so our first target will be to control the sizes of these factors. The first lemma in this section serves this purpose. \\ 

Our first step will be an explicit evaluation of the sum over $\psi$ in \eqref{od-sum-d-smooth-2-after-c-sum-1}.
Opening the Gauss sum and the finite Euler factor $E_\psi$ we arrive at
\begin{align*}
\left(\frac{d}{q}\right)&\sideset{}{^\#}\sum_{\delta_1,\delta_2}\rho(\dots)\left(\frac{d_1}{\delta_1}\right)\;\sum_{\alpha\bmod{pq}}\;\left(\frac{\alpha}{q}\right) e\left(\frac{\alpha}{pq}\right)\\
&\times \sideset{}{^\dagger}\sum_{\psi\bmod{pq}}\:\bar{\psi}(\alpha\:\bar{\ell}\bar{d} j^2\lambda^2c\delta),
\end{align*}
where $\delta=\delta_1\delta_2$.
Now the formula \eqref{res-psi-sum} yields a generic term given by
\begin{align*}
\phi(pq)\sideset{}{^\#}\sum_{\delta_1,\delta_2}\rho(\dots)\left(\frac{d_1}{\delta_1}\right)&\left(\frac{\ell c\delta}{q}\right)\:e\left(\pm\frac{\bar{j}^2\bar{\lambda}^2\bar{c}\bar{\delta}\ell d}{pq}\right)
\end{align*}
and two non-generic terms - 
\begin{align*}
\phi(q)\sideset{}{^\#}\sum_{\delta_1,\delta_2}\rho(\dots)\left(\frac{d_1}{\delta_1}\right)&\left(\frac{\ell c\delta}{q}\right)\:e\left(\pm\frac{\bar{p}\bar{j}^2\bar{\lambda}^2\bar{c}\bar{\delta}\ell d}{q}\right)\left\{-1+p^{1/2}\left(\frac{\ell dc\delta q}{p}\right)\right\}
\end{align*}
and
\begin{align*}
\phi(p)\sideset{}{^\#}\sum_{\delta_1,\delta_2}\rho(\dots)\left(\frac{d_1}{\delta_1}\right)&\left(\frac{\ell c\delta}{q}\right)\:e\left(\pm\frac{\bar{q}\bar{j}^2\bar{\lambda}^2\bar{c}\bar{\delta}\ell d}{p}\right)\left\{-1+q^{1/2}\left(\frac{\ell dc\delta p}{q}\right)\right\}.
\end{align*}
 The contribution of the generic term to $\mathcal{O}_1(C,C^\star)$ is given by
\begin{align}
\label{to-estimate-generic-term}
\mathcal{O}_{\text{gen}}=\frac{\sqrt{\ell}N^{9/2}N_\star^{3/2}}{\Lambda Jp^{5}Q^{9/2}}\;&\sum_{q\in\mathcal{Q}}\;\mathop{\sum\sum}_{\substack{j\sim J\\\lambda\sim \Lambda}}\: \mathop{\sum}_{\substack{m=1}}^\infty \sum_{n\in\mathbb{Z}}\:\lambda_f(m)\;\sum_{c=1}^\infty 
\frac{(\tfrac{d_1pq}{c})}{c^{1/2}}V\left(\frac{c}{C^\star}\right)\mathcal{J}\\
\nonumber &\times \sideset{}{^\#}\sum_{\delta_1,\delta_2}\rho(\dots)\left(\frac{d_1}{\delta_1}\right)\;\left(\frac{\ell \delta}{q}\right) e\left(\frac{\bar{j}^2\bar{\lambda}^2\bar{\delta}\bar{c}\ell d}{pq}\right).
\end{align}
(Note that the sum over $j$ and $\lambda$ are restricted by the coprimality condition $(j\lambda,pq)=1$. But instead of mentioning it explicitly here, we adopt the convention given in Remark~\ref{wt-convention}.)
The contributions of the non-generic terms to $\mathcal{O}_1(C,C^\star)$ are  dominated by
\begin{align}
\label{non-generic-one}
\mathcal{O}_{\text{non-gen},1}= &\frac{\sqrt{\ell}N^{9/2}N_\star^{3/2}}{\Lambda Jp^{11/2}Q^{9/2}}\;\sum_{q\in\mathcal{Q}}\:\mathop{\sum}_{\substack{d=uw}} \;\sum_{\delta|wr(w)}\left|\mathop{\sum}_{\substack{n}}\lambda_f(d+pn^2)F(d,n)\right|\\
\nonumber &\times \left|\sum_{j\sim J}\sum_{\lambda\sim \Lambda}\:\sum_{c\sim C^\star} \frac{\nu_1(c,q)}{c^{1/2}}\left(\frac{uv}{c}\right)e\left(\frac{\bar{j}^2\bar{\lambda}^2\bar{c}\bar{\delta}\bar{p}\ell uw}{q}\right)\right|,
\end{align}
and
\begin{align}
\label{non-generic-two}
\mathcal{O}_{\text{non-gen},2}= &\frac{\sqrt{\ell}N^{9/2}N_\star^{3/2}}{\Lambda Jp^{5}Q^{5}}\;\sum_{q\in\mathcal{Q}}\:\mathop{\sum}_{\substack{d=uw}} \;\sum_{\delta|wr(w)}\left|\mathop{\sum}_{\substack{n}}\lambda_f(d+pn^2)F(d,n)\right|\\
\nonumber &\times \left|\sum_{j\sim J}\sum_{\lambda\sim \Lambda}\:\sum_{c\sim C^\star} \frac{\nu_2(c,q)}{c^{1/2}}\left(\frac{uv}{c}\right)e\left(\frac{\bar{j}^2\bar{\lambda}^2\bar{c}\bar{\delta}\bar{p}\ell uw}{q}\right)\right|.
\end{align}
Here $\nu_1(c,q)$ takes two possible values $(\frac{p^iq}{c})$ with $i=1,2$, and $\nu_2(c,q)$ takes two possible values $(\frac{pq^i}{c})$ with $i=1,2$. Also the weight function is given by
\begin{align}
\label{wt-shift-conv-sum}
F(d,n)=&W\left(\frac{n}{\mathcal{N}}\right) V\left(\frac{(d+pn^2)\ell^2}{M}\right)\\
\nonumber &\times \int \;V(y)e\left(\frac{N_\star(\sqrt{d+pn^2}-\sqrt{pn^2})y}{CQ\sqrt{p}J^2}\right)\mathrm{d}y,
\end{align}
where $V$ are bump functions with support $[1,2]$ and $W$ is a bump function with support $[-1,1]$.\\

In our first lemma we will show that we have a satisfactory bound when  $j$, $\lambda$ are not `too small' or when the power-full part of $d$ is not `too small' or when $C$ is not `too big'.\\

\begin{lemma}
\label{first-bound-for-O1star}
Let $\theta<1/24$.
Suppose $C^\star<D_1^{1/2}W^{1/12}p^{-6\theta}$ and $p>Q> p^{1/2}$. Then the bound \eqref{seek-bd-off-diag-dual1} holds for $\mathcal{O}_{\text{gen}}$ (as given in \eqref{to-estimate-generic-term}) if either
$C\ll p^{1-24\theta}Q^2/J^2$ or if
$W(J\Lambda)^2\gg p^{24\theta}$.
\end{lemma}

\begin{proof}
Recall that $4m-pn^2=d=uw$. To simplify the notations a bit, we will replace $m$ by $d+pn^2$ (ignoring $4$). 
With this the expression in \eqref{to-estimate-generic-term} is dominated by
\begin{align}
\label{last-one}
&\frac{\sqrt{\ell}N^{9/2}N_\star^{3/2}}{\Lambda Jp^{5}Q^{9/2}}\;\sum_{q\in\mathcal{Q}}\:\mathop{\sum}_{\substack{d=uw}} \;\sum_{\delta|wr(w)}\left|\mathop{\sum}_{\substack{n}}\lambda_f(d+pn^2)F(d,n)\right|\\
\nonumber &\times \left|\sum_{j\sim J}\sum_{\lambda\sim \Lambda}\:\sum_{c\sim C^\star} \frac{(\tfrac{uvpq}{c})}{c^{1/2}}e\left(\frac{\bar{j}^2\bar{\lambda}^2\bar{c}\bar{\delta}\ell uw}{pq}\right)\right|,
\end{align}
where $F(d,n)$ is as given in \eqref{wt-shift-conv-sum}.
\\

Consider the dyadic segment $u\sim U$ and $w\sim W$ with $UW\ll D$. Applying Cauchy inequality we see that \eqref{last-one} is dominated by
\begin{align}
\label{simply-cauchy}
&\frac{\sqrt{\ell}N^{9/2}N_\star^{3/2}}{\Lambda J p^{5}Q^{4}}\;\Omega^{1/2}\;\mathcal{Z}^{1/2}
\end{align}
where
\begin{align*}
\mathcal{Z}&=\sum_{q\in\mathcal{Q}}\mathop{\sum}_{\substack{w}} \;\sum_{\delta|wr(w)}\; \sum_u \left|\sum_{j\sim J}\sum_{\lambda\sim \Lambda}\:\sum_{c=1}^\infty \frac{(\tfrac{uvpq}{c})}{c^{1/2}}e\left(\frac{\bar{j}^2\bar{\lambda}^2\bar{c}\bar{\delta}\ell uw}{pq}\right)\right|^2,
\end{align*}
and
\begin{align}
\label{Omega-def}
\Omega=\mathop{\sum}_{\substack{w}} \;\sum_{\delta|wr(w)}\;\sum_u\left|\mathop{\sum}_{\substack{n}}\lambda_f(uw+pn^2)F(uw,n)\right|^2.
\end{align}
The trivial bound for $\Omega$ is given by $O(\sqrt{W}Up^{2+\varepsilon} Q^2/N^2\ell^2)$. In Proposition~\ref{shifted-prop} we show that
\begin{align}
\label{omega-bd-from-prop}
\Omega\ll \frac{p^{1+10\theta+\varepsilon}Q^3}{\ell^3}\:\left(\frac{N_\star^{1/2}}{(CJ^2\ell)^{1/2}}+\frac{N_\star}{CJ^2Q}\right).
\end{align}\\

In the expression for $\mathcal{Z}$ we are allowed to drop the arithmetic conditions on $u$ (like square-freeness), and we can also introduce a smooth bump function. We then open the absolute square to arrive at
\begin{align*}
\mathcal{Z}\leq \sum_{q\in\mathcal{Q}}\mathop{\sum}_{\substack{w}} &\;\sum_{\delta|wr(w)}\;\mathop{\sum\sum}_{j_1,j_2\sim J}\: \mathop{\sum\sum}_{\lambda_1,\lambda_2\sim \Lambda}\:\mathop{\sum\sum}_{c_1,c_2\sim C^\star}\:\frac{(\tfrac{vpq}{c_1c_2})}{(c_1c_2)^{1/2}}\\
&\times \sum_{u\in\mathbb{Z}}\; \left(\frac{u}{c_1c_2}\right)e\left(\frac{\bar{\delta}\ell uw(\bar{j}_1^2\bar{\lambda}_1^2\bar{c}_1-\bar{j}_2^2\bar{\lambda}_2^2\bar{c}_2)}{pq}\right)\;V\left(\frac{u}{U}\right).
\end{align*}
Then we apply the Poisson summation on the sum over $u$ with modulus $pqc_1c_2$. This yields a congruence modulo $pq$ and Gauss sums with modulus $c_1c_2$. Indeed Poisson yields
\begin{align*}
\sum_{u\in\mathbb{Z}}\; \left(\frac{u}{c_1c_2}\right)&e\left(\frac{\bar{\delta}\ell uw(\bar{j}_1^2\bar{\lambda}_1^2\bar{c}_1-\bar{j}_2^2\bar{\lambda}_2^2\bar{c}_2)}{pq}\right)\;V\left(\frac{u}{U}\right)\\
&=\frac{U}{c_1c_2pq}\sum_{u\in\mathbb{Z}}\;\mathcal{C}\;\hat{V}\left(\frac{Uu}{c_1c_2pq}\right)
\end{align*}
where the character sum is given by
\begin{align*}
\mathcal{C}=\sum_{a\bmod{c_1c_2pq}}\left(\frac{a}{c_1c_2}\right)e\left(\frac{\bar{\delta}\ell aw(\bar{j}_1^2\bar{\lambda}_1^2\bar{c}_1-\bar{j}_2^2\bar{\lambda}_2^2\bar{c}_2)}{pq}+\frac{au}{c_1c_2pq}\right),
\end{align*}
and $\hat{V}$ is the Fourier transform of $V$. Since $V$ is a compactly supported bump function, it follows that the contribution of $u$ with $|u|\gg C^{\star 2}p^{1+\varepsilon}Q/U$ is negligibly small. Also the character sum splits as a product of two character sums. The one modulo $pq$ vanishes unless we have the congruence relation 
\begin{align*}
\bar{\delta}\ell w(\bar{j}_1^2\bar{\lambda}_1^2\bar{c}_1-\bar{j}_2^2\bar{\lambda}_2^2\bar{c}_2)+u\bar{c}_1\bar{c}_2\equiv 0\bmod{pq},
\end{align*}
in which case the character sum is equal to $pq$. The character sum modulo $c_1c_2$, on the other hand, is a Gauss sum 
\begin{align*}
\left(\frac{pq}{c_1c_2}\right)\;\sum_{a\bmod{c_1c_2}}\left(\frac{a}{c_1c_2}\right)e\left(\frac{au}{c_1c_2}\right).
\end{align*}
Let $c_1c_2=c_3c_4^2$ with $c_3$ square-free. 
Then the Gauss sum is bounded by $O(c_3^{1/2}c_4^2)= O(p^\varepsilon C^\star c_4)$. Consequently we get
\begin{align*}
\mathcal{Z}\leq \frac{p^\varepsilon U}{C^{\star 2}}\:\mathop{\sum}_{\substack{w}} &\;\sum_{\delta|wr(w)}\:\mathop{\sum_{q\in\mathcal{Q}}\;\mathop{\sum\sum}_{j_1,j_2\sim J} \mathop{\sum\sum}_{\lambda_1,\lambda_2\sim \Lambda}\:\mathop{\sum\sum}_{c_1,c_2\sim C^\star}\: \sum_{|u|\ll C^{\star 2}p^{1+\varepsilon}Q/U}}_{\substack{\bar{\delta}\ell w(\bar{j}^2_1\bar{\lambda}_1^2\bar{c}_1-\bar{j}^2_2\bar{\lambda}_2^2\bar{c}_2)+u\bar{c}_1\bar{c}_2\equiv 0\bmod{pq}}}\; c_4.
\end{align*}
We are now left with a weighted counting problem. First
consider the diagonal case where we have the equality 
\begin{align*}
\ell w(j^2_1\lambda_1^2c_1-j^2_2\lambda_2^2c_2)-u\delta (j_1j_2\lambda_1\lambda_2)^2=0.
\end{align*}
Here $u$ is determined uniquely once the other values are given. Moreover we get $j_1\lambda_1|\ell wj_2^2\lambda_2^2c_2$, which implies that there are $p^\varepsilon$ many possibilities for $(j_1,\lambda_1)$ when the other values are given. Consequently the contribution of the diagonal to $\mathcal{Z}$ is given by
\begin{align*}
\mathcal{Z}_0&\ll \frac{p^\varepsilon U}{C^{\star 2}}\:\mathop{\sum}_{\substack{w}} \;\sum_{\delta|wr(w)}\:\sum_{q\in\mathcal{Q}}\;\mathop{\sum}_{j_2\sim J} \mathop{\sum}_{\lambda_2\sim \Lambda}\:\mathop{\sum}_{c_3\ll C^{\star 2}}\;\sum_{c_4\ll C^\star/c_3^{1/2}} c_4\\
&\ll p^\varepsilon\:UW^{1/2}QJ\Lambda,
\end{align*}
 resulting in a saving of $C^\star J\Lambda$ in the diagonal. In the off-diagonal where we do not have the equality we proceed in the following way. First we observe that there are $p^\varepsilon$ many possibilities for $q$. Then we count the number of $u$ modulo $p$. With this we arrive at 
\begin{align*}
&\frac{p^\varepsilon U}{C^{\star 2}}\:\mathop{\sum}_{\substack{w}} \;\sum_{\delta|wr(w)}\;\mathop{\sum\sum}_{j_1,j_2\sim J} \mathop{\sum\sum}_{\lambda_1,\lambda_2\sim \Lambda}\:\mathop{\sum}_{c_3\ll C^{\star 2}}\;\sum_{c_4\ll C^\star/c_3^{1/2}} c_4\:\left(1+\frac{C^{\star 2}Q}{U}\right)\\
&\ll p^\varepsilon\:UW^{1/2}J^2\Lambda^2\:\left(1+\frac{C^{\star 2}Q}{U}\right)\ll p^\varepsilon\:UW^{1/2}J^2\Lambda^2+p^\varepsilon\:C^{\star 2}Q W^{1/2}J^2\Lambda^2.
\end{align*} 
So here we have saved at least $\min\{QC^\star,U/C^{\star}\}$. Consequently we have shown that
\begin{align*}
\mathcal{Z}\ll p^\varepsilon\sqrt{W}J^2\Lambda^2\left( QC^{\star 2}+U+\frac{UQ}{J\Lambda}\right).
\end{align*}\\

Now using the bounds for $\Omega$ as given in \eqref{omega-bd-from-prop} and the above bound for $\mathcal{Z}$, we see that \eqref{simply-cauchy} is dominated by
\begin{align}
\label{simply-cauchy-2}
\nonumber p^{5\theta+\varepsilon}\frac{N^{9/2}N_\star^{3/2}}{\ell p^{9/2}Q^{5/2}}&\;W^{1/4}\;\left(Q^{1/2}C^{\star}+U^{1/2}+\frac{(UQ)^{1/2}}{(J\Lambda)^{1/2}}\right)\\
\times&\left(\frac{N_\star^{1/4}}{(CJ^2\ell)^{1/4}}+\frac{N_\star^{1/2}}{(CJ^2Q)^{1/2}}\right).
\end{align}
Since $UW\ll D=p^\varepsilon CQ^2p^2J^2/N_\star N\ell$ we get that
\begin{align}
\label{2-terms}
\nonumber &p^{5\theta+\varepsilon}\frac{N^{9/2}N_\star^{3/2}}{\ell p^{9/2}Q^{5/2}}\;W^{1/4}\;\left(U^{1/2}+\frac{(UQ)^{1/2}}{(J\Lambda)^{1/2}}\right)\:\left(\frac{N_\star^{1/4}}{(CJ^2\ell)^{1/4}}+\frac{N_\star^{1/2}}{(CJ^2Q)^{1/2}}\right)\\
&\ll p^{5\theta+\varepsilon}\frac{N^{4}N_\star^{5/4}}{\ell^{3/2} p^{7/2}Q}\;\left(\frac{1}{W^{1/4}Q^{1/2}}+\frac{1}{W^{1/4}(J\Lambda)^{1/2}}\right)\:\left(\frac{(CJ^2)^{1/4}}{\ell^{1/4}}+\frac{N_\star^{1/4}}{Q^{1/2}}\right).
\end{align}
Suppose we are in the situation where $Q\gg p^{12\theta}$ and either 
$$C\ll p^{1-24\theta}Q^2/J^2,\;\;\;\text{or}\;\;\; 
W(J\Lambda)^2\gg p^{24\theta}.$$ 
It then follows that
\begin{align*}
\left(\frac{1}{W^{1/4}Q^{1/2}}+\frac{1}{W^{1/4}(J\Lambda)^{1/2}}\right)\:\left(\frac{(CJ^2)^{1/4}}{\ell^{1/4}}+\frac{N_\star^{1/4}}{Q^{1/2}}\right)\ll p^{-6\theta+\varepsilon}\:(pQ^2)^{1/4}.
\end{align*}
Consequently the expression in \eqref{2-terms} is dominated by the right hand side of \eqref{seek-bd-off-diag-dual1}. Now we consider the remaining term in \eqref{simply-cauchy-2} which is given by
\begin{align}
\label{1-term}
p^{5\theta+\varepsilon}\frac{N^{9/2}N_\star^{3/2}}{\ell p^{9/2}Q^{2}}&\;W^{1/4}\:C^{\star}\left(\frac{N_\star^{1/4}}{(CJ^2\ell)^{1/4}}+\frac{N_\star^{1/2}}{(CJ^2Q)^{1/2}}\right).
\end{align}
Suppose we have 
$$C^\star<D_1^{1/2}W^{1/12}p^{-6\theta}\ll U^{1/2}W^{1/4}p^{-6\theta}.$$
So that 
\begin{align*}
W^{1/4}C^\star\ll (UW)^{1/2}p^{-6\theta}\ll p^\varepsilon \frac{(CJ^2)^{1/2}Qp}{(N_\star N\ell)^{1/2}}p^{-6\theta}.
\end{align*} 
Consequently the bound on the right hand side of \eqref{seek-bd-off-diag-dual1} holds for the expression in \eqref{1-term}. The lemma follows.
\end{proof}

\bigskip

We will now consider the non-generic terms \eqref{non-generic-one} and \eqref{non-generic-two}. \\

\begin{lemma}
\label{third-bound-for-O1star}
Let $\theta<1/24$.
Suppose $C^\star<D_1^{1/2}W^{1/12}p^{-6\theta}$ and $p>Q> p^{1/2}$. Then the bound \eqref{seek-bd-off-diag-dual1} holds for $\mathcal{O}_{\mathrm{non-gen},i}$ with $i=1,2$(as given in \eqref{non-generic-one} and \eqref{non-generic-two}).
\end{lemma}

\begin{proof}
Indeed applying Cauchy we get
\begin{align}
\label{simply-cauchy-non-gen-1}
\mathcal{O}_{\text{non-gen},1}\ll  &\frac{\sqrt{\ell}N^{9/2}N_\star^{3/2}}{\Lambda Jp^{11/2}Q^{4}}\;\Omega^{1/2}\:\mathcal{Z}_1^{1/2}
\end{align}
where
\begin{align*}
\mathcal{Z}_1=\sum_{q\in\mathcal{Q}}\mathop{\sum}_{\substack{d=uw}} \;\sum_{\delta|wr(w)}
\left|\sum_{j\sim J}\sum_{\lambda\sim \Lambda}\:\sum_{c\sim C^\star} \frac{\nu_1(c,q)}{c^{1/2}}\left(\frac{uv}{c}\right)e\left(\frac{\bar{j}^2\bar{\lambda}^2\bar{c}\bar{\delta}\bar{p}\ell uw}{q}\right)\right|^2.
\end{align*}
Compare with the expression in \eqref{simply-cauchy}. Here we have an extra factor $p^{1/2}$ in the denominator, and the modulus of the additive character inside the absolute value is $q$ in place of $pq$. We will follow the same steps as in the previous proof. We can here take out the sum over $j$, $\lambda$, then insert a smooth weight for the $u$ sum and then apply the Poisson summation formula after opening the absolute value. This yields
\begin{align*}
\mathcal{Z}_1\ll p^\varepsilon \frac{UJ\Lambda}{C^{\star 2}}\sum_{j\sim J}\sum_{\lambda\sim \Lambda}\mathop{\sum}_{\substack{w}} \;\sum_{\delta|wr(w)}
\:\mathop{\sum_{q\in\mathcal{Q}}\mathop{\sum\sum}_{c_1,c_2\sim C^\star}\sum_{|u|\ll p^\varepsilon C^{\star 2}Q/U}}_{\bar{j}^2\bar{\lambda}^2\bar{\delta}\bar{p}\ell w(\bar{c}_1-\bar{c}_2)+u\bar{c}_1\bar{c}_2\equiv 0\bmod{q}} \:c_4.
\end{align*}
Recall that we are writing $c_1c_2=c_3c_4^2$ with $c_3$ square-free. Now we solve the weighted counting problem. But unlike the generic case we can afford to be a little wasteful. We just count the number of $u$ satisfying the congruence. This shows that 
\begin{align*}
\mathcal{Z}_1\ll p^\varepsilon \frac{UJ\Lambda}{C^{\star 2}}&\sum_{j\sim J}\sum_{\lambda\sim \Lambda}\:\sum_{\substack{w}} \;\sum_{\delta|wr(w)}
\:\sum_{q\in\mathcal{Q}}\mathop{\sum\sum}_{c_1,c_2\sim C^\star} \:c_4\left(1+\frac{C^{\star 2}}{U}\right)\\
&\ll p^\varepsilon (J\Lambda)^2W^{1/2}\left(C^{\star 2}Q+UQ\right).
\end{align*}
So here we save $\min\{U/C^\star,C^\star \}$ over the trivial bound.
Compare with the bound we obtained for $\mathcal{Z}_0$ in the proof of the previous lemma. The second term does not have the extra saving of $J\Lambda$, as in the previous case, but we have an extra saving of $p^{1/2}$ already. Consequently we have established the bound \eqref{seek-bd-off-diag-dual1} for the non-generic term $\mathcal{O}_{\text{non-gen}, 1}$. The same bound is then obtained for $\mathcal{O}_{\text{non-gen},2}$ in exactly the same manner. The lemma follows.
\end{proof}

\bigskip

\begin{proof}[Proof of Proposition~\ref{prop-for-o1}]
In the light of Lemma~\ref{first-bound-for-O1star} we only need to tackle the range 
$$
p^{1-24\theta}Q^2/J^2\ll C\ll p^{1+\varepsilon}N_\star/\ell J^2N.
$$
We will first deal with the generic contribution \eqref{to-estimate-generic-term} in the tamed situation, i.e. $j=\lambda=w=1$ (so that $\delta_1=\delta_2=1$). In this case the expression in \eqref{to-estimate-generic-term} reduces to
\begin{align}
\label{o-generic}
\mathfrak{O}_{\text{gen}}=\frac{\sqrt{\ell}N^{9/2}N_\star^{3/2}}{p^{5}Q^{9/2}}\;&\sum_{q\in\mathcal{Q}}\;\left(\frac{\ell}{q}\right)\mathop{\sum}_{\substack{m=1}}^\infty \sum_{n\in\mathbb{Z}}\:\lambda_f(m)\\
\nonumber &\times \sum_{c=1}^\infty 
\frac{(\tfrac{dpq}{c})}{c^{1/2}}\:  e\left(\frac{\bar{c}\ell d}{pq}\right)\:V\left(\frac{c}{C^\star}\right)\;\mathcal{J}.
\end{align}
Suppose we further assume that we do not have any restriction (e.g. square-freeness) on $d$. We will now prove the following claim.\\

\textbf{Claim:}
Suppose $C\gg p^{1-A\theta}Q^2/J^2$ for some constant $A>0$ satisfying $(1+A)\theta<1/4$, and $p^{1/2}<Q<p^{1-6\theta}$, then the bound \eqref{seek-bd-off-diag-dual1} holds for $\mathfrak{O}_{\text{gen}}$ if $C^\star<p^{1/2-2(A+2)\theta}Q/\ell$.\\

We will apply the Voronoi summation formula on the sum over $m$. Write $c=c_1c_2^2c_3^2$ with $c_1$ square-free, $c_2|c_1^\infty$ and $(c_3,c_1)=1$. Then we have
\begin{align*}
\left(\frac{d}{c}\right)=\mathbf{1}_{(d,c_3)=1}\:\frac{1}{g_{c_1}}\sum_{a\bmod{c_1}}\left(\frac{a}{c_1}\right)\:e\left(\frac{ad}{c_1}\right).
\end{align*}
Consequently
\begin{align*}
\mathfrak{O}_{\text{gen}}=&\frac{\sqrt{\ell}N^{9/2}N_\star^{3/2}}{p^{5}Q^{9/2}}\;\sum_{q\in\mathcal{Q}}\;\left(\frac{\ell}{q}\right)\mathop{\sum}_{\substack{m=1}}^\infty \sum_{n\in\mathbb{Z}}\:\lambda_f(m)\\
\nonumber &\times \sum_{\substack{c=1\\(c_3,d)=1}}^\infty 
\frac{(\tfrac{pq}{c})}{c^{1/2}g_{c_1}}\:\sum_{a\bmod{c_1}}\left(\frac{a}{c_1}\right)\:e\left(\frac{ad}{c_1}\right)\;  e\left(\frac{\bar{c}\ell d}{pq}\right)\:V\left(\frac{c}{C^\star}\right)\;\mathcal{J}.
\end{align*}
We extract the $m$ sum (opening the integral $\mathcal{J}$)
\begin{align*}
\sum_{a\bmod{c_1}}\left(\frac{a}{c_1}\right)\:e\left(-\frac{apn^2}{c_1}\right)&\mathop{\sum}_{\substack{m=1\\(c_3,4m-pn^2)=1}}^\infty \:\lambda_f(m)\:e\left(\frac{4am}{c_1}\right)\\
&\times  e\left(\frac{\bar{c}\ell 4m}{pq}\right)\;e\left(\frac{2N_\star \sqrt{m}y}{CQ\sqrt{p}J^2}\right)V\left(\frac{m\ell^2}{M}\right).
\end{align*}
Then we use the Mobius function to detect the coprimality condition and use the reciprocity relation to arrive at
\begin{align*}
\sum_{\eta|c_3}\mu(\eta)\sum_{a\bmod{c_1}}\left(\frac{a}{c_1}\right)\:e\left(-\frac{apn^2}{c_1}\right)&\mathop{\sum}_{\substack{m=1\\\eta|4m-pn^2}}^\infty \:\lambda_f(m)\:e\left(\frac{4am}{c_1}-\frac{4\bar{pq}\ell m}{c}\right)\\
&\times e\left(\frac{2N_\star \sqrt{m}y}{CQ\sqrt{p}J^2}+\frac{4\ell m}{cpq}\right)V\left(\frac{m\ell^2}{M}\right).
\end{align*}
Using additive characters to detect the divisibility condition and using the fact that $(c_1,\eta)=1$, we arrive at
\begin{align*}
\sum_{\eta|c_3}\frac{\mu(\eta)}{\eta}\sum_{b\bmod{\eta c_1}}\left(\frac{b\eta}{c_1}\right)e\left(-\frac{bpn^2}{\eta c_1}\right)&\mathop{\sum}_{\substack{m=1}}^\infty \:\lambda_f(m)\:e\left(\frac{4bm}{\eta c_1}-\frac{4\bar{pq}\ell m}{c}\right)\\
&\times e\left(\frac{2N_\star \sqrt{m}y}{CQ\sqrt{p}J^2}+\frac{4\ell m}{cpq}\right)V\left(\frac{m\ell^2}{M}\right).
\end{align*}
The sum is now almost ready for an application of the Voronoi summation formula.
We write
\begin{align*}
e\left(\frac{4bm}{\eta c_1}-\frac{4\bar{pq}\ell m}{c}\right)=e\left(\frac{\xi m}{c}\right)=e\left(\frac{\xi_0m}{c_0}\right),
\end{align*}
where $\xi=\xi(b;q)=\xi(b)=4[b(c_2c_3)^2/\eta-\bar{pq}\ell]$ and $(\xi_0,c_0)=1$. Then by Voronoi summation we essentially get
\begin{align*}
\sum_{\eta|c_3}\frac{\mu(\eta)}{\eta}&\sum_{b\bmod{\eta c_1}}\left(\frac{b\eta}{c_1}\right)e\left(-\frac{bpn^2}{\eta c_1}\right)\:\frac{M}{\ell^2 c_0\sqrt{p}}\mathop{\sum}_{\substack{m=1}}^\infty \:\lambda_f(m)\:e\left(\frac{-\overline{p\xi_0} m}{c_0}\right)\\
&\times \int V(x)e\left(\frac{2N_\star \sqrt{Mx}y}{CQ\sqrt{p}J^2\ell}+\frac{4 Mx}{cpq\ell}\right)J_{\kappa-1}\left(\frac{4\pi\sqrt{mMx}}{c_0\ell\sqrt{p}}\right)\mathrm{d}x.
\end{align*}
Extracting the oscillation of the Bessel function and integrating by parts we see that the integral is negligibly small unless we have
\begin{align*}
1\leq m\ll p^{2\theta+\varepsilon}\left(1+\frac{p^{2A\theta}c_0^2\ell^2}{Q^2}\right)=\mathcal{M}.
\end{align*}
Then taking absolute values we get
\begin{align*}
\mathfrak{O}_{\text{gen}}\ll &\frac{N^{5/2}N_\star^{3/2}}{p^{5/2}Q^{5/2}C^{\star 1/2}\ell^{3/2}}\;\sum_{c\sim C^\star}\:\sum_{\eta|c_3}\:\sum_{\delta|c}\:\frac{1}{c_0}\sum_{\substack{|m|\ll \mathcal{M}}} \sum_{|n|\ll \mathcal{N}}\:\\
\nonumber &\times  
\left|\sum_{q\in\mathcal{Q}}\;\left(\frac{\ell}{q}\right)\frac{1}{g_{c_1}\eta}\:\sum_{\substack{b\bmod{\eta c_1}\\(\xi(b),c)=\delta}}\left(\frac{b\eta}{c_1}\right)e\left(-\frac{bpn^2}{\eta c_1}-\frac{\overline{p\xi_0} m}{c_0}\right)\right|,
\end{align*}
where $\xi_0=\xi(b)/\delta$ and $c_0=c/\delta$. Applying Cauchy we get
\begin{align}
\label{back-to-future}
\mathfrak{O}_{\text{gen}}\ll &\frac{N^{5/2}N_\star^{3/2}}{p^{5/2}Q^{5/2}C^{\star 1/2}\ell^{3/2}}\;\sum_{c\sim C^\star}\:\sum_{\eta|c_3}\:\sum_{\delta|c}\:\frac{1}{c_0}\:(\mathcal{M}\mathcal{N})^{1/2}\:\Phi^{1/2}
\end{align}
where
\begin{align*}
\Phi=\sum_{\substack{|m|\ll \mathcal{M}}} \sum_{|n|\ll \mathcal{N}}\:
\left|\sum_{q\in\mathcal{Q}}\;\left(\frac{\ell}{q}\right)\frac{1}{g_{c_1}\eta}\:\sum_{\substack{b\bmod{\eta c_1}\\(\xi(b),c)=\delta}}\left(\frac{b\eta}{c_1}\right)e\left(-\frac{bpn^2}{\eta c_1}-\frac{\overline{p\xi_0} m}{c_0}\right)\right|^2.
\end{align*}
We introduce suitable bump functions and then open the absolute values and apply the Poisson summation formula on the sum over $m$ and $n$. We thus obtain
\begin{align*}
\Phi\ll \frac{\mathcal{M}\mathcal{N}}{\eta c_1c_0}\:\mathop{\sum\sum}_{q_1,q_2\in\mathcal{Q}}\;\frac{1}{c_1\eta^2}\:\sum_{\substack{|m|\ll c_0/\mathcal{M}}} \sum_{|n|\ll \eta c_1/\mathcal{N}}\:
c_0\left|\mathcal{C}\right|,
\end{align*}
where the character sum $\mathcal{C}$ is given by
\begin{align*}
\mathop{\sum\sum}_{\substack{b_1,b_2\bmod{\eta c_1}\\(\xi_1(b_1),c)=(\xi_2(b_2),c)=\delta\\\delta(\xi_2(b_2)-\xi_1(b_1))\equiv mp\xi_1(b_1)\xi_2(b_2)\bmod{\delta^2c_0}}}\left(\frac{b_1b_2}{c_1}\right)\left[\sum_{\alpha\bmod{\eta c_1}}e\left(\frac{p(b_2-b_1)\alpha^2+n\alpha}{\eta c_1}\right)\right].
\end{align*}
Here $\xi_i(b)=\xi(b;q_i)$.
Let us continue our analysis in the case $(\ell,c_2c_3)=1$. (In general, the analysis below goes through but one needs to keep track of the common factors carefully.) In this case $(\xi(b),(c_2c_3)^2/\eta)=1$, and so $\delta|c_1\eta$. Furthermore the congruence condition modulo $\delta^2c_0$ implies that 
$$\delta (q_2-q_1)\equiv m\ell \bmod{(c_2c_3)^2/\eta}.$$
Since $c_1\eta$ is square-free and $\delta|c_1\eta$, the above character sum splits into a product of two character sums. In the part with modulus $\delta$, $b_i$ are uniquely determined and so this part is bounded by $\delta^{1/2}(n,\delta)^{1/2}$. The character sum modulo $c_1\eta/\delta$ splits into a product of character sums modulo each prime factor $r$ of $c_1\eta/\delta$, and they are given by
\begin{align*}
\mathcal{C}_r=\mathop{\sum\sum}_{\substack{b_1,b_2\bmod{r}\\(\xi_1(b_1),r)=(\xi_2(b_2),r)=1\\
\delta(\xi_2(b_2)-\xi_1(b_1))\equiv mp\xi_1(b_1)\xi_2(b_2)\bmod{r}}}\left(\frac{b_1b_2}{r^i}\right)\left[\sum_{\alpha\bmod{r}}e\left(\frac{A(b_2-b_1)\alpha^2+n\alpha}{r}\right)\right].
\end{align*}
Here $A$ is such that $r\nmid A$ and $i=1$ if $r|c_1$ and $i=0$ if $r|\eta$. Now there are two possibilities if $b_1\equiv b_2\bmod{r}$ then the innermost character sum vanishes unless $r|n$. Also the congruence condition boils down to
$$\delta\bar{p}\ell(\bar{q}_2-\bar{q}_1)\equiv mp\xi_1(b_1)\xi_2(b_1)\bmod{r},$$ 
which implies that the remaining sum is bounded by $O(1)$ unless $r|(m,q_1-q_2)$ in which case the reaming sum is bounded by $O(r)$.
So we can bound the contribution of $b_1=b_2$ by $$O((r,n)(r,m,q_1-q_2)).$$ On the other hand if $r\nmid b_1-b_2$ then the sum reduces to
\begin{align*}
r^{1/2}\:\mathop{\sum\sum}_{\substack{b_1,b_2\bmod{r}\\(\xi(b_1),r)=(\xi(b_2),r)=1\\
\delta(\xi(b_2)-\xi(b_1))\equiv mp\xi(b_1)\xi(b_2)\bmod{r}\\r\nmid b_1-b_2}}\left(\frac{b_1b_2}{r^i}\right)\left(\frac{A(b_2-b_1)}{r}\right)\:e\left(-\frac{\bar{4A(b_2-b_1)}n^2}{r}\right).
\end{align*}
Observe that except one special $b_1$, we have $b_2$ uniquely determined by $b_1$ and in this case Weil yields a square-root cancellation in the sum over $b_1$. Also for the special $b_1$, we have a full sum over $b_2$ and again Weil bound yields a squre-root cancellation. In either case we see that the above term is bounded by $O(r)$.
So we conclude that
\begin{align*}
\mathcal{C}_r\ll r(r,n,m,q_1-q_2).
\end{align*}
Consequently we have
\begin{align*}
\mathcal{C}\ll p^\varepsilon \delta^{1/2}(n,\delta)^{1/2}\: \frac{c_1\eta}{\delta}\left(\frac{c_1\eta}{\delta},n,m,q_1-q_2\right).
\end{align*}
From this we get
\begin{align*}
\Phi\ll p^\varepsilon\frac{\mathcal{M}\mathcal{N}Q}{\delta^{1/2}\eta^2 c_1}\:\mathop{\sum}_{|q|\ll Q}\:\sum_{\substack{|m|\ll c_0/\mathcal{M}\\m\equiv \bar{\ell}\delta q\bmod{(c_2c_3)^2/\eta}}} \sum_{|n|\ll \eta c_1/\mathcal{N}}\:
(n,\delta)^{1/2}\:\left(\frac{c_1\eta}{\delta},n,m,q\right),
\end{align*}
which is bounded by
\begin{align*}
p^\varepsilon\frac{\mathcal{M}\mathcal{N}Q}{\delta^{1/2}\eta^2 c_1}\:\mathop{\sum\sum}_{\substack{\nu_1|\delta\\
\nu_2|c_1\eta/\delta}}\nu_1^{1/2}\nu_2\mathop{\sum}_{\substack{|q|\ll Q\\\nu_2|q}}\:\sum_{\substack{|m|\ll c_0/\mathcal{M}\\\nu_2|m\\m\equiv \bar{\ell}\delta q\bmod{(c_2c_3)^2/\eta}}} \sum_{\substack{|n|\ll \eta c_1/\mathcal{N}\\\nu_1\nu_2|n}}\:1.
\end{align*}
So it follows that
\begin{align*}
\Phi\ll p^\varepsilon\frac{\mathcal{M}\mathcal{N}Q}{\delta^{1/2}\eta^2c_1}\:\mathop{\sum\sum}_{\substack{\nu_1|\delta\\
\nu_2|c_1\eta/\delta}}\nu_1^{1/2}\nu_2\:\left(1+\frac{Q}{\nu_2}\right)\left(1+ \frac{\eta c_1}{\mathcal{M}\nu_2\delta}\right) \left(1+ \frac{\eta c_1}{\mathcal{N}\nu_1\nu_2}\right),
\end{align*}
from which we derive
\begin{align*}
\Phi
\ll &p^\varepsilon\left[\frac{\mathcal{M}\mathcal{N}Q}{\eta\delta}+ \frac{\mathcal{N}Q^2}{\eta\delta}+\frac{\mathcal{M}\mathcal{N}Q^2}{\eta^2c_1}+\frac{\mathcal{M}Q^2}{\delta^{1/2}\eta}+\frac{Q^2c_1}{\delta^{3/2}\eta}\right].
\end{align*}
This we now substitute in \eqref{back-to-future}. After an easy but lengthy computation we arrive at
\begin{align*}
\mathfrak{O}_{\text{gen}}&\ll \frac{N^{1/2}p^{1+2\theta+\varepsilon}Q^{5/2}}{\ell^{5/2}}+\frac{N^{1/2}p^{1+2\theta+\varepsilon}Q^2C^{\star 1/4}}{\ell^{2}}+\frac{N^{1/2}p^{1+\theta+A\theta+\varepsilon}Q^{3/2}C^{\star 1/2}}{\ell^{3/2}}\\
&+\frac{N^{1/2}p^{1+\theta+A\theta+\varepsilon}QC^{\star}}{\ell}+N^{1/2}p^{1+2\theta+2A\theta+\varepsilon}C^{\star 3/2}.
\end{align*}
Here we have used the assumption that $(1+A)\theta<1/4$. The claim now follows by taking $C^\star<p^{1/2-2(A+2)\theta}Q/\ell$. Note that from Lemma~\ref{first-bound-for-O1star} we see that we can take $A=24$. Hence if $\theta < 1/100$ and $C^\star <p^{1/2-52\theta}Q/\ell$, we have \eqref{seek-bd-off-diag-dual1} for $\mathfrak{O}_{\text{gen}}$. So now we have a sufficient bound for the generic term \eqref{to-estimate-generic-term} in the special case where $j=\lambda=1$, and under the assumption that in the expression we have $d$ in place of $d_1$.\\

We will now analyse the sum \eqref{to-estimate-generic-term} in full generality. To this end first consider a generalization of the sum \eqref{last-one}, namely
\begin{align}
\label{last-two}
\mathfrak{O}=\frac{\sqrt{\ell}N^{9/2}N_\star^{3/2}}{\Lambda Jp^{5}Q^{9/2}}\;\sum_{q\in\mathcal{Q}}&\;\sum_{\xi\sim \Xi}\mathop{\sum}_{\substack{d=uw\xi^2}} \:\sum_{\delta|wr(w)}\left|\mathop{\sum}_{\substack{n}}\lambda_f(d+pn^2)F(d,n)\right|\\
\nonumber \times &\left|\sum_{j\sim J}\sum_{\lambda\sim \Lambda}\;\sum_{\substack{c\sim C^\star\\(c,\xi)=1}} \frac{(\tfrac{d_1pq}{c})}{c^{1/2}}\:e\left(\frac{\bar{j}^2\bar{\lambda}^2\bar{c}\bar{\delta}\ell d}{pq}\right)\right|.
\end{align}
If we fix $\xi=1$ then the sum reduces to \eqref{last-one}. In general by setting $w'=w\xi^2$ the sum reduces to \eqref{last-one}, with only one extra coprimality condition $(c,\xi)=1$. The reader will observe that the proof of the above lemma goes through even with this restriction on $c$, and one obtains the bound \eqref{seek-bd-off-diag-dual1} for $\mathfrak{O}$ under the conditions of the Lemma~\ref{first-bound-for-O1star}, with the slight difference of $W'=W\Xi^2$ taking place of $W$. \\

Our job has now reduced to proving the bound \eqref{seek-bd-off-diag-dual1} for $\mathcal{O}_{\text{gen}}$ for 
$$C^\star<D_1^{1/2}W^{1/12}p^{-6\theta},$$ in the generic case, i.e. when $C$ is large enough 
$$p^{1-24\theta}Q^2/J^2\ll C$$ and when $W$, $J$, $\Lambda$ are small, i.e.
$W(J\Lambda)^2\ll p^{24\theta}$. Given a powerful integer $w$ of the size $w\sim W$, we can write uniquely $w=v(d_2d_3)^2$ where $v$ is square-free, $(d_2,d_3)=1$ and $v|d_3$. Then we consider the expression 
\begin{align}
\label{to-return-here}
&\frac{\sqrt{\ell}N^{9/2}N_\star^{3/2}}{jp^{5}Q^{9/2}}\;\mathop{\sum}_{w}\sideset{}{^\#}\sum_{\delta_1,\delta_2}\;\Bigl|\sum_{q\in\mathcal{Q}}\;\left(\frac{\ell \delta}{q}\right)\sum_{c=1}^\infty\:\frac{(\tfrac{pq}{c})}{c^{1/2}}V\left(\frac{c}{C^\star}\right)\\
\nonumber &\times \:\mathop{\sum\sum}_{\substack{1\leq m< \infty \\n\in \mathbb{Z}\\w|m-pn^2\\((m-pn^2)/w,wc)=1\\(m-pn^2)/w\:\text{square-free}}} \lambda_f(m)\left(\frac{(m-pn^2)/(d_2d_3)^2}{\delta_1c}\right) e\left(\frac{\bar{j}^2\bar{\lambda}^2\bar{c}\bar{\delta}\ell (m-pn^2)}{pq}\right)\;\mathcal{J}\Bigr|.
\end{align}
The square-free condition can be removed using Mobius function. With this the above sum is dominated by
\begin{align*}
&\frac{\sqrt{\ell}N^{9/2}N_\star^{3/2}}{jp^{5}Q^{9/2}}\;\mathop{\sum}_{w}\sum_{(\xi,w)=1}\sideset{}{^\#}\sum_{\delta_1,\delta_2}\;\Bigl|\sum_{q\in\mathcal{Q}}\;\left(\frac{\ell \delta}{q}\right)\sum_{c=1}^\infty\:\frac{(\tfrac{pq}{c})}{c^{1/2}}V\left(\frac{c}{C^\star}\right)\\
\nonumber &\times \:\mathop{\sum\sum}_{\substack{1\leq m< \infty \\n\in \mathbb{Z}\\w\xi^2|m-pn^2\\((m-pn^2)/w,wc)=1}} \lambda_f(m)\left(\frac{(m-pn^2)/(d_2d_3)^2}{\delta_1c}\right) e\left(\frac{\bar{j}^2\bar{\lambda}^2\bar{c}\bar{\delta}\ell (m-pn^2)}{pq}\right)\;\mathcal{J}\Bigr|.
\end{align*}
In the light of the above observation, the bound \eqref{seek-bd-off-diag-dual1} holds for the above expression if $W(\Xi J\Lambda)^2\gg p^{24\theta}$. So we only need to consider the above sum for $W(\Xi J\Lambda)^2\ll p^{24\theta}$ and $C\gg p^{1-24\theta}Q^2/J^2$. We observe that the sum inside the absolute value is a slight perturbation of the generic sum $\mathfrak{O}_{\text{gen}}$. Indeed if we take $\xi=w=\lambda=j=1$ then the sum boils down to $\mathfrak{O}_{\text{gen}}$ as defined in \eqref{o-generic}. One will now observe that the analysis presented in the proof of the claim above can be now adopted in the present situation, at the cost of introducing a slightly larger modulus. Hence we are able to prove that there exists $A>0$ such that the desired bound holds for this sum if $C^\star \ll D_1^{1/2}W^{1/12}p^{-A\theta}$. The proposition follows.
\end{proof}

\bigskip


\section{The dual sum $\mathcal{O}_2(C,C^\dagger)$ with small $C^\dagger$}
\label{sec-pre-final}

Our goal in this and the next section is to get a satisfactory bound for the sum $\mathcal{O}_2(C,C^\dagger)$ where $C^\dagger \ll p^{1+A\theta}QD_1^{1/2}/W^{1/12}$ for some $A>0$. Since 
\begin{align*}
\psi^\star(d)g_{\bar{\psi}^\star}\;\frac{g_{\bar{\tilde{\psi}}(\tfrac{d_1}{.})}}{\sqrt{pqd_1}}=\varepsilon_{d_1}\psi^2(d_2d_3) \frac{g_{\bar{\psi}^\star}g_{\bar{\tilde{\psi}}}}{\sqrt{pq}}
\end{align*}
the sum in \eqref{od-sum-d-smooth-2-after-c-sum-2} reduces to
\begin{align}
\label{od-sum-d-smooth-2-after-c-sum-22}
\mathcal{O}_2(C,C^\dagger)&=\frac{\sqrt{\ell}N^{9/2}N_\star^{3/2}}{J\Lambda p^{13/2}Q^{6}C^{\dagger 1/2}}\;\mathop{\sum\sum}_{\substack{j\sim J\\\lambda\sim \Lambda}}\sum_{c}\;\sum_{q\in\mathcal{Q}}\\
\nonumber &\times \sideset{}{^\dagger}\sum_{\psi\bmod{pq}}\:\psi(\ell\bar{j}^2\lambda^2)\tilde{\psi}(c)g_{\bar{\psi}^\star}g_{\bar{\tilde{\psi}}}\\
\nonumber &\times \mathop{\sum\sum}_{\substack{m,n}} \lambda_f(m)\;\varepsilon_{d_1}\psi^2(d_2d_3)\left(\frac{d_1}{c}\right)E_\psi(d)\;\mathcal{J}.
\end{align}
Recall that we are writing $4m-pn^2=d_1d_2^2d_3^2=uw$. In this section we prove a sufficient bound for smaller values of $C^\dagger$. The range we will focus on is again of the size $p^{1/2}Q$ (like $C^\star$ in the previous section) which is like square-root of the initial modulus \eqref{c-sum-length-initial}. As one can predict, and as we have seen in the previous section, the Voronoi summation is effective in this range. \\

\begin{proposition}
\label{prop-for-o2}
There exists an absolute computable constant $B_1>0$, such that for any pair $(B,\theta)$ with
$B\geq 1$ and $0<\theta<1/(B_1+6B)$, 
we have
\begin{align}
\label{seek-bound-full-dual}
\mathcal{O}_2(C,C^\dagger)\ll \frac{N^{1/2}p^{3/2-\theta/2}Q^2}{\ell},
\end{align}
for any $C$ in the range \eqref{c-sum-length-initial}, whenever
$$p^{1/2}<Q<p^{1-10(B+5)\theta} \;\;\;\text{and}\;\;\;C^\dagger\ll p^{1/2+B\theta}Q.$$ 
\end{proposition}

\bigskip

Using the expansion \eqref{epsi} it follows that
\begin{align}
\label{the-expression}
\mathcal{O}_2(C,C^\dagger)\ll \frac{\sqrt{\ell}N^{9/2}N_\star^{3/2}}{J\Lambda p^{13/2}Q^{6}C^{\dagger 1/2}}\;\mathop{\sum\sum}_{\substack{j\sim J\\\lambda\sim \Lambda}}\:\Psi
\end{align}
with
\begin{align}
\label{Psi}
\Psi= &\sum_{w\sim W}\sum_{\delta|wr(w)}\:\sum_{c\sim C^\dagger}\:\Bigl|\sum_{q\in\mathcal{Q}}\;\sideset{}{^\dagger}\sum_{\psi\bmod{pq}}\:G_\psi\:\psi(\xi\bar{\zeta} c)\Bigr|\\
\nonumber &\times \Bigl|\mathop{\sum\sum}_{\substack{m,n\\4m-pn^2=uw\\u\sim U \ll D_1\;\Box-\text{free}}} \lambda_f(m)\left(\frac{u}{c}\right) F(m-pn^2,n)\Bigr|
\end{align}
where $G_\psi=g_{\bar{\psi}^\star}g_{\bar{\tilde{\psi}}}$, and the integers $\xi$, $\zeta$ depend on $w$, $j$, $\ell$ and $F$ is as defined in \eqref{wt-shift-conv-sum}. (More precisely $\xi=\ell(d_2d_3 \lambda)^2$ and $\zeta=j^2\delta$ where $\delta|wr(w)$.) Since we will be employing the large sieve inequality for quadratic characters, we need to write $c=c_1c_2^2$ with $c_1$ square-free. The sums will be restricted in dyadic segments $c_i\sim C_i$ with $C_1C_2^2\sim C^\dagger$. As such we shall write $\Psi(C_1,C_2)$ in place of $\Psi$. We wish to use the Voronoi summation and the Poisson summation inside the second absolute value sign. To this end we first need to control the size of the extra factors, e.g. $w$ and the oscillation in $F$, which boost up the conductor.
Our first lemma establishes the desired bound in the case of $W$ too large.\\

\begin{lemma}
\label{trivial-dual}
The bound \eqref{seek-bound-full-dual} holds if $W\gg C^\dagger Qp^{\theta}$.
\end{lemma}

\begin{proof}
 Estimating trivially, taking into account the square-root cancellation in the sum over $\psi$, we get
\begin{align*}
\Psi\ll C^\dagger\;p^{3/2}Q^{5/2}\; \frac{D\mathcal{N}}{W^{1/2}}.
\end{align*}
Now plugging in the largest possible value for $D$ we arrive at
\begin{align}
\label{range-W}
\mathcal{O}_2(C,C^\dagger)\ll \frac{N^{1/2}p^{3/2-\theta/2}Q^2}{\ell}\;\left(\frac{C^{\dagger}Qp^{\theta}}{\ell^3W}\right)^{1/2}.
\end{align}
The lemma follows.
\end{proof}

\bigskip

Next we will state an expression which will be the basis of further analysis for obtaining stronger bounds for the dual off-diagonal. It gives the desired separation of the variables, which is a crucial feature in this approach to subconvexity.\\

\begin{lemma}
\label{lem-for-psi}
We have
\begin{align*}
\Psi(C_1,C_2)\ll p^\varepsilon\mathcal{A}^{1/2}\:\mathcal{B}^{1/2},
\end{align*}
where
\begin{align*}
\mathcal{A}=\sum_{w\sim W}\sum_{\delta|wr(w)}\sum_{c_2\sim C_2}\sum_{c_1\sim C_1}\:\Bigl|\sum_{q\in\mathcal{Q}}\;\sideset{}{^\dagger}\sum_{\psi\bmod{pq}}\:G_\psi\:\psi(\xi\bar{\zeta} c)\Bigr|^2
\end{align*}
and
\begin{align*}
\mathcal{B}=\sum_{w\sim W}\sum_{c_2\sim C_2}\sideset{}{^\flat}\sum_{c_1\sim C_1}\: \Bigl|\sideset{}{^\flat}\sum_{u\sim U} \beta(u,w)\left(\frac{u}{c_1}\right) \Bigr|^2,
\end{align*}
with
\begin{align*}
\beta(u,w)=\mathop{\sum\sum}_{\substack{m,n\\4m-pn^2=uw}} \lambda_f(m) F(m,n)
\end{align*}
if $(u,c_2)=1$ and $0$ otherwise. The superscript $\flat$ indicates that the sum is over square-free integers.
\end{lemma}

\begin{proof}
This is a direct consequence of the Cauchy inequality.
\end{proof}

\bigskip

Let us also introduce a slight perturbation of the $\mathcal{B}$ sum. We set
\begin{align*}
\mathcal{B}'=\mathcal{B}(W',U')=\sum_{w\sim W'}\sum_{c_2\sim C_2}\sideset{}{^\flat}\sum_{c_1\sim C_1}\: \Bigl|\sideset{}{^\flat}\sum_{u\sim U'} \beta(u,w)\left(\frac{u}{c_1}\right) \Bigr|^2,
\end{align*}
where we will keep $U'W'=UW$, and $W'\geq W$. In the same spirit we introduce $\Psi'$ and $\mathcal{O}_2'$. Our next lemma provides an improved range for $C$ in case $C^\dagger$ is small.\\

\begin{lemma}
\label{c-dagger-1}
Suppose $(B,\theta)$ satisfies $0<\theta<1/2(B+3)$. Then
we have \eqref{seek-bound-full-dual}
if $$C^\dagger\ll p^{1/2+B\theta}Q,\;\;\;\text{and}\;\;\; C\ll p^{1-\theta/2}Q^{3/2}/J^2.$$
\end{lemma}

\begin{proof}
 One can show square-root cancellation in the sum over $\psi$ in $\mathcal{A}$. In fact the sum over $\psi$ can be evaluated precisely, and it yields Salie type sums. So one would not need to appeal to Weil's result in this case. Taking this cancellation into account and trivially estimating the remaining sums one gets
\begin{align}
\label{trivial-a}
\mathcal{A}\ll p^\varepsilon W^{1/2}C_1C_2p^3Q^5.
\end{align}
(Another bound for $\mathcal{A}$ will be obtained in the next section. But that bound is non-trivial only for $C^\dagger>p^{1/2+\delta}Q$ for $\delta>0$, which is not the range we are focussing in this section.)
Also by Heath-Brown's large sieve inequality  for quadratic characters (see \cite{HB}) we get
\begin{align*}
\mathcal{B}\ll p^\varepsilon C_2(C_1+U)\sideset{}{^P}\sum_w\Omega(w)
\end{align*}
where
\begin{align}
\label{omega-w}
\Omega_w=\sum_{u\sim U}\:|\beta(u,w)|^2.
\end{align}\\

The trivial bound 
$$\sum_w\Omega_w\ll p^\varepsilon W^{1/2}U\mathcal{N}^2$$ 
yields the bound 
\begin{align*}
\Psi(C_1,C_2)\ll p^\varepsilon W^{1/2}\: p^{3/2}Q^{5/2}\;C_2\{C_1U(C_1+U)\}^{1/2}\mathcal{N}
\end{align*}
and using the expression in \eqref{the-expression} and noting that in the worst case scenario $UW\ll D$, we obtain 
\begin{align}
\label{the-expression-midway}
\mathcal{O}_2(C,C^\dagger)\ll p^\varepsilon\frac{N^{3}N_\star}{\ell p^{3}Q^{3/2}}\;C^{1/2}\left(\frac{C^\dagger}{C_2^2}+\frac{D}{W}\right)^{1/2}.
\end{align} \\

Now  if $C^\dagger\ll p^{1/2+B\theta}Q$, then using $D\ll CQ^2p^{2+\varepsilon}J^2/NN_\star\ell$, in the bound \eqref{the-expression-midway} we get
\begin{align}
\mathcal{O}_2(C,C^\dagger)\ll p^\varepsilon\frac{N^{3}N_\star}{\ell p^{3}Q^{3/2}}\;C^{1/2}\left(p^{1/4+B\theta/2}Q^{1/2}+\frac{C^{1/2}pQj}{(NN_\star \ell)^{1/2}}\right).
\end{align}
The first term satisfies the desired bound if $C\ll p^{3/2-(B+1)\theta}Q^2$, which is always true as we have \eqref{c-sum-length-initial} and by our choice $\theta<1/2(B+3)$. The second term is fine if $C\ll p^{1-\theta/2}Q^{3/2}/J^2$. This proves the lemma.
\end{proof}

\bigskip
In our next lemma we will extend the range of $C$ further. Indeed for larger $C$ we have a non-trivial bound for $\sum_w\Omega_w$. This is the topic of Section~\ref{sec-shifted}.\\

\begin{lemma}
\label{c-dagger-2}
Suppose $Q<p^{1-10(B+5)\theta}$ and $C^\dagger\ll p^{1/2+B\theta}Q$. Then \eqref{seek-bound-full-dual} holds if either $W\gg p^{22\theta}/\ell^2$ or $C\ll N_\star p^{1-22\theta}/\ell j^2N$.
\end{lemma}

\begin{proof}
If $Q<p$, then in the last section we will establish the following bound 
\begin{align*}
\sum_w\Omega_w\ll\frac{p^{1+10\theta+\varepsilon}Q^3}{\ell^3}\:\left(\frac{N_\star^{1/2}}{(CJ^2\ell)^{1/2}}+\frac{N_\star}{CJ^2Q}\right).
\end{align*}
On substituting this bound it follows that 
\begin{align*}
\Psi(C_1,C_2)\ll p^\varepsilon W^{1/4}\:& \frac{p^{2+5\theta}Q^{4}}{\ell^{3/2}}\;C_2\{C_1(C_1+U)\}^{1/2}\\
&\times \left(\frac{N_\star^{1/4}}{(CJ^2\ell)^{1/4}}+\frac{N_\star^{1/2}}{(CJ^2Q)^{1/2}}\right).
\end{align*}
Substituting this bound in \eqref{the-expression} we obtain 
\begin{align}
\label{the-expression-midway-2}
\mathcal{O}_2(C,C^\dagger)\ll & p^{5\theta+\varepsilon}\frac{N^{9/2}N_\star^{3/2}}{p^{9/2}Q^{2}\ell}\;W^{1/4}\;\left(\frac{C^\dagger}{C_2^2}+\frac{D}{W}\right)^{1/2}\\
\nonumber &\times \left(\frac{N_\star^{1/4}}{(CJ^2\ell)^{1/4}}+\frac{N_\star^{1/2}}{(CJ^2Q)^{1/2}}\right).
\end{align}
\\

Consider the term
\begin{align*}
p^{5\theta+\varepsilon}\frac{N^{9/2}N_\star^{3/2}}{p^{9/2}Q^{2}\ell}\;\frac{D^{1/2}}{W^{1/4}}\:
\left(\frac{N_\star^{1/4}}{(CJ^2\ell)^{1/4}}+\frac{N_\star^{1/2}}{(CJ^2Q)^{1/2}}\right).
\end{align*}
Using that $D\ll CQ^2p^{2+\varepsilon}J^2/NN_\star\ell$ \eqref{l-def}, we see that the above term is dominated by
\begin{align*}
p^{5\theta+\varepsilon}\frac{N^{9/2}N_\star^{3/2}}{p^{9/2}Q^{2}\ell}\;\frac{QpJ}{W^{1/4}(NN_\star\ell)^{1/2}}\:
\left(\frac{C^{1/4}N_\star^{1/4}}{J^2\ell}+\frac{N_\star^{1/2}}{(J^2Q)^{1/2}}\right).
\end{align*}
Now plugging in the upper bound for $C$ from \eqref{c-sum-length-initial} and that for $N_\star$ we dominate the term by
\begin{align*}
\frac{N^{1/2}p^{3/2+5\theta+\varepsilon}Q^2}{\ell^{3/2}W^{1/4}}.
\end{align*}
This is satisfactory if $W\gg p^{22\theta}/\ell^2$. Also it follows that this term is satisfactory if $C\ll N_\star p^{1-22\theta+\varepsilon}/\ell J^2N$. (One will note that this contribution will be satisfactory if we have some saving in the sum $\mathcal{A}$, which we are so far estimating trivially. This will be used in the next section.) \\

Next consider the term
\begin{align*}
p^{5\theta+\varepsilon}\frac{N^{9/2}N_\star^{3/2}}{p^{9/2}Q^{2}\ell}\;W^{1/4}\;\frac{C^{\dagger 1/2}}{C_2}\:\left(\frac{N_\star^{1/4}}{(CJ^2\ell)^{1/4}}+\frac{N_\star^{1/2}}{(CJ^2Q)^{1/2}}\right).
\end{align*}
Suppose $C^\dagger\ll p^{1/2+B\theta}Q$. Then this is dominated by
\begin{align*}
&p^{(B+5)\theta+\varepsilon}\;\frac{N^{9/2}N_\star^{3/2}}{p^{17/4}Q^{3/2}\ell}\;W^{1/4}\:\left(\frac{N_\star^{1/4}}{(CJ^2\ell)^{1/4}}+\frac{N_\star^{1/2}}{(CJ^2Q)^{1/2}}\right).
\end{align*}
From Lemma~\ref{trivial-dual} we see that it is enough to take $W\ll C^\dagger Qp^{\theta}$. So $W\ll p^{1/2+(B+1)\theta}Q^2$, and plugging in this bound we arrive at
\begin{align*}
&p^{(5B+21)\theta/4+\varepsilon}\;\frac{N^{9/2}N_\star^{3/2}}{p^{33/8}Q\ell}\:\left(\frac{N_\star^{1/4}}{(CJ^2\ell)^{1/4}}+\frac{N_\star^{1/2}}{(CJ^2Q)^{1/2}}\right).
\end{align*}
Now from Lemma~\ref{c-dagger-1} we see that it is enough to take $C$ in the range $C\gg p^{1-\theta/2}Q^{3/2}/J^2$. Substituting this we see that the above term is dominated by
\begin{align*}
&p^{(5B+21)\theta/4+\varepsilon}\;\frac{N^{9/2}N_\star^{3/2}}{p^{33/8}Q\ell}\:\left(\frac{N_\star^{1/4}p^{\theta/8}}{(p\ell)^{1/4}Q^{3/8}}+\frac{N_\star^{1/2}p^{\theta/4}}{p^{1/2}Q^{5/4}}\right).
\end{align*}
This is bounded by
\begin{align*}
\frac{N^{1/2}p^{3/2-1/8+(5B+23)\theta/4+\varepsilon}Q^{2+1/8}}{\ell}.
\end{align*}
The lemma follows by restricting the size of $Q$ accordingly.
\end{proof}

\bigskip

\begin{proof}[Proof of Proposition~\ref{prop-for-o2}]
In the light of the last two lemmas, if $(B,\theta)$ is such that $\theta<1/2(B+3)$ and if $Q<p^{1-10(B+5)\theta/2}$, $C^\dagger<p^{1/2+B\theta}Q$, then to settle the proposition we only need to tackle the case where 
$$W\ll p^{22\theta}/\ell^2,\;\;\;\;\text{and}\;\;\;\; N_\star p^{1-22\theta}/\ell J^2N\ll C\ll N_\star p^{1+\varepsilon}/\ell J^2N.$$ 
(Recall \eqref{c-sum-length-initial}.)
So in particular $\ell\ll p^{11\theta}$. We return to the expression given in \eqref{Psi}, and consider the term within the second absolute value 
\begin{align}
\label{general}
&\mathop{\sum\sum}_{\substack{m,n\\4m-pn^2=uw\\u\sim U \ll D_1\;\Box-\text{free}}} \lambda_f(m)\left(\frac{u}{c}\right) F(m-pn^2,n).
\end{align}
First let us study an idealized version of this sum where $w=1$, $C\asymp N_\star p/\ell J^2N$ so that there is no oscillation in the integral, and suppose there is no square-freeness condition on $u$. In this case the sum essentially reduces to
\begin{align}
\label{ideal}
&\mathop{\sum\sum}_{\substack{1\leq m<\infty\\n\in \mathbb{Z}}} \lambda_f(m)\left(\frac{4m-pn^2}{c}\right) V\left(\frac{m\ell^2}{M}\right)W\left(\frac{n}{\mathcal{N}}\right),
\end{align}
where $V$ is a bump function with support $[1,2]$ and $W$ is a bump function with support $[-1,1]$. Then we write $c=c_1c_2^2c_3^2$ where $c_1$ is square-free, $c_2|c_1^\infty$ and $(c_3,c_1)=1$. (Note that we are using a slightly different decomposition instead of just writing $c=c_1c_2^2$.) So we can write the sum as
\begin{align*}
&\mathop{\sum\sum}_{\substack{1\leq m<\infty\\n\in \mathbb{Z}\\(4m-pn^2,c_3)=1}} \lambda_f(m)\left(\frac{4m-pn^2}{c_1}\right) V\left(\frac{m\ell^2}{M}\right)W\left(\frac{n}{\mathcal{N}}\right).
\end{align*}
Then detecting the coprimality condition using Mobius and shifting to multiplicative characters using Gauss sums we arrive at
\begin{align*}
&\sum_{\eta|c_3}\mu(\eta)\frac{1}{g_{c_1}}\sum_{\alpha\bmod{c_1}}\left(\frac{\alpha}{c_1}\right)\\
&\times \mathop{\sum\sum}_{\substack{1\leq m<\infty\\n\in \mathbb{Z}\\\eta|4m-pn^2}} \lambda_f(m)e\left(\frac{\alpha(4m-pn^2)}{c_1}\right) V\left(\frac{m\ell^2}{M}\right)W\left(\frac{n}{\mathcal{N}}\right).
\end{align*}
Then we detect the divisibility condition using additive characters to get
\begin{align*}
&\sum_{\eta|c_3}\sum_{\xi|\eta}\frac{\mu(\eta)}{\eta g_{c_1}}\sideset{}{^\star}\sum_{\beta\bmod{\xi}}\sum_{\alpha\bmod{c_1}}\left(\frac{\alpha}{c_1}\right)\\
&\times \mathop{\sum\sum}_{\substack{1\leq m<\infty\\n\in \mathbb{Z}}} \lambda_f(m)e\left(\frac{(\alpha\xi+\beta c_1)(4m-pn^2)}{c_1\xi}\right) V\left(\frac{m\ell^2}{M}\right)W\left(\frac{n}{\mathcal{N}}\right).
\end{align*}
Observe that $(\alpha\xi+\beta c_1,c_1\xi)=1$. So applying the Voronoi summation formula on the sum over $m$ and the Poisson summation on the sum over $n$ we transform the above sum to
\begin{align*}
&\sum_{\eta|c_3}\sum_{\xi|\eta}\frac{\mu(\eta)}{\eta g_{c_1}}\sideset{}{^\star}\sum_{\beta\bmod{\xi}}\sum_{\alpha\bmod{c_1}}\left(\frac{\alpha}{c_1}\right)\;\frac{M}{\ell^2c_1\xi p^{1/2}}\\
&\times \mathop{\sum}_{\substack{m=1}}^\infty \lambda_f(m)e\left(-\frac{\bar{4p(\alpha\xi+\beta c_1)}m}{c_1\xi}\right)\int  V\left(x\right)J_{\kappa-1}\left(\frac{4\pi\sqrt{Mmx}}{c_1\xi p^{1/2}\ell}\right)\mathrm{d}x\\
&\times \frac{\mathcal{N}}{c_1\xi}\mathop{\sum}_{\substack{n\in \mathbb{Z}}} \sum_{\gamma\bmod{c_1\xi}}e\left(\frac{-(\alpha\xi+\beta c_1)p\gamma^2+n\gamma}{c_1\xi}\right)\;\hat W\left(-\frac{n\mathcal{N}}{c_1\xi}\right).
\end{align*}
The effective length of the $m$ sum is $(c_1\xi)^2p\ell^2/M$, and that of the $n$ sum is $c_1\xi/\mathcal{N}$. Now consider the character sum
\begin{align*}
&\sideset{}{^\star}\sum_{\beta\bmod{\xi}}\sum_{\alpha\bmod{c_1}}\left(\frac{\alpha}{c_1}\right)\;e\left(-\frac{\bar{4p(\alpha\xi+\beta c_1)}m}{c_1\xi}\right)\\
&\times \sum_{\gamma\bmod{c_1\xi}}e\left(\frac{-(\alpha\xi+\beta c_1)p\gamma^2+n\gamma}{c_1\xi}\right).
\end{align*}
The innermost sum is a Gauss sum which can be evaluated explicitly. In the typical case where $c_1\xi\equiv 1\bmod{4}$, the above sum reduces to
\begin{align*}
&(c_1\xi)^{1/2}\:\sideset{}{^\star}\sum_{\beta\bmod{\xi}}\sum_{\alpha\bmod{c_1}}\left(\frac{\alpha}{c_1}\right)\left(\frac{(\alpha\xi+\beta c_1)p}{c_1\xi}\right)\;e\left(-\frac{\bar{4p(\alpha\xi+\beta c_1)}(m-n^2)}{c_1\xi}\right).
\end{align*}
Observe that the character sum modulo $c_1$ is a Ramanujan sum and that modulo $\xi$ is a Gauss sum. Consequently the above sum is bounded by $O(\xi c_1^{1/2}(c_1,m-n^2))$. Hence \eqref{ideal} is bounded by
\begin{align*}
&\frac{M\mathcal{N}}{\ell^2c_1^{2} p^{1/2}}\sum_{\eta|c_3}\sum_{\xi|\eta}\frac{1}{\eta \xi}\:\mathop{\sum}_{\substack{1\leq m\ll (c_1\xi)^2p^{1+\varepsilon}\ell^2/M}}\:\mathop{\sum}_{\substack{|n|\ll p^\varepsilon c_1\xi/\mathcal{N}}} \:(c_1,m-n^2).
\end{align*}
If we plug in this bound for the innermost term in \eqref{Psi}, we get that the total contribution of this term to $\mathcal{O}_2(C,C^\dagger)$ is dominated by
\begin{align*}
&\frac{\sqrt{\ell}N^{9/2}N_\star^{3/2}}{p^{5}Q^{7/2}C^{\dagger 1/2}}\:\frac{M\mathcal{N}}{\ell^2p^{1/2}}\:\sum_{c\sim C^\dagger}\frac{1}{c_1^2}\sum_{\eta|c_3}\sum_{\xi|\eta}\frac{1}{\eta \xi}\\
\nonumber &\times \mathop{\sum}_{\substack{1\leq m\ll (c_1\xi)^2p^{1+\varepsilon}\ell^2/M}}\:\mathop{\sum}_{\substack{|n|\ll p^\varepsilon c_1\xi/\mathcal{N}}} \:(c_1,m-n^2).
\end{align*}
The contribution of $n=0$ is dominated by
\begin{align*}
&\frac{\sqrt{\ell}N^{9/2}N_\star^{3/2}}{p^{5}Q^{7/2}C^{\dagger 1/2}}\:\frac{M\mathcal{N}}{\ell^2p^{1/2}}\:\sum_{c\sim C^\dagger}\frac{1}{c_1^2}\sum_{\eta|c_3}\sum_{\xi|\eta}\frac{1}{\eta \xi}\: \mathop{\sum}_{\substack{1\leq m\ll (c_1\xi)^2p^{1+\varepsilon}\ell^2/M}}\:(c_1,m)\\
&\ll \frac{\sqrt{\ell}N^{9/2}N_\star^{3/2}}{p^{5}Q^{7/2}}\:p^{1/2+\varepsilon}\mathcal{N}\:C^{\dagger 1/2}\ll \frac{N^{1/2}p^{3/2-1/4+(B+11)\theta/2+\varepsilon}Q}{\ell}.
\end{align*}
(Recall that $\ell\ll p^{11\theta}$.) Now consider the case $n\neq 0$. The contribution of the terms satisfying $m=n^2$ is bounded by
\begin{align*}
&\frac{\sqrt{\ell}N^{9/2}N_\star^{3/2}}{p^{5}Q^{7/2}C^{\dagger 1/2}}\:\frac{M\mathcal{N}}{\ell^2p^{1/2}}\:\sum_{c\sim C^\dagger}\frac{1}{c_1^2}\sum_{\eta|c_3}\sum_{\xi|\eta}\frac{1}{\eta \xi}\: \mathop{\sum}_{\substack{1\leq |n|\ll p^\varepsilon c_1\xi/\mathcal{N}}}\:c_1\\
&\ll \frac{\sqrt{\ell}N^{9/2}N_\star^{3/2}}{p^{5}Q^{7/2}}\:\frac{p^{\varepsilon}MC^{\dagger 1/2}}{\ell^2p^{1/2}}\ll \frac{N^{1/2}p^{3/2-1/4+B\theta/2+\varepsilon}Q^{2}}{\ell},
\end{align*}
which dominates the contribution of $n=0$.
Finally the contribution of $n\neq 0$ and $m\neq n^2$, when $c_1\sim C_1$, is bounded by
\begin{align*}
&\frac{\sqrt{\ell}N^{9/2}N_\star^{3/2}}{p^{5}Q^{7/2}C^{\dagger 1/2}}\:\frac{M\mathcal{N}}{\ell^2p^{1/2}}\:\sum_{c_3\ll C^{\dagger 1/2}/C_1^{1/2}}\frac{1}{C_1^2}\sum_{\eta|c_3}\sum_{\xi|\eta}\frac{1}{\eta \xi}\: \mathop{\sum\sum}_{\substack{1\leq |n|\ll p^\varepsilon C_1\xi/\mathcal{N}\\1\leq m\ll (C_1\xi)^2p^{1+\varepsilon}\ell^2/M}}\:C_1\\
&\ll \frac{\sqrt{\ell}N^{9/2}N_\star^{3/2}}{p^{5}Q^{7/2}}\:p^{1/2+\varepsilon}C^{\dagger 3/2}\ll \frac{N^{1/2}p^{3/2+1/4+3(B+11)\theta/2+\varepsilon}Q}{\ell}.
\end{align*}
This dominates the previous contributions if $Q>p^{1/2}$.
\\

Now we return to the general sum \eqref{general}. Though $w$ may not be one or $C$ may not be optimal, but we can take them to be nearly one and nearly optimal respectively. This is the main outcome of our two  previous lemmas. We will now tackle the square-free condition.
 Detecting the square-free condition using Mobius we arrive at
\begin{align*}
&\mathop{\sum\sum}_{\substack{m,n\\4m-pn^2=uw\\u\sim U \ll D_1\;\Box-\text{free}}} \lambda_f(m)\left(\frac{u}{c}\right) F(m,n)\\
&=\sum_{\substack{\delta=1\\(\delta,c)=1}}^\infty \mu(\delta)\mathop{\sum\sum}_{\substack{m,n\\4m-pn^2=u\delta^2w}} \lambda_f(m)\left(\frac{u}{c}\right) F(m,n).
\end{align*}
Then we write $u=u'u_0$ with $u'$ maximal square-free part with $(u',u_0\delta w)=1$. Then we set $w'=u_0\delta^2 w$. It follows that \eqref{Psi} is dominated by
\begin{align*}
\Psi\leq \mathcal{A}^{1/2}\:\mathcal{B}^{' 1/2},
\end{align*}
where
\begin{align*}
\mathcal{B}'=\sideset{}{^P}\sum_{w\sim W'}\sum_{c_2\sim C_2}\sideset{}{^\flat}\sum_{c_1\sim C_1}\: \Bigl|\sideset{}{^\flat}\sum_{u'\sim U'} \beta(u',w')\left(\frac{u'}{c_1}\right) \Bigr|^2.
\end{align*}
The conclusion of the previous lemma holds in this context as well. It follows that after detecting the square-free condition using the Mobius function, we can restrict the range of $\delta$ by $w\delta^2\ll p^{22\theta}/\ell^2$. So now we have a slight perturbation of the sum \eqref{ideal}, where we have the mild condition that $w\delta^2|4m-pn^2$. Our treatment above works even in this case, though our bound may not be as strong as before. In any case we are able to conclude that there exists an absolute constant $B_2>11$ (not depending on $\theta$ or $B$) such that 
\begin{align*}
\mathcal{O}_2(C,C^\dagger)\ll \frac{N^{1/2}p^{3/2+1/4+3(B+B_2)\theta/2+\varepsilon}Q}{\ell}.
\end{align*}
The proposition now follows if we take the absolute constant in the statement to be $B_1=2(3B_2+1)$.
\end{proof}

\bigskip


\section{Bounding $\mathcal{A}$: Final bound for $\mathcal{O}_2(C,C^\dagger)$}
\label{sec-final}

It remains to extend the bound in Proposition~\ref{prop-for-o2} to larger values of $C^\dagger$. In this section we will prove the following.\\

\begin{proposition}
\label{prop-for-o2-part2}
Suppose $A$ and $B$ are two positive numbers with $B>12$. Suppose $p^{1/2+B\theta}Q\ll C^\dagger\ll p^{1+A\theta}QD_1^{1/2}W^{-1/12}$. Then we have
\begin{align}
\label{seek-bound-full-dual-part2}
\mathcal{O}_2(C,C^\dagger)\ll \frac{N^{1/2}p^{3/2-\theta/2+\varepsilon}Q^2}{\ell},
\end{align} 
for any value of $C$ in the range \eqref{c-sum-length-initial} as long as we pick $Q$ in the range $$p^{7/10+6(11+A)\theta/5}<Q<p.$$
\end{proposition}

\bigskip

The extra saving now comes from getting cancellation in $\mathcal{A}$.
To get another bound for $\mathcal{A}$ we will open the absolute square and apply the Poisson summation formula on the sum over $c$. As a first step we glue back $c_1$ and $c_2$, and introduce a suitable bump function with compact support in $(0,\infty)$, to arrive at
\begin{align*}
\mathcal{A}\ll \sum_{w\sim W}\sum_{c\in \mathbb{Z}}\:W\left(\frac{c}{C^\dagger}\right)\Bigl|\sum_{q\in\mathcal{Q}}\;\sideset{}{^\dagger}\sum_{\psi\bmod{pq}}\:G_\psi\:\psi(\xi\bar{\zeta} c)\Bigr|^2.
\end{align*}
We first prove the following improved estimate.\\

\begin{lemma}
\label{bound-for-A}
We have
\begin{align*}
\mathcal{A}\ll p^\varepsilon W^{1/2}C^\dagger p^3Q^5\;\left(\frac{1}{Q}+\frac{p^{1/2}Q}{C^\dagger}\right).
\end{align*}
\end{lemma}

\begin{proof}
Opening the absolute square and pushing in the sum over $c$ we get
\begin{align*}
\sideset{}{^P}\sum_{w\sim W}\:\mathop{\sum\sum}_{q_1,q_2\in\mathcal{Q}}\;\mathop{\sideset{}{^\dagger}\sum\sideset{}{^\dagger}\sum}_{\psi_i\bmod{pq_i}}\:G_{\psi_1}\overline{G_{\psi_2}}\:\psi_1\bar{\psi}_2(\xi\bar{\zeta})\sum_c W\left(\frac{c}{C^\dagger}\right) \psi_1\bar{\psi}_2(c).
\end{align*}
After Poisson we arrive at the sum
\begin{align*}
\mathop{\sum\sum}_{q_1,q_2\in\mathcal{Q}}\;\frac{C^\dagger}{pq_1q_2}\;\sum_{c\in\mathbb{Z}}\;\mathfrak{C}\mathfrak{I}
\end{align*}
where the character sum is given by
\begin{align*}
\mathfrak{C}=\mathop{\sideset{}{^\dagger}\sum\sideset{}{^\dagger}\sum}_{\substack{\psi_1\bmod{pq_1}\\\psi_2\bmod{pq_2}}}\:
G_{\psi_1}\bar{G}_{\psi_2}\:\psi_1\bar{\psi}_2(\xi\bar{\zeta})\:\sum_{a\bmod{pq_1q_2}}\psi_1(a)\bar{\psi}_2(a)e\left(\frac{ac}{pq_1q_2}\right)
\end{align*}
and the integral $\mathfrak{I}$ is a Fourier transform of a bump function. So $\mathfrak{I}$ is negligibly small if $|c|\gg p^{1+\varepsilon}Q^2/C^\dagger$, and is bounded by $O(1)$ otherwise. \\

Now we shall investigate the character sum in detail. For the zero frequency $c=0$ we observe that the character sum vanishes unless $q_1=q_2$, $\psi_1=\psi_2$, in which case we have
\begin{align*}
\mathfrak{C}\ll p^4Q^5.
\end{align*}
The contribution of this term to $\mathcal{A}$ is dominated by 
\begin{align*}
W^{1/2}C^\dagger p^3Q^4.
\end{align*}
This gives a saving of $Q$ over the trivial bound, which is what is expected as the number of terms inside the absolute square is $Q$. Next consider the sum over $\psi_1$ which is given by
\begin{align*}
\mathop{\sideset{}{^\dagger}\sum}_{\substack{\psi_1\bmod{pq_1}}}\:
G_{\psi_1}\:\psi_1(a \xi\bar{\zeta}).
\end{align*}
Substituting the expression for $G_\psi$, and opening the Gauss sums we arrive at
\begin{align*}
\mathop{\sum\sum}_{b_1,b_2\bmod{pq_1}}\;\left(\frac{b_1}{p}\right)\left(\frac{b_2}{q_1}\right)\; e\left(\frac{b_1+b_2}{pq_1}\right)\;\mathop{\sideset{}{^\dagger}\sum}_{\substack{\psi_1\bmod{pq_1}}}\:\psi_1(a \xi\bar{\zeta}\overline{b_1b_2}).
\end{align*}
The last character sum roughly boils down to $\phi(pq_1)$ and yields a congruence relation
\begin{align*}
b_1b_2\zeta \equiv a\xi\bmod{pq_1}.
\end{align*}
Thus we arrive at
\begin{align*}
\phi(pq_1)\;\left(\frac{a\xi\zeta}{q_1}\right)\;\mathop{\sum}_{b_1\bmod{pq_1}}\;\left(\frac{b_1}{pq_1}\right)\; e\left(\frac{b_1+a\xi\bar{\zeta}\bar{b_1}}{pq_1}\right).
\end{align*}
Consequently it follows that
\begin{align*}
\mathfrak{C}=&\phi(pq_1)\phi(pq_2)\:\left(\frac{\xi\zeta}{q_1q_2}\right)\;\sum_{a\bmod{pq_1q_2}}\left(\frac{a}{q_1q_2}\right)e\left(\frac{ac}{pq_1q_2}\right)\\
&\times \mathop{\sum}_{b_1\bmod{pq_1}}\;\left(\frac{b_1}{pq_1}\right)\; e\left(\frac{b_1+a\xi\bar{\zeta}\bar{b_1}}{pq_1}\right)\mathop{\sum}_{b_2\bmod{pq_2}}\;\left(\frac{b_2}{pq_2}\right)\; e\left(\frac{b_2+a\xi\bar{\zeta}\bar{b_2}}{pq_2}\right)
\end{align*}
In the generic case $q_1\neq q_2$, the character sum splits as a product of three character sums
\begin{align*}
\mathfrak{C}=\phi(pq_1)\phi(pq_2)\:\left(\frac{\xi\zeta}{q_1q_2}\right)\;\mathfrak{C}_p\;\mathfrak{C}_{q_1}\;\mathfrak{C}_{q_2}
\end{align*}
where the character sums modulo $q_1$ and $q_2$ are similar. The character sum modulo $p$ is given by
\begin{align*}
\mathfrak{C}_p=&\mathop{\sum\sum\sum}_{a, b_i\bmod{p}}\;e\left(\frac{ac\overline{q_1q_2}}{p}\right)\\
\nonumber &\times \left(\frac{b_1b_2}{p}\right)\; e\left(\frac{b_1\bar{q}_1+a\xi\bar{\zeta}\bar{b_1}\bar{q}_1}{p}\right)\; e\left(\frac{b_2\bar{q}_2+a\xi\bar{\zeta}\bar{b_2}\bar{q}_2}{p}\right).
\end{align*}
The sum over $a$ now yields a congruence and we arrive at
\begin{align}
\label{weil-riemann}
\mathfrak{C}_p=&p\mathop{\sum}_{b_1\bmod{p}}\;\left(\frac{q_1\xi(c\zeta b_1+\xi q_2)}{p}\right)\; e\left(\frac{b_1\bar{q}_1-q_1\bar{q}_2^2\bar{\zeta}\xi\overline{(c\bar{q}_2+\xi\bar{\zeta}\bar{b}_1)}}{p}\right).
\end{align}
We get square root cancellation using Weil's bound for curves over finite fields. The character sum modulo $q_1$ is given by 
\begin{align*}
\mathfrak{C}_{q_1}=&\sum_{a\bmod{q_1}}\left(\frac{a}{q_1}\right)e\left(\frac{ac\overline{pq_2}}{q_1}\right)\; \mathop{\sum}_{b_1\bmod{q_1}}\;\left(\frac{b_1}{q_1}\right)\; e\left(\frac{b_1\bar{p}+a\xi\bar{\zeta}\bar{b_1}\bar{p}}{q_1}\right),
\end{align*}
which then reduces to
\begin{align*}
\mathfrak{C}_{q_1}=q_1^{1/2}\; \mathop{\sum}_{b_1\bmod{q_1}}\;\left(\frac{q_2p\zeta(cb_1\zeta+\xi q_2)}{q_1}\right)\; e\left(\frac{b_1\bar{p}}{q_1}\right).
\end{align*}
Again we have square-root cancellation in the remaining character sum, and hence this is bounded by $q_1$. 
So in general we expect
\begin{align*}
\mathfrak{C}\ll p^{7/2}Q^4,
\end{align*}
and this will imply that the contribution of the non-zero frequencies is bounded by
\begin{align*}
W^{1/2}\;\;p^{7/2+\varepsilon} Q^6.
\end{align*}
This yields a saving of $C^\dagger/\sqrt{p}Q$ over the trivial bound. The lemma follows.
\end{proof}

\bigskip

The above lemma yields improved estimates for the dual off-diagonal, at least for $C^\dagger\gg p^{1/2+\theta}Q$, which is the focus of this section. As an example we first show the following.\\

\begin{lemma}
\label{range-W-new}
We have \eqref{seek-bound-full-dual-part2} if 
$$
W\gg \max\left\{\frac{C^\dagger p^{\theta}}{\ell^3},\frac{p^{1/2+\theta}Q^2}{\ell^3}\right\}.
$$
\end{lemma}

\begin{proof}
We return to the bound for $\Psi$ given in Lemma~\ref{lem-for-psi}. We replace the trivial bound for $\mathcal{B}$, and plug in the improved bound from the previous lemma for $\mathcal{A}$. This gives
\begin{align*}
\mathcal{O}_2(C,C^\dagger)\ll \frac{N^{1/2}p^{3/2-\theta/2}Q^2}{\ell}\;\left(\frac{C^{\dagger}p^{\theta}}{\ell^3W}+\frac{Q^2p^{1/2+\theta}}{\ell^3W}\right)^{1/2}.
\end{align*}
The lemma follows.
\end{proof}

The above lemma shows that we only need to consider $W$ which are not too big. Our next lemma will show that the contribution of small $C$ for $C^\dagger \gg p^{1/2}Q$ is also satisfactory.\\

\begin{lemma}
\label{lem:range-for-c-for-c-dagger}
The bound in \eqref{seek-bound-full-dual-part2} holds 
if  $C^\dagger \gg p^{1/2}Q^2$ and 
\begin{align}
\label{one}
C\ll  \min\left\{\frac{p^{9-\theta}Q^8}{C^\dagger N^5N_\star^2J^2},\frac{W^{1/2}p^{7/2-\theta/2}Q^3\ell^{1/2}}{N^2N_\star^{1/2} J^2}\right\}, 
\end{align}
or if $C^\dagger \ll p^{1/2}Q^2$ and
\begin{align}
\label{two}
C\ll \frac{(WC^\dagger)^{1/2}p^{13/4-\theta/2}Q^{2}\ell^{1/2}}{N^2N_\star^{1/2} J^2}.
\end{align}
\end{lemma}

\begin{proof}
We return to the proof of Lemma~\ref{c-dagger-1}. In the derivation of the bound given in \eqref{the-expression-midway} we will now substitute the improved bound for $\mathcal{A}$. Now if $C^\dagger \gg p^{1/2}Q^2$, then we save $Q$ in $\mathcal{A}$ and hence the bound in \eqref{the-expression-midway} reduces to
\begin{align*}
\mathcal{O}_2(C,C^\dagger)\ll p^\varepsilon\frac{N^{3}N_\star}{\ell p^{3}Q^{2}}\;C^{1/2}\left(\frac{C^\dagger}{C_2^2}+\frac{D}{W}\right)^{1/2},
\end{align*}
which is dominated by
\begin{align*}
\frac{N^{1/2}p^{3/2-\theta/2}Q^2}{\ell}
\end{align*}
if
\begin{align*}
C\left(\frac{C^\dagger}{C_2^2}+\frac{D}{W}\right)\ll \frac{p^{9-\theta}Q^8}{N^5N_\star^2J^2},
\end{align*}
which holds under \eqref{one}.\\

 On the other hand if $C^\dagger \ll p^{1/2}Q^2$ then we save $C^\dagger/p^{1/2}Q$ in $\mathcal{A}$ and hence \eqref{the-expression-midway} reduces to
\begin{align*}
\mathcal{O}_2(C,C^\dagger)\ll p^\varepsilon\frac{N^{3}N_\star}{\ell p^{3}Q^{3/2}}\;C^{1/2}\left(p^{1/2}Q+\frac{Dp^{1/2}Q}{WC^\dagger}\right)^{1/2},
\end{align*}
which satisfies the desired bound if
\begin{align*}
C\left(p^{1/2}Q+\frac{Dp^{1/2}Q}{WC^\dagger}\right)\ll \frac{p^{9-\theta}Q^7}{N^5N_\star^2J^2},
\end{align*}
which in turn boils down to \eqref{two}.
\end{proof}

\bigskip

Recall that we only need to consider the case where $C^\dagger \gg p^{1/2+B\theta}Q$. We return to the expression given in \eqref{the-expression-midway-2}, where we now substitute the improved bound for $\mathcal{A}$ in place of the trivial bound. Indeed this transforms the bound in \eqref{the-expression-midway-2} to
\begin{align*}
\mathcal{O}_2(C,C^\dagger)\ll & p^{5\theta+\varepsilon}\frac{N^{9/2}N_\star^{3/2}}{p^{9/2}Q^{2}\ell}\;W^{1/4}\;\left(\frac{C^\dagger}{C_2^2}+\frac{D}{W}\right)^{1/2}\\
\nonumber &\times \left(\frac{N_\star^{1/4}}{(CJ^2\ell)^{1/4}}+\frac{N_\star^{1/2}}{(CJ^2Q)^{1/2}}\right)\:\left(\frac{1}{Q^{1/2}}+\frac{p^{1/4}Q^{1/2}}{C^{\dagger 1/2}}\right).
\end{align*}
Consider the term
\begin{align*}
& p^{5\theta+\varepsilon}\frac{N^{9/2}N_\star^{3/2}}{p^{9/2}Q^{2}\ell}\;\frac{D^{1/2}}{W^{1/4}}\: \left(\frac{N_\star^{1/4}}{(CJ^2\ell)^{1/4}}+\frac{N_\star^{1/2}}{(CJ^2Q)^{1/2}}\right)\:\left(\frac{1}{Q^{1/2}}+\frac{p^{1/4}Q^{1/2}}{C^{\dagger 1/2}}\right).
\end{align*}
The calculations given in the paragraph following \eqref{the-expression-midway-2} now yields that this term is bounded by
\begin{align*}
\frac{N^{1/2}p^{3/2+5\theta+\varepsilon}Q^2}{\ell}\:\left(\frac{1}{Q^{1/2}}+\frac{p^{1/4}Q^{1/2}}{C^{\dagger 1/2}}\right).
\end{align*}
This satisfies the bound in \eqref{seek-bound-full-dual-part2} if $Q\gg p^{12\theta}$ and $B> 12$.\\ 

To end the proof of the proposition we only need to establish the bound \eqref{seek-bound-full-dual-part2} for the remaining term, which is given by
\begin{align*}
\Theta:=& p^{5\theta+\varepsilon}\frac{N^{9/2}N_\star^{3/2}}{p^{9/2}Q^{2}\ell}\;W^{1/4}\:\frac{C^{\dagger 1/2}}{C_2}\\
\nonumber &\times \left(\frac{N_\star^{1/4}}{(CJ^2\ell)^{1/4}}+\frac{N_\star^{1/2}}{(CJ^2Q)^{1/2}}\right)\:\left(\frac{1}{Q^{1/2}}+\frac{p^{1/4}Q^{1/2}}{C^{\dagger 1/2}}\right).
\end{align*}
The estimation depends on whether $C^\dagger$ is larger or smaller compared to $p^{1/2}Q^2$. Accordingly we split our computation into two separate lemmas. First we deal with the case where $C^\dagger\ll p^{1/2}Q^2$.\\

\begin{lemma}
\label{small-c-dagger}
If $p^{1/2}Q\ll C^\dagger\ll p^{1/2}Q^2$, then  the bound \eqref{seek-bound-full-dual-part2} holds for $\Theta$ if $\theta<1/24$ and $Q>p^{1/2}$.
\end{lemma}

\begin{proof}
Indeed in this range of $C^\dagger$ we have
\begin{align}
\label{theta-terms}
\Theta\ll p^{5\theta+\varepsilon}\frac{N^{9/2}N_\star^{3/2}}{p^{17/4}Q^{3/2}\ell}\;W^{1/4}\: \left(\frac{N_\star^{1/4}}{(CJ^2\ell)^{1/4}}+\frac{N_\star^{1/2}}{(CJ^2Q)^{1/2}}\right).
\end{align}
Since we can take $C$ in the complementary range given in \eqref{two}, we get that the first term in the above expression is dominated by
\begin{align*}
p^{5\theta+\varepsilon}\frac{N^{9/2}N_\star^{3/2}}{p^{17/4}Q^{3/2}\ell}\;W^{1/4}\: \frac{N_\star^{1/4}N^{1/2}N_\star^{1/8}}{(WC^\dagger)^{1/8}p^{13/16-\theta/8}Q^{1/2}\ell^{3/8}}.
\end{align*}
Applying Lemma~\ref{range-W-new}, which gives an upper bound for $W$, and using the lower bound $C^\dagger>p^{1/2}Q$, we see that the above term is bounded by
\begin{align*}
p^{5\theta+\varepsilon}\frac{N^{9/2}N_\star^{3/2}}{p^{17/4}Q^{3/2}\ell}\: \frac{N_\star^{1/4}N^{1/2}N_\star^{1/8}}{p^{13/16-\theta/8}Q^{1/2}}\:\frac{p^{1/16+\theta/8}Q^{1/4}}{p^{1/16}Q^{1/8}}.
\end{align*}
Plugging in the upper bound for $N_\star$ it follows that this is dominated by
\begin{align*}
\frac{N^{1/2}p^{3/2-3/16+11\theta/2+\varepsilon}Q^{2-1/8}}{\ell},
\end{align*}
which satisfies the required bound \eqref{seek-bound-full-dual-part2} for say $\theta<1/24$ if $Q>p^{1/2}$.\\

In the complementary range of \eqref{two} the second term of \eqref{theta-terms} is bounded by
\begin{align*}
p^{5\theta+\varepsilon}\frac{N^{9/2}N_\star^{3/2}}{p^{17/4}Q^{3/2}\ell}\;W^{1/4}\: \frac{N_\star^{1/2}NN_\star^{1/4}}{(WC^\dagger)^{1/4}p^{13/8-\theta/4}QQ^{1/2}}.
\end{align*}
The $W$ term cancels out and using the fact that $C^\dagger>p^{1/2}Q$, we get
\begin{align*}
p^{5\theta+\varepsilon}\frac{N^{9/2}N_\star^{3/2}}{p^{17/4}Q^{3/2}\ell}\: \frac{N_\star^{3/4}N}{p^{1/8}Q^{1/4}p^{13/8-\theta/4}Q^{3/2}}.
\end{align*}
This is dominated by
\begin{align*}
\frac{N^{1/2}p^{3/2-1/4+11\theta/2+\varepsilon}Q^{2-3/4}}{\ell},
\end{align*}
which is smaller than the bound we obtained for the first term above. The lemma follows.
\end{proof}

\bigskip

\begin{lemma}
\label{big-c-dagger}
Suppose $p^{1/2}Q^2\ll C^\dagger \ll p^{1+A\theta}QD_1^{1/2}/W^{1/12}$. Then the bound \eqref{seek-bound-full-dual-part2} holds for $\Theta$ if $Q\gg p^{7/10+6(11+A)\theta/5}$.
\end{lemma}

\begin{proof}
For this range of $C^\dagger$ we have
\begin{align*}
\Theta\ll p^{5\theta+\varepsilon}\frac{N^{9/2}N_\star^{3/2}}{p^{9/2}Q^{5/2}\ell}\;W^{1/4}\:C^{\dagger 1/2}\: \left(\frac{N_\star^{1/4}}{(CJ^2\ell)^{1/4}}+\frac{N_\star^{1/2}}{(CJ^2Q)^{1/2}}\right).
\end{align*}
Recall that $C^\dagger \ll p^{1+A\theta}QD_1^{1/2}/W^{1/12}$, so that 
\begin{align}
\label{w-c-dagger}
W^{1/2}C^\dagger \ll p^{1+A\theta}QD^{1/2}W^{1/12}\ll \frac{C^{1/2}Q^2p^{2+A\theta}JW^{1/12}}{(N_\star N\ell)^{1/2}}.
\end{align}
It follows that
\begin{align*}
\Theta\ll p^{(10+A)\theta/2+\varepsilon}\frac{N^{17/4}N_\star^{5/4}W^{1/24}}{p^{7/2}Q^{3/2}\ell^{5/4}}\;\left(\frac{N_\star^{1/4}}{\ell^{1/4}}+\frac{N_\star^{1/2}}{(CJ^2)^{1/4}Q^{1/2}}\right).
\end{align*}
Hence the contribution of the first term is bounded by
\begin{align*}
p^{(10+A)\theta/2+\varepsilon}\frac{N^{17/4}N_\star^{3/2}W^{1/24}}{p^{7/2}Q^{3/2}\ell^{3/2}},
\end{align*}
where we plug in the trivial bound $W\ll D\ll p^{3+\varepsilon}Q^2/N^2\ell^2$. It follows that this term is bounded by
\begin{align*}
p^{(10+A)\theta/2+\varepsilon}\frac{N^{17/4-1/12}N_\star^{3/2}}{p^{7/2-1/8}Q^{3/2-1/12}\ell^{3/2}}\ll \frac{N^{1/2}p^{43/24+(10+A)\theta/2+\varepsilon}Q^{19/12}}{\ell}.
\end{align*}
This satisfies the bound in \eqref{seek-bound-full-dual-part2} if  $Q\gg p^{7/10+6(11+A)\theta/5}$.\\

It remains to analyse the second  term which is given by
\begin{align}
\label{snd-term}
p^{(10+A)\theta/2+\varepsilon}\frac{N^{17/4}N_\star^{5/4}W^{1/24}}{p^{7/2}Q^{3/2}\ell^{5/4}}\;\frac{N_\star^{1/2}}{(CJ^2)^{1/4}Q^{1/2}}.
\end{align}
Now, we only need to consider $C$ in the complementary range given in \eqref{one}. If 
\begin{align*}
C\gg \frac{p^{9-\theta}Q^8}{C^\dagger N^5N_\star^2J^2}\gg \frac{p^{17/2-\theta}Q^6}{ N^5N_\star^2J^2}
\end{align*}
then \eqref{snd-term} is dominated by
\begin{align*}
p^{(10+A)\theta/2+\varepsilon}\frac{N^{17/4}N_\star^{5/4}W^{1/24}}{p^{7/2}Q^{3/2}\ell^{5/4}}\;\frac{N^{5/4}N_\star}{p^{17/8-\theta/4}Q^{2}},
\end{align*}
where we plug in the trivial bound $W\ll D\ll p^{3+\varepsilon}Q^2/N^2\ell^2$. It follows that this term is bounded by
\begin{align*}
p^{(10+A)\theta/2+\varepsilon}\frac{N^{17/4}N_\star^{5/4}p^{1/8}Q^{1/12}}{p^{7/2}Q^{3/2}\ell^{5/4}N^{1/12}}\;\frac{N^{5/4}N_\star}{p^{17/8-\theta/4}Q^{2}}\ll \frac{N^{1/2}p^{5/3+(11+A)\theta/2+\varepsilon}Q^{2-11/12}}{\ell}.
\end{align*}
This satisfies the bound in \eqref{seek-bound-full-dual-part2} if  $Q\gg p^{2/11+6(12+A)\theta/11}$.\\

On the other hand if
\begin{align*}
C\gg \frac{W^{1/2}p^{7/2-\theta/2}Q^{3}\ell^{1/2}}{N^2N_\star^{1/2} J^2},
\end{align*}
then \eqref{snd-term} is dominated by
\begin{align*}
p^{(10+A)\theta/2+\varepsilon}\frac{N^{17/4}N_\star^{5/4}}{p^{7/2}Q^{3/2}\ell^{5/4}}\;\frac{N^{1/2}N_\star^{5/8}}{p^{7/8-\theta/8}Q^{5/4}}.
\end{align*}
It follows that this term is bounded by
\begin{align*}
\frac{N^{1/2}p^{7/4+(11+A)\theta/2+\varepsilon}Q}{\ell}.
\end{align*}
This satisfies the bound in \eqref{seek-bound-full-dual-part2} if  $Q\gg p^{1/4+(12+A)\theta/2}$. The lemma follows.
\end{proof}

\bigskip

From Lemmas~\ref{small-c-dagger} and \ref{big-c-dagger} we conclude that for $$p^{1/2}Q\ll C^\dagger \ll p^{1+A\theta}QD_1^{1/2}/W^{1/12}$$ the bound \eqref{seek-bound-full-dual-part2} holds for $\Theta$ if $Q\gg p^{7/10+6(11+A)\theta/5}$ and $\theta<1/24$. This together with the observation we made preceding Lemma~\ref{small-c-dagger}, we conclude that the bound \eqref{seek-bound-full-dual-part2} holds for the dual sum $\mathcal{O}_2(C,C^\dagger)$ if 
$$p^{1/2+B\theta}Q\ll C^\dagger \ll p^{1+A\theta}QD_1^{1/2}/W^{1/12}$$
with $B>12$ and 
$$p^{7/10+6(11+A)\theta/5}< Q<p.$$
This concludes the proof of  Proposition~\ref{prop-for-o2-part2}.

\bigskip

We will now conclude the proof of Theorem~\ref{mthm}. In Proposition~\ref{prop-for-o1} and Proposition~\ref{prop-for-o2} we get two computable absolute constants, $A$ and $B_1$ respectively. For Proposition~\ref{prop-for-o2-part2} we pick $B=12+\varepsilon$. Then we are forced to take $\theta<1/(B_1+72)$, (suppose $B_1>32$) and we need to pick $Q$ satisfying 
$$
p^{7/10+6(11+A)\theta/5}< Q<p^{1-170\theta}.
$$
So the optimal choice for $\theta$ is obtained by equating the two bounds, which yields $\theta<3/4(3A+458)$. Hence our theorem holds with
$$
\delta=\theta/2=\min\left\{\frac{1}{2(B_1+72)},\frac{3}{8(3A+458)}\right\}.
$$

\bigskip

\section{A shifted convolution sum problem}
\label{sec-shifted}

We return to the definition of the sum $\Omega$ given in \eqref{Omega-def}. Extending the range of summation of the $d$ sum we get that 
\begin{align*}
\Omega\leq \mathcal{S}:=\sum_{d\in \mathbb{Z}}\left|\sum_{n\in\mathbb{Z}}\;\lambda_f(d+pn^2)\:F(d,n)\right|^2,
\end{align*}
where recall that $F(d,n)$ is defined in \eqref{wt-shift-conv-sum}. For the convenience of the reader we recall the definition
\begin{align*}
F(d,n)=&W\left(\frac{n}{\mathcal{N}}\right)\; V\left(\frac{(d+pn^2)\ell^2}{M}\right)\\
\nonumber &\times \int_\mathbb{R}\:V(y)e\left(\frac{N_\star\sqrt{d+pn^2}y}{CQ\sqrt{p}J^2}-\frac{N_\star ny}{CQJ^2}\right)\mathrm{d}y.
\end{align*}
Also we shall recall the sizes of the parameters
\begin{align*}
p^{1-\theta}&<N<p^{1+\varepsilon},\;\;\;p^{1/2+\varepsilon}<Q,\;\;\;
\mathcal{N}=p^\varepsilon\frac{pQ}{N\ell}\\
\frac{p^{3-\varepsilon}Q^2}{N^2}&\ll M\ll \frac{p^{3+\varepsilon}Q^2}{N^2},\;\;\; N_\star \ll \frac{p^{2+\varepsilon}Q^2}{N},\;\;\;
C \ll \frac{N_\star p}{NJ^2\ell}.
\end{align*}
Opening the absolute square and interchanging
the order of summations we arrive at 
\begin{align}
\label{the-sum-s}
\mathcal{S}=\mathop{\sum\sum}_{\substack{n,r\in \mathbb{Z}}}\sum_{d\in\mathbb{Z}}\lambda_f(d+pn^2)\lambda_f(d+pr^2)F(d,n)\overline{F(d,r)}.
\end{align}
The main aim of this section is to prove the following result. \\

\begin{proposition}
\label{shifted-prop}
Suppose $Q<p$, then we have
\begin{align*}
\mathcal{S}\ll \;\frac{p^{1+10\theta+\varepsilon}Q^3}{\ell^3}\:\left(\frac{N_\star^{1/2}}{(CJ^2\ell)^{1/2}}+\frac{N_\star}{CJ^2Q}\right).
\end{align*}
\end{proposition}

\bigskip
\begin{remark}
\label{shifted-p>Q}
A non-trivial bound can also be obtained under the weaker condition $Q<p^2$. In this case we have some additional terms, for example in \eqref{p>Q} we get the term $p^{7\theta+\varepsilon}Q^4/\ell^4$. Roughly speaking, without the condition that $Q<p$, the total saving in $\mathcal{S}$ (or $\Omega$) is $\min\{Q,p\}$. In hindsight it is more natural to impose the restriction $Q<p$ right from the start, as it simplifies certain technicalities.
\end{remark}

\bigskip

Before we embark to prove the proposition let us highlight the main output of it. Given the fact that the $d$ sum in $\mathcal{S}$ effectively ranges upto $|d|\ll D$, where $D$ is as given in \eqref{l-def}, we see that the trivial bound is given by 
\begin{align*}
\mathcal{S}\ll p^\varepsilon D\mathcal{N}^2\ll \frac{p^{1+3\theta+\varepsilon}Q^4}{\ell^3}\:\frac{CJ^2}{N_\star}.
\end{align*}
Since $C$ ranges upto $N_\star$, roughly speaking the bound is of the order $pQ^4$. In contrast the proposition gives a bound which is roughly of the size $pQ^3$. So we have saved $Q$, which is the maximum possible given that the number of harmonics inside the absolute value in the definition of $\mathcal{S}$ is $Q$. However for smaller $C$ the bound in the proposition can be worse than the trivial bound. Indeed we see that the proposition produces a non-trivial bound only in the range
\begin{align*}
C\gg p^{5\theta}\:\max\left\{\frac{N_\star}{Q^{2/3}J^2\ell^{1/3}},\frac{N_\star}{QJ^2}\right\}.
\end{align*}
(Here as always we will be sometimes little wasteful when it comes to the coefficient of $\theta$ in the power of $p$.)\\

Now we proceed to prove the proposition.
We begin by realizing the sum as a shifted convolution sum.
Consider the inner sum over $d$. Changing the
variable of summation and writing $u=r^2-n^2$ we arrive at the sum
\begin{align*}
\Sigma=\sum_{m=1}^\infty\lambda_f(m)\lambda_f(m+pu)F(m-pn^2,n)\overline{F(m-pn^2,r)}.
\end{align*}
We will now employ the circle method to study
this sum. Set $v=m+pu$. We rewrite this equation as a congruence
$v\equiv m\bmod{p}$ and an integral equation of smaller size
$(v-m-pu)/p=0$. If we retain the divisibility condition, then the
last equation can be detected using the delta method with modulus
ranging up to $Q/\ell$. The letter $q$ will be used in this section (and only in this section) to denote the modulus coming from the circle method. Hopefully it will not create any confusion.\\

We now briefly recall the $\delta$ method of Duke, Friedlander, Iwaniec  \cite{DFI-1} and Heath-Brown \cite{H}. The starting point is a smooth approximation of the $\delta$-symbol. We will follow the exposition of Heath-Brown in \cite{H}. \\

\begin{lemma}\label{delta-symbol}
For any $Q^\star>1$ there is a positive constant $c_0$, and a smooth function $h(x,y)$ defined on $(0,\infty)\times\mathbb R$, such that
\begin{equation}
\delta(n,0)=\frac{c_0}{Q^{\star 2}}\sum_{q=1}^{\infty}\;\sideset{}{^{\star}}\sum_{a \bmod{q}}e\left(\frac{an}{q}\right)\;h\left(\frac{q}{Q^\star},\frac{n} {Q^{\star 2}}\right).\label{cm}
\end{equation}
The constant $c_0$ satisfies $c_0=1+O_A(Q^{\star -A})$ for any $A>0$. Moreover $h(x,y)\ll x^{-1}$ for all $y$, and $h(x,y)$ is non-zero only for $x\leq\max\{1,2|y|\}$. 
\end{lemma} 

\bigskip

In practice, to detect the equation $n=0$ for a sequence of integers in the range $[-X,X]$, it is logical to choose $Q^\star=X^{1/2}$, so that in the generic range for $q$ there is no oscillation in the weight function $h$.  The smooth function $h(x,y)$ (see \cite{HB}) is defined as 
\begin{align*}
h(x,y)=\sum_{j=1}^\infty \frac{1}{jx}\left\{w(jx)-w\left(\frac{|y|}{jx}\right)\right\}
\end{align*}
where $w$ is smooth `bump function' supported in $(1/2,1)$, with $0\leq w(x)\ll 1$ and $\int w=1$. It follows that the function $h$ satisfies
\begin{eqnarray}
x^{i} \frac{\partial^i}{\partial x^i}h(x,y)\ll_i x^{-1} & \textnormal{and} & \frac{\partial}{\partial y}h(x,y)=0 \label{hbound1}
\end{eqnarray}
for $x\leq 1$ and $|y|\leq x/2$. Also for $|y|>x/2$, we have
\begin{equation}
x^i y^j \frac{\partial^{i}}{\partial x^i}\frac{\partial^{j}}{\partial y^j}h(x,y)\ll_{i,j} x^{-1}. \label{hbound2}
\end{equation}
Furthermore, for $x$ small we have the stronger estimate
\begin{equation}
x^i y^j \frac{\partial^{i}}{\partial x^i}\frac{\partial^{j}}{\partial y^j}h(x,y)\ll_{i,j} x^{-1}\left(x^N+\min\{1,(x/|y|)^N\}\right). \label{hbound3}
\end{equation}
for every positive integer triplets $(i,j,N)$. the implied constant depends on these parameters. The main implication of the last inequality is the fact that $h(x,y)$ is negligibly small if $x$ is small and $|y|$ is much larger than $x$. In particular, consider the integral 
\begin{align*}
H=\int_{Y_1}^{Y_2} |h(x,y)|dy.
\end{align*}
The above inequality implies that
\begin{align*}
H&\ll x^{N-1}(Y_2-Y_1)+x^{-1}\:\int_{Y_1}^{Y_2}\min\{1,(x/|y|)^N\} dy\\
&\ll x^{\eta N-1}(Y_2-Y_1)+x^{-1}\:\int_{|y|\leq x^{1-\eta}}\:dy.
\end{align*}
So if $0<x<p^{-c}$ and $|Y_i|<p^A$ for some $c, A>0$, then we conclude that $H\ll_\varepsilon p^{\varepsilon}$ for any $\varepsilon>0$.\\

We now apply the above lemma to detect the event $(v-m-pu)/p=0$. Since the integers in the sequence is bounded by $\mathcal{N}^2$. The optimum choice of $Q^\star$ is $\mathcal{N}$, which is of size $Q/\ell$ when $N=p$. As such we set $Q^\star=Q/\ell$. With this we arrive at the expression 
\begin{align*}
\frac{\ell^2}{Q^2}\sum_{q\ll
\mathcal{N}p^\theta}\;\;&\sideset{}{^\star}\sum_{a\bmod{q}}\;\mathop{\sum\sum}_{\substack{m,v=1\\m\equiv v\bmod{p}}}^\infty\lambda_f(m)\lambda_f(v)\;e\left(\frac{a(v-m-pu)}{pq}\right)\\
&\times F(m-pn^2,n)\overline{F(v-pr^2,r)}h\left(\frac{q\ell}{Q},\frac{(v-m-pu)\ell^2}{pQ^2}\right).
\end{align*}
Note that from the Lemma~\ref{delta-symbol} we get that $q$ ranges upto 
$$q\leq \frac{Q}{\ell}+2\frac{|v-m-pu|\ell}{pQ}\ll \frac{p\mathcal{N}}{N}\ll \mathcal{N}p^\theta.$$
Since we take $Q<p^{1-2\theta-\varepsilon}$, we have $\mathcal{N}p^\theta\ll p^{1-\varepsilon}$ and hence $(p,q)=1$. So we can replace $a$ by $ap$.
Finally we detect the congruence condition using additive characters to get
\begin{align*}
\frac{\ell^2}{pQ^2}\sum_{q\ll
\mathcal{N}p^\theta}\;\;&\sideset{}{^\star}\sum_{a\bmod{q}}\;\sum_{b\bmod{p}}\;\mathop{\sum\sum}_{\substack{m,v=1}}^\infty\lambda_f(m)\lambda_f(v)\;e\left(\frac{(ap+bq)(v-m-pu)}{pq}\right)\\
&\times F(m-pn^2,n)\overline{F(v-pr^2,r)}\;h\left(\frac{q\ell}{Q},\frac{(v-m-pu)\ell^2}{pQ^2}\right).
\end{align*}
The case of $b=0$ is atypical, and we will first deal with it. In this case the sum reduces to
\begin{align}
\label{b=0}
\Sigma_0:=\frac{\ell^2}{pQ^2}&\sum_{q\ll
\mathcal{N}p^\theta}\;\sideset{}{^\star}\sum_{a\bmod{q}}\;\mathop{\sum\sum}_{\substack{m,v=1}}^\infty\lambda_f(m)\lambda_f(v)\;e\left(\frac{a(v-m-pu)}{q}\right)\\
\nonumber &\times F(m-pn^2,n)\overline{F(v-pr^2,r)}\;h\left(\frac{q\ell}{Q},\frac{(v-m-pu)\ell^2}{pQ^2}\right).
\end{align}
Our next lemma provides a sufficient bound for this sum.\\

\begin{lemma}
We have
\begin{align*}
\Sigma_0\ll\;\frac{p^{8\theta+\varepsilon}Q^{3/2}N_\star}{CJ^2\ell^{5/2}}.
\end{align*}
\end{lemma}

\begin{proof}
We will the Voronoi summation formula on the sum over $m$ and $v$.
First consider the sum over $m$ which is given by
\begin{align*}
\mathop{\sum}_{\substack{m=1}}^\infty\lambda_f(m)\;&e\left(-\frac{am}{q}\right)\;e\left(\frac{N_\star\sqrt{m}y}{CQ\sqrt{p}J^2}\right)\\
&\times  V\left(\frac{m\ell^2}{M}\right)\;h\left(\frac{q\ell}{Q},\frac{(v-m-pu)\ell^2}{pQ^2}\right).
\end{align*}
Applying Voronoi summation we get
\begin{align}
\label{vv}
\frac{M}{\ell^2 \sqrt{p}q}\mathop{\sum}_{\substack{m=1}}^\infty &\lambda_f(m)\;e\left(\frac{\bar{p}\bar{a}m}{q}\right)\;\int e\left(\frac{N_\star\sqrt{Mz}\:y}{CQ\sqrt{p}J^2\ell}\right)\\
\nonumber &\times  V\left(z\right)h\left(\frac{q\ell}{Q},-z+\frac{(v-pu)\ell^2}{pQ^2}\right)J_{\kappa-1}\left(\frac{4\pi\sqrt{Mmz}}{\sqrt{p} q\ell}\right)\mathrm{d}z.
\end{align}
Extracting the oscillations from  the Bessel and integrating by parts it follows that the $z$ integral is negligibly small unless
\begin{align*}
\left|\frac{N_\star\sqrt{M}y}{CQ\sqrt{p}J^2\ell}\pm \frac{2\sqrt{Mm}}{\sqrt{p} q\ell}\right|\ll p^\varepsilon\left\{1+ \frac{Q}{q\ell}\right\}.
\end{align*}
We can take it as a restriction on the number of contributing $m$. Indeed it follows that
\begin{align*}
\left|\frac{N_\star q y}{CQJ^2}\pm 2\sqrt{m}\right|\ll \frac{p^{1+\varepsilon}}{N}.
\end{align*}
Accordingly we set 
\begin{align*}
\Xi_y^\pm=\left[\left(\pm\frac{N_\star qy}{2CQJ^2}-C(\varepsilon)\frac{p^{1+\varepsilon}}{N}\right)^2,\:\left(\pm\frac{N_\star qy}{2CQJ^2}+C(\varepsilon)\frac{p^{1+\varepsilon}}{N}\right)^2\right]
\end{align*}
for some constant $C(\varepsilon)$ depending only on $\varepsilon$, and then set 
\begin{align*}
\Xi_y=(\Xi_y^+\cup \Xi_y^-)\cap [1,\infty).
\end{align*}
Observe that among the variables of summation, $\Xi_y$ only depends on $q$. The above restriction on $m$ can now be written as $m\in \Phi_y$. It follows that  the number of $m$ contributing is 
$$p^\varepsilon\left(1+\frac{N_\star qp}{CQJ^2N}\right)$$ 
(for any fixed $y$). 
Also the Bessel function is bounded by
\begin{align*}
\frac{p^{1/4}(q\ell)^{1/2}}{(Mm)^{1/4}}\ll \left(\frac{N\ell}{p}\right)^{1/2}\min\left\{\frac{q^{1/2}}{Q^{1/2}},\frac{(CJ^2)^{1/2}}{N_\star^{1/2}}\right\}.
\end{align*}
The Voronoi summation on the sum over $v$ acts the same way, and we end up getting a similar restriction on the number of contributing frequencies $v$ in  the dual sum.\\

Moreover the sum over $a$ yields the Kloosterman sum $S(u,v-m;q)$, for which we have the Weil bound $q^{1/2}(u,v-m,q)^{1/2}$.  Indeed after two applications of the Voronoi summation formula, the expression in \eqref{b=0} reduces to
\begin{align*}
\Sigma_0\ll & p^\varepsilon\frac{M^2}{(pQ\ell)^2}\sum_{q\ll
\mathcal{N}p^\theta}\;\frac{1}{q^2}\\
&\times \iint \sum_{m\in \Xi_{y_1}}\;\mathop{\sum}_{\substack{v\in \Xi_{y_2}}}\: |S(u,v-m;q)|\:\times \:(\text{integral})\:\mathrm{d}y_1\mathrm{d}y_2.
\end{align*}
The integral here is a two dimensional analogue of the integral above, with two Bessel functions.
Taking into account the size of the Bessel function and using the observation regarding the integral $H$ (as given above) to treat the $z$ integral, we get the following bound 
\begin{align}
\label{b=0-2}
p^\varepsilon\frac{p^4Q^2}{N^4\ell^2}&\sum_{q\ll
\mathcal{N}p^\theta}\;\frac{1}{q^{3/2}}\;\left(1+\frac{(N_\star qp)^2}{(CQJ^2N)^2}\right)\;\frac{N\ell}{p}\min\left\{\frac{q}{Q},\frac{CJ^2}{N_\star}\right\}.
\end{align}
This is then bounded by
\begin{align*}
&p^\varepsilon\frac{p^{3+2\theta}Q^2}{N^3\ell}\frac{N_\star}{CJ^2}\sum_{q\ll
\mathcal{N}p^\theta}\;\frac{1}{q^{3/2}}\;\left(1+\frac{N_\star q}{CQJ^2}\right)\;\max\left\{\frac{q}{Q},\frac{CJ^2}{N_\star}\right\}\min\left\{\frac{q}{Q},\frac{CJ^2}{N_\star}\right\}\\
&\ll p^\varepsilon\frac{p^{3+2\theta}Q}{N^3\ell}\:\sum_{q\ll
\mathcal{N}p^\theta}\;\frac{1}{q^{1/2}}\;\left(1+\frac{N_\star q}{CQJ^2}\right)\ll \frac{p^{6\theta+\varepsilon}Q^{3/2}}{\ell^{3/2}}+\frac{p^{8\theta+\varepsilon}Q^{3/2}N_\star}{CJ^2\ell^{5/2}}.
\end{align*}
Using \eqref{c-sum-length-initial} we see that the second term dominates the first and hence the lemma follows.
\end{proof}

\bigskip

In the generic case $b\neq 0$ the expression reduces to
\begin{align*}
\Sigma_1:=\frac{\ell^2}{pQ^2}\sum_{q\ll
\mathcal{N}p^\theta}\;\;&\sideset{}{^\star}\sum_{a\bmod{pq}}\;\mathop{\sum\sum}_{\substack{m,v=1}}^\infty\lambda_f(m)\lambda_f(v)\;
e\left(\frac{a(v-m-pu)}{pq}\right)\\
&\times F(m-pn^2,n)\overline{F(v-pr^2,r)}h\left(\frac{q\ell}{Q},\frac{(v-m-pu)\ell^2}{pQ^2}\right).
\end{align*}
Our next lemma gives a non-trivial bound for this expression. One will see that compared to the special case $b=0$ we have lost a $p$, but overall we have saved $Q^{1/2}$ compared with the trivial bound. This is not enough for proving the proposition, but it is the first step.\\

\begin{lemma}
We have
\begin{align*}
\Sigma_1\ll \;\frac{N_\star p^{1+8\theta+\varepsilon} Q^{3/2}}{CJ^2\ell^{5/2}}.
\end{align*}
\end{lemma}

\begin{proof}
We now apply the Voronoi summation to the sum over
$m$ and $v$. Applying the Voronoi summation to the $m$ sum
\begin{align*}
\mathop{\sum}_{\substack{m=1}}^\infty\lambda_f(m)\;&e\left(-\frac{am}{pq}\right)\;e\left(\frac{N_\star\sqrt{m}y}{CQ\sqrt{p}J^2}\right)\\
&\times  W\left(\frac{m\ell^2}{M}\right)h\left(\frac{q\ell}{Q},\frac{(v-m-pu)\ell^2}{pQ^2}\right),
\end{align*}
we get
\begin{align*}
\frac{M}{\ell^2 pq}\mathop{\sum}_{\substack{m=1}}^\infty &\lambda_f(m)\;e\left(\frac{\bar{a}m}{pq}\right)\;\int e\left(\frac{N_\star\sqrt{Mz}y}{CQ\sqrt{p}J^2\ell}\right)\\
&\times  W\left(z\right)h\left(\frac{q\ell}{Q},-z+\frac{(v-pu)\ell^2}{pQ^2}\right)J_{\kappa-1}\left(\frac{4\pi\sqrt{Mmz}}{pq\ell}\right)\mathrm{d}z.
\end{align*} 
Again extracting the oscillations from  the Bessel and integrating by parts, it follows that the $z$ integral is negligibly small if
\begin{align*}
\left|\frac{N_\star\sqrt{M}y}{CQ\sqrt{p}J^2\ell}\pm \frac{2\sqrt{Mm}}{pq\ell}\right|\gg p^\varepsilon\left\{1+ \frac{Q}{q\ell}\right\}.
\end{align*}
So we only need to consider those $m$ which satisfy the inequality
\begin{align}
\label{restriction-on-m}
\left|\frac{N_\star \sqrt{p} qy}{CQJ^2}\pm 2\sqrt{m}\right|\ll \frac{p^{3/2+\varepsilon}}{N}.
\end{align}
Accordingly we set 
\begin{align*}
\Phi_y^\pm=\left[\left(\pm\frac{N_\star \sqrt{p} qy}{2CQJ^2}-C(\varepsilon)\frac{p^{3/2+\varepsilon}}{N}\right)^2,\:\left(\pm\frac{N_\star \sqrt{p} qy}{2CQJ^2}+C(\varepsilon)\frac{p^{3/2+\varepsilon}}{N}\right)^2\right]
\end{align*}
for some constant $C(\varepsilon)$ depending only on $\varepsilon$, and then set 
\begin{align*}
\Phi_y=(\Phi_y^+\cup \Phi_y^-)\cap [1,\infty).
\end{align*}
Observe that among the variables of summation, $\Phi_y$ only depends on $q$. The above restriction on $m$ can now be written as $m\in \Phi_y$. It follows that  the number of $m$ contributing is 
$$p^\varepsilon\left(1+\frac{N_\star p^2q}{CQJ^2N}\right).$$ 
The Voronoi summation on the sum over $v$ acts the same way, and as in the proof of the previous lemma we arrive at the expression
\begin{align*}
\Sigma_1\ll & p^\varepsilon\frac{M^2}{p^3(Q\ell)^2}\sum_{q\ll
\mathcal{N}p^\theta}\;\frac{1}{q^2}\\
&\times \iint \sum_{m\in \mathcal{M}_{y_1}}\;\mathop{\sum}_{\substack{v\in V_{y_2}}}\: |S(pu,v-m;pq)|\:\times \:(\text{integral})\:\mathrm{d}y_1\mathrm{d}y_2.
\end{align*}
The character sum $S(pu,v-m;pq)$ is a Ramanujan sum modulo $p$ and a Kloosterman sum modulo $q$, and so from Weil we conclude that 
\begin{align*}
S(pu,v-m;pq)\ll q^{1/2}(u,v-m,q)^{1/2}(v-m,p).
\end{align*}
We use the trivial bound $J_u(x)\ll \min\{1,x^{-1/2}\}$ for the Bessel function and use the on average bound for the size of the function $h$ to get the following bound for the off-diagonal contribution, i.e. for $p\nmid m-v$,
\begin{align*}
p^\varepsilon\frac{p^3Q^2}{N^4\ell^2}&\sum_{q\ll
\mathcal{N}p^\theta}\;\frac{1}{q^{3/2}}\;\left(1+\frac{(N_\star qp^2)^2}{(CQJ^2N)^2}\right)\;\min\left\{1,\frac{CQp^{1/2}J^2\ell}{N_\star M^{1/2}}\right\}
\end{align*}
This is then bounded by
\begin{align*}
p^\varepsilon\frac{p^3Q^2}{N^4\ell^2}+p^\varepsilon\frac{p^{6}}{N^5\ell}\frac{N_\star}{CJ^2}&\sum_{q\ll
\mathcal{N}p^\theta}\;q^{1/2}\ll \frac{p^{1+8\theta+\varepsilon}N_\star Q^{3/2}}{CJ^2\ell^{5/2}}.
\end{align*}
The last inequality follows by showing that the second term in the sum is dominating as we take $Q<p^2$. Now consider the diagonal contribution where $p|m-v$. In this case we get the bound 
\begin{align*}
p^\varepsilon\frac{p^4Q^2}{N^4\ell^2}&\sum_{q\ll
\mathcal{N}p^\theta}\;\frac{1}{q^{3/2}}\;\left(1+\frac{N_\star qp^2}{CQJ^2N}\right)
\end{align*}
where we use the bound $O(1)$ for the Bessel function. This is then dominated by
\begin{align*}
p^\varepsilon\frac{p^4Q^2}{N^4\ell^2}+p^\varepsilon\frac{p^6Q}{N^5\ell^2}\frac{N_\star }{CJ^2}&\sum_{q\ll
\mathcal{N}p^\theta}\;\frac{1}{q^{1/2}}\ll \frac{p^{1+6\theta+\varepsilon}N_\star Q^{3/2}}{CJ^2\ell^{5/2}}.
\end{align*}
Here again the second term dominates the first term as $Q<p^2$.
 The lemma follows.
\end{proof}

So in the generic case $b\neq 0$, Voronoi summation gives a saving of $p$ less compared to the degenerate case $b=0$. Roughly speaking (when the parameters are in generic ranges) this means that we have saved $Q^{1/2}$ in $\Omega$. This is about square root of the number of terms inside the absolute value. 
To save more
we can use the sums over $n$ and $r$. Consider the sum over $n$ given by
\begin{align*}
\Sigma_2:=\sum_{n\in\mathbb{Z}}W\left(\frac{n}{\mathcal{N}}\right)e\left(\frac{an^2}{q}\right)e\left(-\frac{N_\star ny}{CQJ^2}\right)h\left(\frac{q\ell}{Q},z_1-z_2+\frac{(n^2-r^2)\ell^2}{Q^2}\right).
\end{align*}
Here $W$ is a smooth bump function with support $[-1,1]$. The next lemma indicates that from Poisson we can expect to save $\sqrt{q}p^{-\theta}$ (even in the case $b=0$).\\

\begin{lemma}
We have
\begin{align*}
\Sigma_2\ll p^{\theta+\varepsilon}\:\frac{\mathcal{N}}{\sqrt{q}}\;\;\frac{Q}{q\ell}.
\end{align*}
\end{lemma}

\begin{proof}
Applying the Poisson summation
on the sum $\Sigma_2$ with modulus $q$, we get
\begin{align*}
&\frac{\mathcal{N}}{q}\sum_{n\in\mathbb{Z}}\sum_{\beta\bmod{q}}e\left(\frac{a\beta^2+n\beta}{q}\right)\\
&\times \int W\left(x\right)e\left(-\frac{N_\star \mathcal{N}xy}{CQJ^2}-\frac{\mathcal{N}nx}{q}\right)h\left(\frac{q\ell}{Q},z_1-z_2+\frac{(\mathcal{N}^2x^2-r^2)\ell^2}{Q^2}\right)\mathrm{d}x.
\end{align*}
The character sum can be expressed in terms of the Gauss sum. By repeated integration by parts we see that the integral is negligibly small unless
\begin{align*}
\left|\frac{N_\star qy}{CQJ^2}+n\right|\ll p^\varepsilon \frac{p}{N}.
\end{align*}
We set
\begin{align*}
\Psi_y=\left[-\frac{N_\star qy}{CQJ^2}-C(\varepsilon)\frac{p^{1+\varepsilon}}{N},\:-\frac{N_\star qy}{CQJ^2}+C(\varepsilon)\frac{p^{1+\varepsilon}}{N}\right],
\end{align*}
so that the above restriction on $n$ can be written as $n\in \Psi_y$.
Hence the number of $n$ (for any given $y$) contributing to the sum is $O(p^{\theta+\varepsilon})$.
This shows that the last expression is bounded by
\begin{align*}
p^{\theta+\varepsilon}\:\frac{\mathcal{N}}{\sqrt{q}}\;\frac{Q}{q\ell}.
\end{align*}
The last factor reflects the size of the $h$ function.
\end{proof}

\bigskip


\begin{proof}[Proof of proposition]
Let us summarize the above analysis as follows. This eventually proves the proposition. The application of the modified delta method, followed by two applications of the Voronoi summation formula and one application of the Poisson summation transforms the sum given in \eqref{the-sum-s} to
\begin{align}
\label{the-leading-exp}
\sum_{|r|\ll \mathcal{N}}\;\frac{\ell^2}{pQ^2}\sum_{q\ll
Q/\ell}\;\frac{\mathcal{N}}{q}\;\left(\frac{M}{\ell^2 pq}\right)^2&\;\mathop{\sum\sum}_{\substack{m,v=1}}^\infty\lambda_f(m)\lambda_f(v)\;
\sum_{n\in\mathbb{Z}}\;\mathfrak{C}\;\mathfrak{I}\\
\nonumber &+O\left(\frac{p^{10\theta+\varepsilon}Q^{3}}{\ell^{3}}\:\frac{N_\star}{CJ^2\ell}\right),
\end{align}
where the character sum is given by
\begin{align*}
\mathfrak{C}=\sideset{}{^\star}\sum_{a\bmod{pq}}e\left(-\frac{\bar{a}(v-m)}{pq}\right)e\left(\frac{-ar^2}{q}\right)\sum_{\beta\bmod{q}}e\left(\frac{a\beta^2+n\beta}{q}\right)
\end{align*}
and the integral is given by
\begin{align*}
\mathfrak{I}=&\iiint F\left(\frac{Mz_1}{\ell^2}-p\mathcal{N}^2x^2,\mathcal{N}x\right)\overline{F\left(\frac{Mz_2}{\ell^2}-pr^2,r\right)}\\
&\times e\left(-\frac{\mathcal{N}nx}{q}\right)J_{\kappa-1}\left(\frac{4\pi\sqrt{Mmz_1}}{pq\ell}\right)J_{\kappa-1}\left(\frac{4\pi\sqrt{Mvz_2}}{pq\ell}\right)\\
&\times h\left(\frac{q\ell}{Q},z_1-z_2+\frac{(\mathcal{N}^2x^2-r^2)\ell^2}{Q^2}\right)\mathrm{d}x\mathrm{d}z_1\mathrm{d}z_2.
\end{align*}
Note that the error term takes into account the contribution of the term arising from $b=0$. 
The remaining character sum is a product of a Ramanujan sum modulo $p$, which is bounded by $O((p,v-m))$, and a Salie type sum modulo $q$ (which for $q\equiv 1\bmod{4}$ odd is simply given by)
\begin{align*}
\mathfrak{C}_q(r,4(v-m)+pn^2)=\sqrt{q}\:\sum_{a\bmod{q}}\:\left(\frac{a}{q}\right)\:e\left(\frac{\bar{a}\bar{p}(v-m)+\bar{4}\bar{a}n^2+ar^2}{q}\right),
\end{align*}
 which is bounded by $O(q(r^2,q)^{1/2})$. (For even $q$ we have a similar expression and the bound holds.) From the integral we extract the restrictions on the sums over $(m,v,n)$ by opening the integrals present in $F$ and taking them outside. With this the leading expression in \eqref{the-leading-exp} reduces to
\begin{align*}
& \frac{p^4Q^3}{N^5\ell^3}\sum_{|r|\ll \mathcal{N}}\:\sum_{q\ll
\mathcal{N}p^\theta}\;\frac{1}{q^3}\\
&\times \iint \sum_{n\in \Psi_{y_1}}\:\sum_{m\in \Phi_{y_1}}\;\mathop{\sum}_{\substack{v\in \Phi_{y_2}}}\: (\text{char sum})\:\times \:(\text{integral})\:\mathrm{d}y_1\mathrm{d}y_2.
\end{align*}\\

At this stage we differentiate between two cases. The diagonal contribution, i.e. when $m\equiv v\bmod{p}$, and the off-diagonal $p\nmid v-m$. For the diagonal contribution, the integral is bounded trivially taking into account the size of the Bessel functions and the fact that the average size of the $h$ function is $O(1)$. We get
 \begin{align*}
(\text{integral})\ll \;\min\left\{1,\frac{CQp^{1/2}J^2\ell}{N_\star M^{1/2}}\right\}
\end{align*}
The diagonal contribution is then bounded by
\begin{align*}
&p^\varepsilon\frac{p^{5+\theta}Q^3}{N^5\ell^3}\sum_{|r|\ll \mathcal{N}}\:\sum_{q\ll
\mathcal{N}p^\theta}\;\frac{(r^2,q)^{1/2}}{q^2}\; \left(1+\frac{N_\star qp^2}{CQJ^2N}\right)\;\min\left\{1,\frac{CQp^{1/2}J^2\ell}{N_\star M^{1/2}}\right\}.
\end{align*}
Summing over $r$ using the inequality $(r^2,q)^{1/2}\ll (r,q)$, we arrive at
\begin{align*}
&p^\varepsilon\frac{p^{6+\theta}Q^4}{N^6\ell^4}\:\sum_{q\ll
\mathcal{N}p^\theta}\;\frac{1}{q^2}\; \left(1+\frac{N_\star qp^2}{CQJ^2N}\right)\;\min\left\{1,\frac{CQp^{1/2}J^2\ell}{N_\star M^{1/2}}\right\},
\end{align*}
which is then bounded by
\begin{align}
\label{p>Q}
p^\varepsilon\frac{p^{6+\theta}Q^4}{N^6\ell^4}+p^\varepsilon\frac{p^{6+\theta}Q^4}{N^6\ell^4}\frac{(N_\star\ell)^{1/2}p^{3/2}}{(CJ^2N)^{1/2}Q}&\sum_{q\ll
\mathcal{N}p^\theta}\;\frac{1}{q}\ll \frac{p^{1+8\theta+\varepsilon}Q^{3}}{\ell^{3}}\:\frac{N_\star^{1/2}}{(CJ^2\ell)^{1/2}}.
\end{align}
Here in the last inequality we are using the assumption that $Q<p$.\\
 
 Next we turn to the off-diagonal $p\nmid m-v$. In this case the above analysis yields a bound which unsatisfactory when $C$ is small. To obtain a better estimate we will try to get some extra cancellation. We apply the Poisson summation formula on the sum over $r$. This produces the expression
\begin{align*}
& \frac{p^5Q^4}{N^6\ell^4}\:\sum_{q\ll
\mathcal{N}p^\theta}\;\frac{1}{q^4}\\
&\times \iint \mathop{\sum\sum}_{\substack{n\in \Psi_{y_1}\\r\in \Psi_{y_2}}}\:\mathop{\sum\sum}_{\substack{m\in \Phi_{y_1}\\v\in \Phi_{y_2}\\p\nmid m-v}}\: (\text{char sum})\:\times \:(\text{integral})\:\mathrm{d}y_1\mathrm{d}y_2.
\end{align*}
Here the character sum is given by $q$ times the Ramanujan sum 
\begin{align*}
\mathfrak{c}_q(4(v-m)+p(n^2-r^2))=\sideset{}{^\star}\sum_{a\bmod{q}}\:e\left(\frac{a(4(v-m)+p(n^2-r^2))}{q}\right).
\end{align*}
Observe that since $p\nmid v-m$ and $p$ is a prime the character sum modulo $p$ is given by $-1$. Consequently we get that the above expression is bounded by
\begin{align*}
& p^\varepsilon\frac{p^5Q^4}{N^6\ell^4}\:\sum_{q\ll
\mathcal{N}p^\theta}\;\frac{1}{q^3}\: \iint \\
&\times\mathop{\sum\sum}_{\substack{n\in \Psi_{y_1}\\r\in \Psi_{y_2}}}\:\mathop{\sum\sum}_{\substack{m\in \Phi_{y_1}\\v\in \Phi_{y_2}\\p\nmid m-v}}\: (4(v-m)+p(n^2-r^2),q)\:\min\left\{1,\frac{CQp^{1/2}J^2\ell}{N_\star M^{1/2}}\right\}\:\mathrm{d}y_1\mathrm{d}y_2.
\end{align*}
Note that here we have plugged in the size of the two Bessel functions, and have used the fact that the $h$ function on average is of size $O(1)$. The sum over $v$ is dominated by
\begin{align*}
\mathop{\sum}_{\substack{v\in \Phi_{y_2}}}\: (4(v-m)+p(n^2-r^2),q)\ll \sum_{d|q}\:d\:\mathop{\sum}_{\substack{v\in \Phi_{y_2}\\v\equiv A\bmod{d}}} 1
\end{align*}
where $A$ is a congruence class modulo $q$ determined by $(m,n,r)$. The above sum is thus dominated by
\begin{align*}
\sum_{d|q}\:d\:\left(1+\frac{N_\star qp^2}{dCQJ^2N}\right)\ll p^\varepsilon\left(q+\frac{N_\star qp^2}{CQJ^2N}\right).
\end{align*}
With this we see that the total contribution of the off-diagonal is dominated by
\begin{align*}
p^\varepsilon\frac{p^{5+2\theta}Q^4}{N^6\ell^4}\sum_{q\ll
\mathcal{N}p^\theta}\;\frac{1}{q^2}\:\left(1+\frac{N_\star p^2}{CQJ^2N}\right)\left(1+\frac{N_\star qp^2}{CQJ^2N}\right)\:\min\left\{1,\frac{CQp^{1/2}J^2\ell}{N_\star M^{1/2}}\right\}.
\end{align*}
Now
\begin{align*}
\left(1+\frac{N_\star p^2}{CQJ^2N}\right)\:\min\left\{1,\frac{CQp^{1/2}J^2\ell}{N_\star M^{1/2}}\right\}\ll \frac{p\ell}{Q},
\end{align*} 
where again we used the assumption that $Q<p$. Plugging in this bound we arrive at
\begin{align*}
p^\varepsilon\frac{p^{6+2\theta}Q^3}{N^6\ell^3}\sum_{q\ll
\mathcal{N}p^\theta}\;\frac{1}{q^2}\:\left(1+\frac{N_\star qp^2}{CQJ^2N}\right)\ll \frac{p^{1+10\theta+\varepsilon}Q^3}{\ell^3}\:\frac{N_\star}{CJ^2Q}.
\end{align*}
This dominates the error term in \eqref{the-leading-exp} as $Q<p$.
This completes the proof of the proposition.
\end{proof}


\end{document}